\documentclass[11pt, a4paper, UKenglish]{article}
\usepackage[UKenglish]{babel}
\usepackage[utf8]{inputenc} 
\usepackage{eucal} 
\usepackage[toc, page]{appendix}
\usepackage[dvipsnames]{xcolor} 

\usepackage{hyperref} 
\hypersetup{
    colorlinks=true, 
    linktoc=page, 
    linkcolor=Mahogany, 
    citecolor=OliveGreen, 
    urlcolor=Blue} 
 
\usepackage{amsmath,amsthm,amssymb,amsfonts,mathtools} 
\usepackage{enumitem}
\usepackage{dsfont}

\usepackage{textcomp} 
\usepackage{setspace} 
\usepackage{romannum}

\usepackage[nottoc,notlot,notlof]{tocbibind} 
\usepackage{titlesec}
\titlespacing{\section}{0pt}{12pt plus 2pt minus 1 pt}{-2pt}
\titlespacing{\subsection}{0pt}{10pt}{-2pt}
\titlespacing{\subsubsection}{0pt}{10pt}{-2pt}

\usepackage{graphicx} 
\usepackage{float}
\usepackage{wrapfig}
\usepackage[hang, small]{caption} 
\usepackage{subcaption}
\usepackage{comment}

\setlist[itemize]{itemsep = 0pt, topsep=-5pt}

\allowdisplaybreaks

\makeatletter
\def\@endtheorem{\endtrivlist}
\makeatother

\newtheorem{theorem}{Theorem}[section]
\newtheorem{lemma}[theorem]{Lemma}
\newtheorem{proposition}[theorem]{Proposition}
\newtheorem{corollary}[theorem]{Corollary}
\theoremstyle{definition}
\newtheorem{definition}[theorem]{Definition}
\numberwithin{equation}{section}

\newtheorem*{assumption}{Assumption}

\newcommand{\hypgeom}{{}_2F_1}
\DeclareMathOperator{\arccot}{arccot}
\DeclareMathOperator{\expo}{\mathrm{e}}
\newcommand{\iu}{i\mkern1mu}
\newcommand{\norm}[1]{\left\lVert#1\right\rVert}

\author{Aleksandra Korzhenkova}
\date{\vspace{-5ex}}
\title{The exploration process of~critical Boltzmann planar~maps decorated by~a~triangular~$O(n)$~loop~model}

\begin{document}
\maketitle
\pagenumbering{arabic}
\begin{abstract}
\vspace{-0.9ex}In this paper we investigate pointed $(\mathbf{q}, g, n)$-Boltzmann loop-decorated maps with loops traversing only inner triangular faces. Using peeling exploration \cite{TB18} modified to this setting we show that its law in the non-generic critical phase can be coded in terms of a random walk confined to the positive integers by a new specific boundary condition. Under a technical assumption that we believe to be true, combining this observation with explicit quantities for the peeling law we derive the large deviations property for the distribution of the so-called nesting statistic and show that the exploration process possesses exactly the same scaling limit as in the rigid loop model on bipartite maps that is a specific self-similar Markov process introduced in \cite{TB18}. Besides, we conclude the equivalence of the admissible weight sequences related by the so-called fixed point equation by proving the missing direction in the argument of \cite{BBG12}.
\end{abstract}

\section{Introduction}
\label{sec:intro}
In this paper we study random planar maps coupled to an $O(n)$ loop model. In the physics literature, maps decorated with a particular class of $O(n)$ loop models where loops visit only vertices of degree~three were analyzed already in the nineties via matrix integral techniques \cite{DK_o(n), Kostov_o(n), KS_o(n), EZJ_o(n), EK_o(n), EYNARD_o(N)} (see \cite[Section 1.2]{BBG11} for a short summary of these results). About ten years ago a new approach to study loop-decorated maps, the so-called gasket decomposition, was proposed in \cite{BBG11} (see also Section 8 in \cite{LGM11}). This method is based on decomposing the map using the nesting structure of the decorating loops and allows to reduce the study of loop-decorated maps to investigating undecorated maps with faces of large degrees, studied in \cite{LGM11}. It brought along a very fruitful wave of research on the geometry of random planar maps coupled to an $O(n)$ loop model and their gaskets \cite{CCM17, TB18, BC17, BCM18}.
 
In most of the aforementioned rigorous works the authors only deal with bipartite maps decorated by the rigid $O(n)$ loop model. The goal of this work is to extend the results of \cite{TB18} to the case of non-bipartite Boltzmann random planar maps coupled to the \emph{triangular} $O(n)$ loop model, where the loops intersect only triangular faces (see Figure \ref{fig:def_exmpl}). To be precise, we adjust the peeling process of \cite{TB18} to our triangular setting and introduce a new process, which we call \emph{accelerated $\mathfrak{p}$-ricocheted random walk (ARRW)}, to study the law of this peeling exploration. The additional ``acceleration" term stems from the fact that, contrary to the case of the rigid $O(n)$ loop configurations, in the triangular $O(n)$ loop configurations the lengths of the inner and outer boundaries of the concatenation of faces covered by a loop are not necessarily equal. As explained below, this extra term will, however, disappear in the scaling limit.

We now briefly discuss the main results of this article which we obtain using the aforementioned adjusted peeling procedure and the ARRW. The precise statements and a discussion of the related literature and strategies of the proofs, as well as all necessary definitions can be found in Section \ref{sec:preliminaries}. 

The first statement, Theorem \ref{thm:admissibility}, is about the equivalence of admissibility criteria for the weight sequences of our loop-decorated maps and certain undecorated maps, when the two sequences are related by the fixed-point equation (\ref{fpe}). Recall that admissibility here refers to the finiteness of the partition function. The proof of this theorem is carried out using an auxiliary result, Lemma \ref{lemma:admis_bullet}, which ensures that under the assumption that a weight sequence $\mathbf{\hat{q}}$ is admissible for the set of planar undecorated maps, the corresponding partition function for the set of planar undecorated maps with a distinguished vertex is finite. This auxiliary result might be of independent interest. 

In Theorem \ref{thm:nesting} we derive the distribution of the nesting statistic, i.e. the number of loops separating the marked vertex from the root face, in terms of functions $h_\mathfrak{p}^r$ related to the ARRW (see Definition \ref{def:h^r_fct}) and find its large deviations. The main ingredient to conclude the latter result is the asymptotics of $h_\mathfrak{p}^r$ derived in Proposition \ref{prop:h-props} $(\romannumeral 3)$. 

The last and most substantial result of this article is stated in Theorem \ref{thm:scallim} and says that the perimeter process in the non-generic critical phase rescaled exactly in the same way as for the rigid loop model converges in distribution in the $J_1$-Skorokhod topology towards a positive self-similar Markov process introduced in \cite[Section 6]{TB18}. In particular, this confirms a certain universality of the limit, and leads to the conjecture that the perimeter process of the model with bending energy \cite{BBG12}, which covers both triangular and rigid loop models, possesses the same scaling limit.

\subsection{Outline}
Section \ref{sec:preliminaries} contains all necessary definitions, a description of the precise setting we work with in this article and the precise statements of the aforementioned results. In Section \ref{sec:peeling} we recap the peeling procedure of \cite{TB18} and directly adjust it to rooted planar maps decorated by triangular loop-configurations. In Section \ref{sec:Boltzmann} we discuss the law of this peeling exploration when applied to a pointed $(\mathbf{q}, g, n)$-Boltzmann loop-decorated map $(\mathfrak{m}_\bullet, \boldsymbol\ell)$ under the assumption that $(\mathbf{q}, g, n)$ is admissible and the expected number of vertices of $(\mathfrak{m}_\bullet, \boldsymbol\ell)$ is finite. Section \ref{sec:ricochetedRW} introduces a Markov process on $\mathbb{Z}^2$, which we call accelerated ricocheted random walk (ARRW), whose transition probabilities mimic the law of the peeling exploration performed on a pointed Boltzmann loop-decorated map. We also discuss the proof of Theorem \ref{thm:nesting} by highlighting the changes which could be done to make the proof from \cite{TB18} work in our case. Section \ref{sec:admissibilityPf} contains the proof of the part of Theorem \ref{thm:admissibility} concerning the admissibility of $(\mathbf{q}, g, n)$ via algorithmic reconstruction of a $(\mathbf{q}, g, n)$-Boltzmann planar map using peeling exploration. As a by-product we obtain asymptotics of the expected volume of a pointed $(\mathbf{q}, g, n)$-Boltzmann loop-decorated map. In Section \ref{sec:scalLim} we prove our main result, Theorem \ref{thm:scallim}, by showing the corresponding result first for a class of ARRWs and then relating it to the perimeter process in the peeling exploration. In Appendix \ref{sec:Appendix} we collect technical computations omitted in the proofs throughout the article.

\subsection*{Acknowledgements}
The core of this article was accomplished in our master's thesis \cite{thesis} under supervision of Eveliina Peltola. We thank Eveliina Peltola and Juhan Aru for insightful discussions and useful comments on our work. We thank an anonymous referee for the valuable suggestions on the previous versions of this paper.
This work was supported by Eccellenza grant 194648 of the Swiss National Science Foundation.

\section{Preliminaries and precise statements of the main results}
\label{sec:preliminaries}
\subsection{The triangular \texorpdfstring{$O(n)$}{O(n)} loop model and its phase diagram}
\label{subsubsec:intro_loop_model}
The combinatorial objects investigated in this work are planar maps with a distinguished oriented edge, called \emph{root edge}, decorated by loop configurations. A \emph{planar map} is a connected graph with all vertices of finite degree in which loops and multiple edges are allowed and that is properly embedded in the sphere. The face to the right of the root edge is the \emph{root face}, all the other faces are called \emph{inner} faces. The \emph{perimeter} of the map is the degree of its root face. A \emph{loop configuration} on the rooted planar map $\mathfrak{m}$ is a collection $\boldsymbol\ell = \{\ell_1, \ldots, \ell_k\}$ of disjoint unoriented loops on the dual map $\mathfrak{m}^\dagger$ non-intersecting the root face. The loop configurations investigated in this work, which we call \emph{triangular}, consist of loops visiting only inner triangular faces of $\mathfrak{m}$, i.e. visiting only vertices of $\mathfrak{m}^\dagger$ of degree three which correspond to inner faces of $\mathfrak{m}$ (see Figure \ref{fig:def_exmpl}). A pair $(\mathfrak{m}, \boldsymbol\ell)$ consisting of a rooted planar map $\mathfrak{m}$ and a triangular loop configuration $\boldsymbol\ell$ on $\mathfrak{m}$ is called a \emph{loop-decorated map}. The set of all such loop-decorated maps with a fixed perimeter $p$ modulo homeomorphisms of the sphere is denoted by $\mathcal{LM}^{(p)}$. 
\begin{figure}[ht]
    \centering
    \includegraphics[width = 0.4\linewidth]{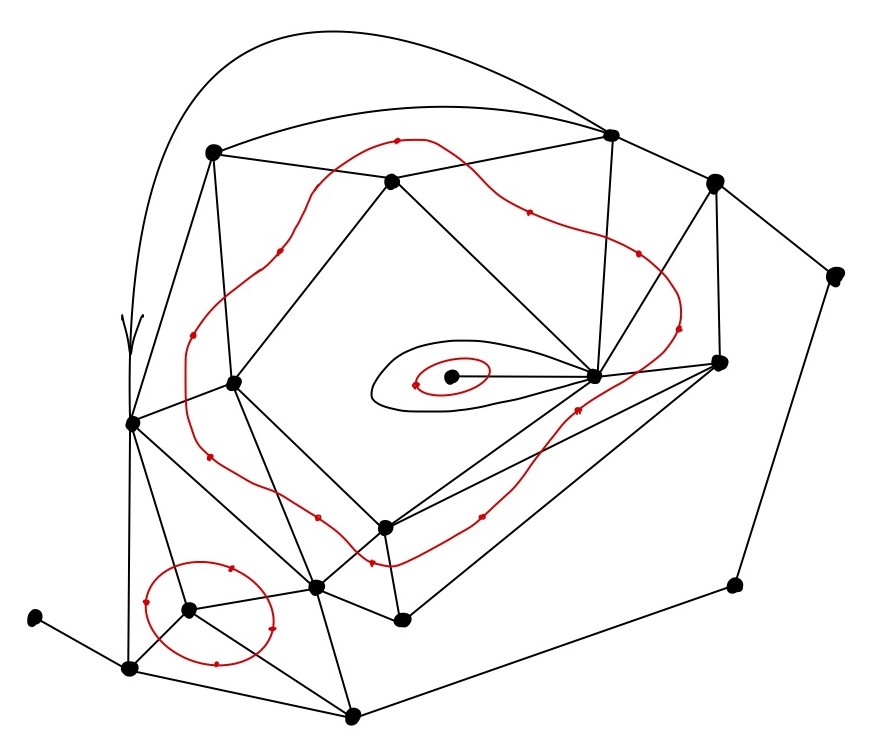}
    \caption{A loop-decorated planar map $\mathfrak{m}$ of perimeter $9$ with a triangular loop configuration $\boldsymbol\ell = \{\ell_1, \ell_2, \ell_3\}$. The weight of this map is $\mathbf{w}_{\mathbf{q}, g, n}(\mathfrak{m}, \boldsymbol\ell) = n^3g^{16}q_3q_5q_7$.}
    \label{fig:def_exmpl}
\end{figure}

Given a sequence of non-negative real numbers $\mathbf{q} = (q_1, q_2, \ldots)$ and $g, n \geq 0$, the \emph{weight} $\mathbf{w}_{\mathbf{q}, g, n}(\mathfrak{m}, \boldsymbol\ell)$ of a loop-decorated map is defined as
\begin{align}
    \mathbf{w}_{\mathbf{q}, g, n}(\mathfrak{m}, \boldsymbol\ell) \coloneqq 
    \left[\prod_{\ell \in \boldsymbol\ell} ng^{|\ell|}\right] \left[\prod_{f} q_{\mathrm{deg}(f)}\right], \label{weight}
\end{align}
where $|\ell|$ is the length of the loop $\ell$, the second product ranges over all inner faces $f$ of $\mathfrak{m}$ that are not visited by a loop and $\mathrm{deg}(f)$ is the degree of the face $f$. The $O(n)$ \emph{model partition function} $F^{(p)}(\mathbf{q}, g, n)$ is given by the total weight of all loop-decorated maps of fixed perimeter $p$:
\begin{align}
    F^{(p)}(\mathbf{q}, g, n) \coloneqq 
    \sum_{(\mathfrak{m}, \boldsymbol\ell) \in \mathcal{LM}^{(p)}} \mathbf{w}_{\mathbf{q}, g, n}(\mathfrak{m},  \boldsymbol\ell). \label{part_fct}
\end{align}
The triplet $(\mathbf{q}, g, n)$ is said to be \emph{admissible} iff $F^{(p)}(\mathbf{q}, g, n) < \infty$ for all $p \geq 1$. In this case it is possible to define the probability measure on $\mathcal{LM}^{(p)}$ induced by $\mathbf{w}_{\mathbf{q}, g, n}$ via normalization by $1/F^{(p)}(\mathbf{q}, g, n)$. A rooted planar map of perimeter $p$ together with this probability measure is called $(\mathbf{q}, g, n)$-\emph{Boltzmann loop-decorated map of perimeter} $p$. If $g = 0$ or $n = 0$, then loops are suppressed and the random map is the $\mathbf{q}$-\emph{Boltzmann map}. In this case the partition function is denoted by $W^{(p)}(\mathbf{q}) \coloneqq F^{(p)}(\mathbf{q}, 0, 0)$, and the weight sequence $\mathbf{q}$ is called \emph{admissible} iff $W^{(p)}(\mathbf{q}) < \infty$ for all $p \geq 1$. We set $W^{(0)} = F^{(0)} \coloneqq 1$. The usual $\mathbf{q}$-Boltzmann maps were studied in \cite{TB16} and extensively discussed in the bipartite setup in \cite{Curien_notes}. 

The central method used in the literature to study Boltzmann loop-decorated maps is the \emph{gasket decomposition}. It was first introduced in \cite{BBG11} where it was applied to the \emph{rigid} $O(n)$ loop model on bipartite maps to yield the results about its phase diagram and resolvent. In the rigid $O(n)$ model loops intersect only inner quadrangular faces of the map through their opposite edges. Later in \cite{BBG12} the outcomes were extended to the $O(n)$ loop model with bending energy, which covers both rigid and triangular cases of loop configurations. The method consists in splitting the loop-decorated map $\mathfrak{m}$ into its \emph{gasket}, outer loop rings, i.e. faces visited by outermost loops, and maps contained inside the outer loop rings. The gasket $\mathfrak{g}(\mathfrak{m}, \boldsymbol\ell)$ is defined as the undecorated map obtained from $\mathfrak{m}$ by removing all edges traversed by any loop in $\boldsymbol\ell$ and keeping only the connected component containing the root. Via the gasket decomposition it was shown in \cite{BBG11} that if $(\mathbf{q}, g, n)$ is admissible then so is $\hat{\mathbf{q}}$ given by the \emph{fixed point equation}
\begin{align}
    \hat{q}_k = q_k + n\sum_{k' \geq 0} \binom{k + k' - 1}{k'} g^{k + k'} F^{(k')}(\mathbf{q}, g, n), \label{fpe}
\end{align}
and that the gasket $\mathfrak{g}(\mathfrak{m}, \boldsymbol\ell)$ of a $(\mathbf{q}, g, n)$-Boltzmann loop-decorated map is distributed as a $\hat{\mathbf{q}}$-Boltzmann map. The converse is also true, but was left implicit in both \cite{BBG11} and \cite{BBG12}. Here we prove this statement by relying on the corresponding proof for the rigid $O(n)$ loop model in \cite{TB18}.

\begin{theorem}[Equivalence of admissibility]
\label{thm:admissibility}
    For $n \in [0, 2], g \geq 0$ and a sequence of non-negative real numbers $\mathbf{q} = (q_1, q_2, \ldots)$, the triplet $(\mathbf{q}, g, n)$ is admissible iff there exists an admissible weight sequence $\hat{\mathbf{q}}$ s.t.
    \begin{align*}
        q_k = \hat{q}_k - n\sum_{k' \geq 0} \binom{k + k' - 1}{k'} g^{k + k'} W^{(k')}(\hat{\mathbf{q}}) \geq 0.
    \end{align*}
    In this case $F^{(p)}(\mathbf{q}, g, n) = W^{(p)}(\hat{\mathbf{q}})$ and the expected number of vertices in a $(\mathbf{q}, g, n)$-Boltzmann loop-decorated map is finite.
\end{theorem}

Given any \emph{$\bullet$-admissible} $\mathbf{\hat{q}}$, meaning that the partition function of pointed\footnote{with a distinguished vertex} $\mathbf{\hat{q}}$-Boltzmann maps of perimeter $p$, $W^{(p)}_\bullet(\hat{\mathbf{q}})$, is finite for all $p\in \mathbb{N}$, it was shown in \cite[Section 6.1]{BBG12} that the \emph{resolvent}, defined as
\begin{align}
\label{resolvent}
    \mathcal{W}(x) = \sum_{p \geq 0} \frac{W^{(p)}(\mathbf{\hat{q}})}{x^{p + 1}},
\end{align}
can be analytically continued into a holomorphic function on $\mathbb{C}\setminus [\gamma_-, \gamma_+]$ with finite discontinuity on $[\gamma_-, \gamma_+] \subset \mathbb{R}$. Moreover, $\gamma_\pm = \gamma_\pm(\mathbf{\hat{q}})$ satisfy $\gamma_+ \geq |\gamma_-|$, and $\gamma_+ = |\gamma_-| = -\gamma_-$ hold exactly in the bipartite case. We introduce a non-universal constant $r \coloneqq -\frac{\gamma_-}{\gamma_+} \in (-1, 1)$. We exclude the bipartite case for simplicity since the results have to be slightly modified each time and this case was covered in \cite{TB18} for the rigid $O(n)$ loop model and in \cite{TB16} for maps without loops. In addition to that, in \cite{TB16} it was established that for any fixed $k \geq 0$,
\begin{align}
    \gamma_+^k(\mathbf{\hat{q}}) = \lim_{p \rightarrow \infty} 
    \frac{W_\bullet^{(p + k)}(\mathbf{\hat{q}})}{W_\bullet^{(p)}(\mathbf{\hat{q}})}.
    \label{gamma+}
\end{align}
We extend both results by showing that it suffices to assume the admissibility of $\mathbf{\hat{q}}$ in the sense $W^{(p)}(\mathbf{\hat{q}}) < \infty$ for all $p\in \mathbb{N}$ instead, see \ref{A:subsec:admis_no_loop}.
Moreover, the following characterization of $\gamma_+$ for an admissible $\mathbf{\hat{q}}$ holds:
\begin{align}
\label{gamma+_notarget}
    \gamma_+(\mathbf{\hat{q}}) = \lim_{p \rightarrow \infty} 
    \frac{W^{(p + 1)} (\mathbf{\hat{q}})}{W^{(p)} (\mathbf{\hat{q}})}.
\end{align}
We refer the reader to the end of Section \ref{A:subsec:ineq_implication} for the proof of (\ref{gamma+_notarget}).  

We call a triplet $(\mathbf{q}, g, n)$ \emph{non-generic critical} if it is admissible, $n \neq 0$, and $g\gamma_+(\hat{\mathbf{q}}) = \frac{1}{2}$, where $\hat{\mathbf{q}}$ is defined as in (\ref{fpe}) and $\gamma_+(\mathbf{\hat{q}})$ as in (\ref{gamma+_notarget}) with $F^{(p)}(\mathbf{q}, g, n) = W^{(p)}(\hat{\mathbf{q}})$, cf. Theorem \ref{thm:admissibility}. 
In \cite{BBG12} the alternative characterization of such triplets was provided in terms of the four phases of the $O(n)$ loop model with bending energy. For these phases to be well-defined, we need to restrict ourselves to triplets $(\mathbf{q}, g, n)$ in the domain
\begin{equation*}
    \mathcal{D} \coloneqq [0, \infty)^d \times (0, \infty) \times (0, 2),
\end{equation*}
where $d$ is a fixed positive integer and $\mathbf{q} \in [0, \infty )^d$ means that $q_k = 0$ for any $k > d$. According to \cite[(3.31), (3.32)]{BBG12}, if $(\mathbf{q}, g, n) \in \mathcal{D}$ is admissible, then 
\begin{align}
    F^{(p)}(\mathbf{q}, g, n) \sim C \gamma_+^p(\mathbf{\hat{q}}) p^{-a} \quad \text{as} \;\; p \rightarrow \infty,
    \label{asymptotic}
\end{align}
for some $C > 0$ depending on $(\mathbf{q}, g, n)$ but not on $p$, and $a$ taking only the values $\frac{3}{2}, \frac{5}{2}, 2 + b \; \text{and} \; 2 - b$, where $b = \frac{1}{\pi} \arccos\left(\frac{n}{2}\right) \in \left(0, \frac{1}{2}\right)$. The triplet $(\mathbf{q}, g, n)$ is then called \emph{subcritical, generic critical, non-generic critical dilute} and \emph{non-generic critical dense}, respectively, in these four cases. In the last two cases the triplet is also non-generic critical in the aforementioned sense.

The techniques in \cite{BBG12} and \cite{BBD18} are, however, not sufficient to rigorously establish the phase diagram, in particular (\ref{asymptotic}). This issue was resolved in the Appendix of \cite{TB18} for the rigid $O(n)$ loop model. Nevertheless, the numerical derivation of the phase diagram in the physics papers \cite{BBG12, BBD18} makes the following assumption plausible:
\begin{assumption}
    Throughout this paper we assume that the non-generic critical dense and dilute phases exist and are non-empty and that given $(\mathbf{q}, g, n) \in \mathcal{D}$ non-generic critical (dilute or dense), (\ref{asymptotic}) holds for $a = 2 \pm b$ with $\gamma_+(\mathbf{\hat{q}})$ as in (\ref{gamma+_notarget}) and $\hat{\mathbf{q}}$ as in (\ref{fpe}).
\end{assumption} \vspace{-0.3cm}

We now discuss in more detail the rigorously established results of \cite{BBG12, BBD18} and explain the missing step to rigorously verify the assumption. \\
A set of conditions satisfied by the \emph{resolvent of the triangular $O(n)$ loop model} $\mathcal{F}(x) \coloneqq \sum_{p \geq 0} \frac{F^{(p)}}{x^{p + 1}}$ was derived in Sections 2.3 and 6 of \cite{BBG12} and Propositions 4.1, 4.3 and 4.5 of \cite{BBD18}, respectively. These conditions are composed of two parts: 
\begin{itemize}
    \item $x \mapsto \mathcal{F}(x)$ is holomorphic on a domain of the form $\mathbb{C}\setminus[\gamma_-,\gamma_+]$ with $[\gamma_-,\gamma_+] \subset \mathbb{R}$, and satisfies some growth constraints and a functional linear equation on the segment $[\gamma_-,\gamma_+]$.
    \item Characterization of $\gamma_-, \gamma_+ \in \mathbb{R}$ for a given $\mathbf{\hat{q}}$ defined by (\ref{fpe}) for some $(\mathbf{q}, g, n)$. \\
    Since the weights of $\mathbf{\hat{q}}$ depend in turn on $\mathcal{F}(x)$, as opposed to the case without loops, this condition can not alone determine $\gamma_-, \gamma_+$, but should be considered as a system with the first item.
\end{itemize}
It was shown that for a given interval $[\gamma_-,\gamma_+] \subset \mathbb{R}$ such that $\gamma_+ \geq |\gamma_-|$, the solution to the first bullet-point is unique, and an explicit function in parametric form (elliptic parametrization), $\tilde{\mathcal{F}}(x)$, satisfying the first condition was constructed, see \cite[Section 5]{BBD18} and \cite[Section 3.3]{BBG12}. From the latter, (\ref{asymptotic}) was concluded with the given $\gamma_+$. Furthermore, it was proven that for a given sequence of weights $\mathbf{\hat{q}}$, the second condition uniquely determines $\gamma_-, \gamma_+$. This, in particular, is the case, when we assume to know $F^{(p)}(\mathbf{q}, g, n) = W^{(p)}(\hat{\mathbf{q}})$ for all $p\in \mathbb{N}$, then also (\ref{gamma+_notarget}) holds. However, to guarantee that the function $\tilde{\mathcal{F}}(x)$ constructed as the solution to the first condition with $\gamma_\pm$ chosen to correspond to the resolvent of the triangular loop model $\mathcal{F}(x)$ is indeed $\mathcal{F}(x)$, one has to show that the system of the two aforementioned conditions has a unique solution. This step is missing in both \cite{BBG12} and \cite{BBD18} and would have verified the assumption. 

We also mention that the proof of our assumption for bipartite maps decorated by the rigid $O(n)$ loop configuration by Linxiao Chen in \cite{TB18} relies on the symmetry $\gamma_- = - \gamma_+$ for bipartite maps, which reduces the number of parameters in the system, and can not be directly extended to our setting.

\subsection{Nesting statistic}
Let $\mathcal{LM}_\bullet^{(p)}$ be the set of loop-decorated maps with a distinguished vertex, called \emph{target}. If $(\mathbf{q}, g, n)$ is admissible, it is possible to define the \emph{pointed} $(\mathbf{q}, g, n)$-Boltzmann loop-decorated map as $(\mathfrak{m}_\bullet, \boldsymbol\ell) \in \mathcal{LM}_\bullet^{(p)}$ with probability distribution proportional to $\mathbf{w}_{\mathbf{q}, g, n}(\mathfrak{m}_\bullet, \boldsymbol\ell) \coloneqq \mathbf{w}_{\mathbf{q}, g, n}(\mathfrak{m}, \boldsymbol\ell)$. This is well-defined since the expected number of vertices of the $(\mathbf{q}, g, n)$-Boltzmann loop-decorated map is finite by Theorem \ref{thm:admissibility}. On the pointed loop-decorated maps a new variable, the \emph{nesting statistic} $N$, can be introduced. More precisely, $N$ is the number of loops that separate the target from the root face. Its distribution in the case of the $O(n)$ loop model with bending energy on triangulations of large size was studied in detail in \cite{BBD18}. In the case of rigid $O(n)$ loop model on large bipartite maps in the non-generic critical setting an alternative more intuitive and direct proof of a weaker result via peeling exploration was given in \cite{TB18}. In this work both these approaches are combined to derive Theorem \ref{thm:nesting}. Before stating the result we must present the functions $h^r_\mathfrak{p}: \mathbb{Z} \rightarrow \mathbb{R}$ for $\mathfrak{p} \in [0, 1]$ by setting $h^r_\mathfrak{p}(0) = 1, \; h^r_\mathfrak{p}(p) = 0$ for $p < 0$, and for $p \geq 1$
\begin{equation}
    \label{h-fct}
    \begin{aligned}
        h^r_0(p) &= \left( \frac{-r}{4} \right)^p \binom{2p}{p} 
        \hypgeom\left( \frac{1}{2}, -p; \; \frac{1}{2} - p; -\frac{1}{r}\right), \\
        h^r_1(p) &= 1,
    \end{aligned}
\end{equation}
where $\hypgeom$ is the hypergeometric function, i.e. $\hypgeom(a, b; c; z) = \sum_{n \geq 0} \frac{(a)_n (b)_n}{(c)_n} \frac{z^n}{n!}$ in terms of the rising Pochhammer symbol $(a)_n \coloneqq a(a + 1) \ldots (a + n - 1)$. For other values of $\mathfrak{p} \in (0, 1)$ the functions $h^r_\mathfrak{p}(p)$ are introduced in Section \ref{subsec:ARRW}, Definition \ref{def:h^r_fct}, and are defined as the hitting probabilities of a certain random walk.

\begin{theorem}[Properties of the nesting statistic]
\label{thm:nesting}
Suppose that {\bfseries Assumption} holds. Let $n \in (0, 2), \; (\mathbf{q}, g, n)$ be non-generic critical, and $(\mathfrak{m}_\bullet, \boldsymbol\ell)$ be a pointed $(\mathbf{q}, g, n)$-Boltzmann loop-decorated map of perimeter $p$. Then the number $N$ of loops surrounding the target has probability generating function 
\begin{align*}
    \mathbb{E}_\bullet^{(p)} \left[ x^N \right] = \frac{h^r_{xn/2}(p)}{h^r_{n/2}(p)} 
    \quad \text{for} \;\; x \in \left[ 0, \frac{2}{n} \right).
\end{align*}
The following convergence in probability takes place
\begin{align*}
    \frac{N}{\log p} \xrightarrow{\mathbb{P}} \frac{n}{\pi\sqrt{4 - n^2}} \quad \text{as} \;\; p \rightarrow\infty,
\end{align*}
and the large deviation property 
\begin{equation*}
    \begin{aligned}
        \frac{\log \mathbb{P}_\bullet^{(p)}\left[N < \lambda \log p\right]}{\log p} &\xrightarrow{p \rightarrow \infty} -\frac{1}{\pi} J_{n/2}(\pi\lambda) \quad &&\text{for} \;\; 0 < \lambda < \frac{n}{\pi\sqrt{4 - n^2}},\\
        \frac{\log \mathbb{P}_\bullet^{(p)}\left[N > \lambda \log p\right]}{\log p} &\xrightarrow{p \rightarrow \infty} -\frac{1}{\pi} J_{n/2}(\pi\lambda) \quad &&\text{for} \;\; \lambda > \frac{n}{\pi\sqrt{4 - n^2}},
    \end{aligned}
\end{equation*}
holds with
\begin{align*}
    J_\mathfrak{p}(x) = x\log\left(\frac{x}{\mathfrak{p}\sqrt{1 + x^2}}\right) + \arccot x - \arccos\mathfrak{p}.
\end{align*}
\end{theorem}
We recall that in order for the {\bfseries Assumption} to hold we have to restrict ourselves to the triplets $(\mathbf{q}, g, n) \in \mathcal{D}$ that are associated with the maps of bounded face degrees.

\subsection{Geometry of large loop-decorated maps}
In Chapter \ref{sec:peeling} we introduce an exploration process on pointed loop-decorated maps $(\mathfrak{m}_\bullet, \boldsymbol\ell) \in \mathcal{LM}_\bullet^{(p)}$ and study their properties. This specific exploration procedure, which we call \emph{targeted peeling process}, is an adaptation of the procedure used in \cite{TB18} on planar maps with rigid loop configurations to our triangular setup. The exploration process of undecorated maps presented in this work is exactly the one from \cite{TB18} that in turn is a modification of the \emph{lazy} peeling process from the earlier work \cite{TB16}.

The (targeted) peeling process iteratively constructs an increasing sequence of submaps of $(\mathfrak{m}_\bullet, \boldsymbol\ell)$ corresponding to the explored region including the root face at each stage of the exploration, and terminates once the marked vertex is encountered (see Figure \ref{fig:loop_peeling}). Each exploration step consists of choosing an edge on the boundary of the explored part according to an arbitrary but fixed \emph{peeling algorithm}, and discovering what is incident to this edge in the unexplored region. The peeling algorithm used for exploration can be either deterministic or random, but in the latter case the randomness involved has to be independent of the unexplored region of the map. If a new face or loop is discovered, it is added to the explored region together with the still unexplored connected components without target. When applied to a pointed Boltzmann loop decorated map the law of this peeling exploration can be described (see Proposition \ref{prop:h-trafo}) in terms of the \emph{perimeter process} $(P_i)$, which is a Markov process on $\mathbb{Z}$ tracking the length of the boundary of the explored region.

For $b \in \left(0, \frac{1}{2}\right)$ fixed as above, following \cite{TB18} we introduce the L\'{e}vy process $(\xi^\downarrow_t)_{t \geq 0}$ started at $\xi^\downarrow_0 = 0$ with Laplace exponent 
\begin{align*}
    \Psi^\downarrow(z) \coloneqq \log\mathbb{E}\left[ \mathrm{e}^{z\xi^\downarrow_1} \right] = \frac{1}{\pi} \Gamma(1 + 2b - z) \Gamma(1 - b + z) \left[ \cos{(\pi b)} - \cos{(\pi z - \pi b)} \right].
\end{align*}
It has no killing and drifts to $-\infty$ almost surely. By \cite{Lmp72}, a time-changed exponential of such a L\'{e}vy process determines a \emph{positive self-similar Markov process} (pssMp) that dies continuously at zero. The pssMp $(X^\downarrow_t)$ associated to the \emph{Lamperti representation} $(\xi^\downarrow_s)_{s \geq 0}$ is given by
\begin{align*}
    X^\downarrow_t = \mathrm{e}^{\xi^\downarrow_{s(t)}}, \quad s(t) \coloneqq \inf\left\{ s \geq 0: 
    \int_0^s \mathrm{e}^{(1 + b)\xi^\downarrow_u} \mathrm{d}u \geq t \right\} \in \mathbb{R}\cup \{ \infty\}.
\end{align*}
Here, $\xi^\downarrow_\infty = 0$ by convention. It can be constructed from stable processes confined to the half-line by the type of boundary condition first presented in \cite{TB18}, where these processes were called \emph{ricocheted stable processes}.

In Section \ref{subsec:scalLim} we show that the perimeter process of a non-generic critical Boltzmann triangular loop-decorated map possesses a scaling limit described by a ricocheted stable process. The surprising part of this result is that the obtained limiting process is the same as in the case of non-generic critical Boltzmann rigid loop-decorated maps. This allows us to use the theory about ricocheted stable processes developed in \cite[Section 6]{TB18} to describe the limiting object. Furthermore, we might conclude that the difference in distributions of rescaled perimeter processes of the triangular loop model and the rigid loop model has to be negligible in the limit.
\begin{theorem}[Scaling limit of perimeter process]
\label{thm:scallim}
    Suppose that {\bfseries Assumption} holds. Let $(\mathbf{q}, g, n) \in \mathcal{D}$ be in the non-generic critical phase (dense or dilute) and let $(P_i)_{i \geq 0}$ be the perimeter process of a pointed $(\mathbf{q}, g, n)$-Boltzmann loop-decorated map with perimeter $p$. Then there exists a constant $c > 0$ such that we have the convergence in distribution 
    \begin{align*}
        \left( \frac{P_{\lfloor cp^\theta t \rfloor}}{p} \right)_{t \geq 0} \xrightarrow[p \rightarrow \infty]{(\text{d})} (X^\downarrow_t)_{t \geq 0}
    \end{align*}
    in the $J_1$-Skorokhod topology, where $(X^\downarrow_t)_{t \geq 0}$ is the pssMp defined above with parameter $b = |\theta - 1|$ and $\theta = a - 1$ as in (\ref{asymptotic}). 
\end{theorem}
We emphasize that by restricting to the triplets $(\mathbf{q}, g, n) \in \mathcal{D}$, we assume that the face degrees of the considered maps are bounded.

\section{Peeling of loop-decorated random planar maps}
\label{sec:peeling}
In this chapter we perform a peeling process on a fixed rooted planar map $\mathfrak{m}$ of perimeter $p$ coupled to a triangular loop configuration $\boldsymbol\ell$, i.e. visiting only inner triangles, and on its pointed version $(\mathfrak{m}_\bullet, \boldsymbol\ell)$. Before describing the precise procedure we need to introduce some notions. A \emph{map with holes} is a map together with a distinguished set of faces, called \emph{holes}, that are assumed to be \emph{simple} and vertex-disjoint, i.e. each vertex of the map is incident to at most one corner of a hole. A \emph{hollow map} of perimeter $p$ is the map consisting only of the root face and a single hole both of degree $p$. By convention the root face cannot be a hole. We define a \emph{loop-decorated map with holes} $(\mathfrak{e}, \boldsymbol\ell)$ as a map $\mathfrak{e}$ with holes together with a loop configuration $\boldsymbol\ell$, such that the loops avoid the holes. Throughout this section we assume that some deterministic procedure has been fixed that takes a map with holes and distinguishes an edge on the boundary of each hole.

Given a (loop-decorated) map with holes $\mathfrak{e}$ with a hole $h$ of degree $l$ and a second (loop-decorated) map $\mathfrak{u}$ of perimeter $l$, there is a natural operation of \emph{gluing $\mathfrak{u}$ into $h$} which yields a new map with holes $\mathfrak{e}'$. More precisely, one pairwise identifies the edges in the contour of $h$ with those in the contour of the root face of $\mathfrak{u}$, such that the distinguished edge incident to $h$ is identified with the root edge of $\mathfrak{u}$. This procedure is well-defined since the holes of $\mathfrak{e}$ are simple.

A map $\mathfrak{e}$ with holes $h_1,\ldots, h_k$ is called a \emph{submap} of a map with holes $\mathfrak{e}'$, denoted $\mathfrak{e} \subset \mathfrak{e}'$, if there exists a sequence of maps (with holes) $\mathfrak{u}_1, \ldots, \mathfrak{u}_k$ such that gluing $\mathfrak{u}_i$ into hole $h_i$ for each $i = 1,\ldots, k$ results in $\mathfrak{e}'$. In particular, $\mathfrak{e} \subset \mathfrak{e}$ by choosing all $\mathfrak{u_i}$ to be hollow maps. It is easy to see that the gluing procedure is \emph{rigid}, in the sense that $\mathfrak{e}$ and $\mathfrak{e}'$ uniquely determine the sequence of maps $\mathfrak{u_i}$. Moreover, $\subset$ defines a partial order on the set of all maps with holes.
On the set of all loop-decorated maps with holes we define a partial order relation $\subset$ by letting $(\mathfrak{e}, \boldsymbol\ell) \subset(\mathfrak{e}', \boldsymbol\ell')$ iff $\mathfrak{e} \subset \mathfrak{e}'$ and $\boldsymbol\ell$ contains all the loops of $\boldsymbol\ell'$ that intersect a face of $\mathfrak{e}'$ that also appears in $\mathfrak{e}$.

Let $(\mathfrak{e},\boldsymbol\ell)$ be a loop-decorated map with holes, we divide its edges into two disjoint sets: \emph{active} and \emph{inactive} edges, $\text{Edges}(\mathfrak{e}) = \text{Active}(\mathfrak{e}) \cup \text{Inactive}(\mathfrak{e})$, depending on whether they are incident to a hole or not. So every inactive edge of submap $\mathfrak{e} \subset \mathfrak{m}$ corresponds to a unique edge in $\mathfrak{m}$, while a pair of active edges incident to the same hole $h_i$ can be identified while gluing $\mathfrak{u}_i$ into $h_i$.

Let $(\mathfrak{e},\boldsymbol\ell^*) \subset (\mathfrak{m},\boldsymbol\ell)$ be a loop-decorated submap with holes and $e \in \text{Active}(\mathfrak{e})$. Then there exists a unique smallest (w.r.t. $\subset$) loop-decorated submap $(\mathfrak{e}', \boldsymbol\ell^\prime)$ such that $(\mathfrak{e}, \boldsymbol\ell^*) \subset (\mathfrak{e}', \boldsymbol\ell^\prime) \subset (\mathfrak{m},\boldsymbol\ell)$ and in which $e$ becomes inactive. This map is denoted by $\text{Peel}((\mathfrak{e}, \boldsymbol\ell^*), e, (\mathfrak{m}, \boldsymbol\ell))$ and called the result of \emph{peeling the edge} $e$. Iteration of this procedure builds an increasing sequence of loop-decorated submaps of $(\mathfrak{m},\boldsymbol\ell)$. For uniqueness of such a sequence we fix a \emph{peeling algorithm} $\mathcal{A}$ that associates to any (loop-decorated) map with holes $\mathfrak{e}$ an active edge $\mathcal{A}(\mathfrak{e}) \in \text{Active}(\mathfrak{e})$. Then the \emph{peeling exploration of a loop-decorated map} $(\mathfrak{m}, \boldsymbol\ell)$ with algorithm $\mathcal{A}$ can be defined as the sequence 
\begin{equation}
    \label{loop_peel}
    \begin{aligned}
        (\mathfrak{e}_0, \boldsymbol\ell_0) \subset (\mathfrak{e}_1, \boldsymbol\ell_1) \subset \ldots \subset (\mathfrak{e}_n, \boldsymbol\ell_n) = (\mathfrak{m}, \boldsymbol\ell) \quad
        \text{with} \quad (\mathfrak{e}_{i+1}, \boldsymbol\ell_{i+1}) = \text{Peel}((\mathfrak{e}_i, \boldsymbol\ell_i), \mathcal{A}(\mathfrak{e}_i), (\mathfrak{m}, \boldsymbol\ell))
    \end{aligned}
\end{equation}
being the minimal larger submap containing the peeled edge $\mathcal{A}(\mathfrak{e}_i)$ as an inactive edge, $\mathfrak{e}_0$ is the hollow map of the same perimeter as $\mathfrak{m}$ and $\boldsymbol\ell_0 = \emptyset$.

The operations that can occur when peeling an edge $e \in \text{Active}(\mathfrak{e})$ incident to the hole $h$ are threefold. 
\begin{itemize}
    \item If there is another active edge $e'$ incident to $h$ that gets identified with $e$ in the full map $\mathfrak{m}$, then $\text{Peel}((\mathfrak{e}, \boldsymbol\ell^*), e, (\mathfrak{m}, \boldsymbol\ell))$ is the map with holes obtained from $\mathfrak{e}$ by gluing $e$ to $e'$. This case is denoted by $\mathrm{G}_{k_1, k_2}$, where $k_1$ (resp. $k_2$) is the number of active edges incident to $h$ in between $e$ and $e'$ to the left (resp. right) of $e$ seen from outside the hole. In this situation the number of holes may decrease by one (if $k_1 = k_2 = 0$), remain unchanged (if either $k_1 = 0$ or $k_2 = 0$), or increase by one (if $k_1, k_2 > 0$). 
    \item If $e$ corresponds to an edge in $(\mathfrak{m},\boldsymbol\ell)$ incident to a face of degree $k$ that does not occur in $\mathfrak{e}$ and is not intersected by a loop of $\boldsymbol\ell$, then $\text{Peel}((\mathfrak{e}, \boldsymbol\ell^*), e, (\mathfrak{m}, \boldsymbol\ell))$ is the loop-decorated map with holes obtained from $(\mathfrak{e},\boldsymbol\ell^*)$ by gluing a new $k$-gon to the edge $e$ inside the hole $h$. We denote this case by $\mathrm{C}_k$. The number of holes remains unchanged. 
    \item The last type occurs when the edge $e$ is incident to a new triangle in $(\mathfrak{m},\boldsymbol\ell)$ intersected by a loop $\ell \in \boldsymbol\ell$ of length $|\ell|$. To be consistent with the indexing of these events we assume that all loops are oriented counterclockwise seen from inside of the loop. We start with the triangle with the peeled edge as a base and assign $k_1 = \texttt{out}$ to it, further we move in the direction of the loop and assign to each next triangle a parameter $k_i \in \{\texttt{in}, \texttt{out}\}$, where $k_i = \texttt{out}$ if two vertices of the current triangle are outside of the loop, $k_i = \texttt{in}$ otherwise. Any loop-decorated submap containing $e$ as an inactive edge should contain the full loop $\ell$. Hence, the minimal such submap $(e_{i+1}, \boldsymbol\ell_{i+1})$ is obtained by gluing a \emph{ring of length $|\ell|$ and configuration $k_1, \ldots, k_{|\ell|}$} inside the hole of $\mathfrak{e}_i$ to the edge $e$. A \emph{ring of length $m$ of configuration $k_1, \ldots, k_{m}$} is a loop-decorated map of perimeter $v \coloneqq |\{ k_i = \texttt{out}\}|$ with a single hole of degree $u \coloneqq |\{ k_i = \texttt{in}\}|$ constructed by gluing $v$ triangles from outer side and $u$ triangles from side of the hole into a ring and covering it by a single loop of length $m = u + v$. We denote the case of gluing of the ring of configuration $k_1, \ldots, k_{|\ell|}$ by $\mathrm{L}_{k_1, \ldots, k_{|\ell|}}$. This event increases the number of holes by one (if $|\{ k_i = \texttt{in}\}| > 0$) or does not change it (if $|\{ k_i = \texttt{in}\}| = 0$). 
\end{itemize}
Often a configuration of edges crossed by a loop will not be important for us, so we introduce the notation $\mathrm{L}^{u}_v$ for the set of events $\mathrm{L}_{k_1, \ldots, k_{u + v}}$ with $|\{ k_i = \texttt{in}\}| = u, |\{k_i = \texttt{out}\}| = v$. 

In all three cases the number of inactive edges not traversed by a loop of $\boldsymbol\ell$ increases exactly by one, therefore, the number of steps in the peeling exploration (\ref{loop_peel}) is exactly the number of edges of $\mathfrak{m}$ not traversed by a loop of $\boldsymbol\ell$.

When the peeling algorithm $\mathcal{A}$ and the perimeter of the map $\mathfrak{m}$ are fixed, the loop-decorated map $(\mathfrak{m}, \boldsymbol\ell)$ is completely determined by a finite sequence of events in $\{\mathrm{G}_{k_1,k_2}, \mathrm{C}_k$, $\mathrm{L}_{k_1, \ldots, k_m}\}$. This fact will play an important role in the proof of Theorem \ref{thm:admissibility}.

If we now distinguish a vertex of $\mathfrak{m}$ - call it target - we get a pointed loop-decorated map $(\mathfrak{m}_\bullet, \boldsymbol\ell)$. Let $(\mathfrak{e},\boldsymbol\ell^*) \subset (\mathfrak{m}_\bullet, \boldsymbol\ell)$ be a loop-decorated submap with holes $h_1,\ldots,h_k$. Let $\mathfrak{u}_1, \ldots, \mathfrak{u}_k$ be unique maps to be glued in the holes of $\mathfrak{e}$ to retrieve $(\mathfrak{m}_\bullet,\boldsymbol\ell)$. The \emph{filled-in} submap $\text{Fill}((\mathfrak{e},\boldsymbol\ell^*), (\mathfrak{m}_\bullet,\boldsymbol\ell)) \subset (\mathfrak{m}_\bullet,\boldsymbol\ell)$ is the loop-decorated map obtained from $\mathfrak{e}$ by gluing $\mathfrak{u}_i$ into $h_i$ if $\mathfrak{u}_i$ does not contain the marked vertex for each $i = 1, \ldots, k$. Since the holes of $\mathfrak{e}$ are disjoint, the filled-in submap has at most one hole.
 
Now we can define the \emph{targeted peeling exploration} of $(\mathfrak{m}_\bullet,\boldsymbol\ell)$ with algorithm $\mathcal{A}$ by iteratively peeling an edge followed by filling in the unpointed holes, i.e. as the sequence 
\begin{equation}
    \begin{aligned}
        (\mathfrak{e}_0, \boldsymbol\ell_0) &\subset (\mathfrak{e}_1, \boldsymbol\ell_1) \subset \ldots \subset (\mathfrak{e}_n, \boldsymbol\ell_n) = (\mathfrak{m}_\bullet, \boldsymbol\ell) \\
        &\text{with} \; (\mathfrak{e}_{i+1},\boldsymbol\ell_{i+1}) = \text{Fill}( \text{Peel}((\mathfrak{e}_i,\boldsymbol\ell_{i}), \mathcal{A}(\mathfrak{e}_i), (\mathfrak{m}_\bullet,\boldsymbol\ell)), (\mathfrak{m}_\bullet, \boldsymbol\ell)).
    \end{aligned}
    \label{loop_target_peel}
\end{equation}
\begin{figure}[ht]
    \centering
    \includegraphics[width = \linewidth]{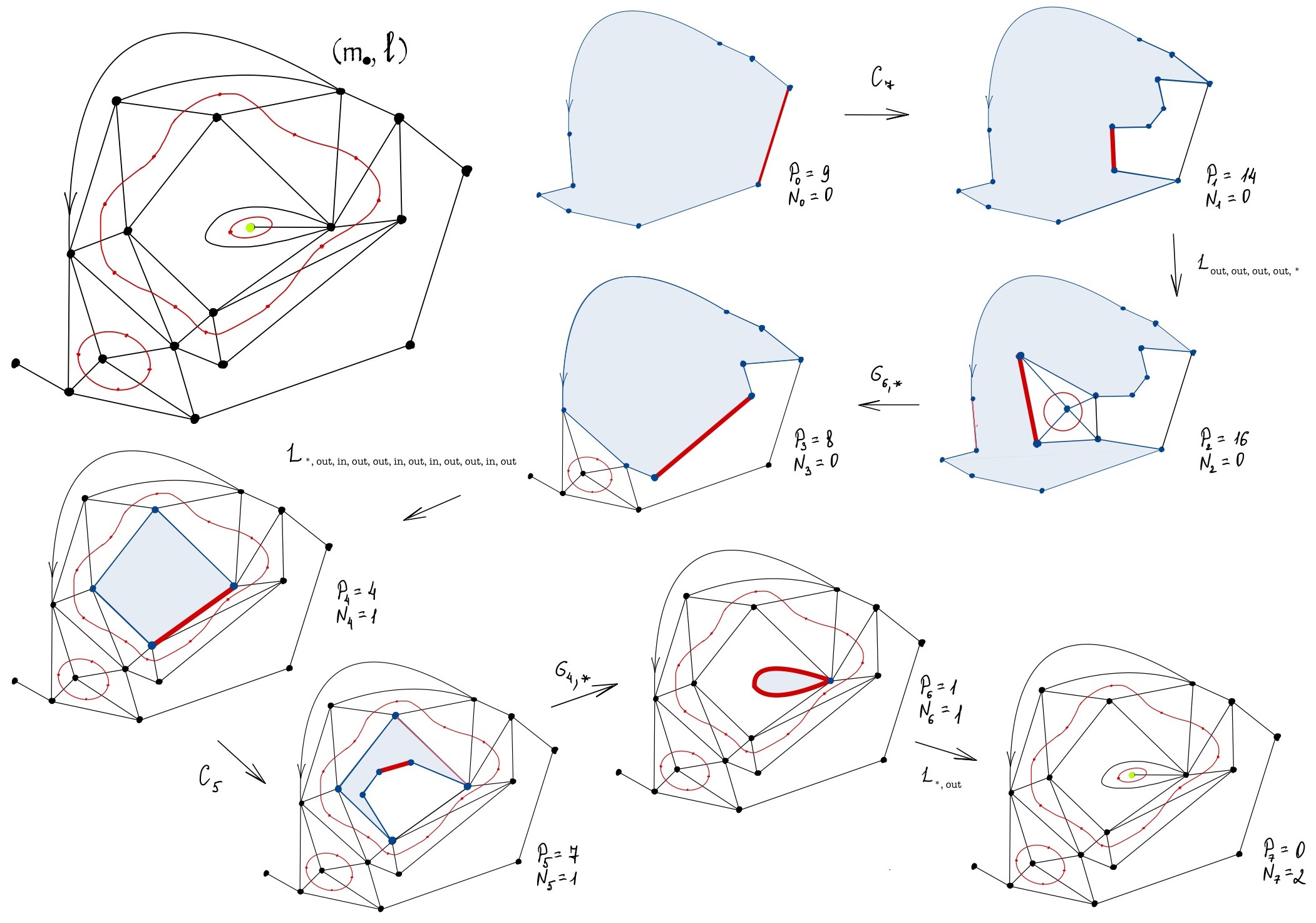}
    \caption{The targeted peeling of the pointed (green vertex) loop-decorated map $(\mathfrak{m}_\bullet, \boldsymbol\ell)$. The peeled edges are indicated by thick red lines, the inactive edges and inner vertices are black. The corresponding perimeter process is $(P_i)_{i = 0}^7 = (9, 14, 16, 8, 4, 7, 1, 0)$ and the nesting process is $(N_i)_{i = 0}^7 = (0, 0, 0, 0, 1, 1, 1, 2)$.}
    \label{fig:loop_peeling}
\end{figure}

The operations that can occur upon peeling $\mathcal{A}(\mathfrak{e}_i)$ in $(\mathfrak{e}_i, \boldsymbol\ell_i)$ and then filling are of five types. Namely, 
\begin{itemize}
    \item $\mathrm{C}_k$ as before since in this case no filling is needed;
    \item $\mathrm{G}_{k_1,*}$ upon filling in the left hole in the event $\mathrm{G}_{k_1,k_2}$, and $\mathrm{G}_{*,k_2}$ upon filling in the right hole in the event $\mathrm{G}_{k_1,k_2}$;
    \item $\mathrm{L}_{k_1, \ldots, k_{|\ell|}, *}$ (resp. $\mathrm{L}_{*, k_1, \ldots, k_{|\ell|}}$) corresponding to the gluing of a ring via event $\mathrm{L}_{k_1, \ldots, k_{|\ell|}}$ followed by filling in the hole inside (resp. outside) the ring (see Figure \ref{fig:loop_peeling}).
\end{itemize} 
Analogously to untargeted peeling we introduce the notions $\mathrm{L}^{*, u}_v, \mathrm{L}^u_{v, *}$ for the sets of events $\mathrm{L}_{*, k_1, \ldots, k_{u + v}}$ and $\mathrm{L}_{k_1, \ldots, k_{u + v}, *}$, respectively.

Since for each $i = 0, \ldots, n-1$ the submap $(\mathfrak{e}_i, \boldsymbol\ell_{i})$ has exactly one hole, we can introduce the \emph{perimeter process} $(P_i)_{i=0}^{n}$ by setting $P_i$ equal to the degree of the hole of $\mathfrak{e}_i$ and $P_n=0$ by convention. Note that $P_0$ is the perimeter of $\mathfrak{m}_\bullet$ and that the perimeter process depends only on $(\mathfrak{m}_\bullet, \boldsymbol\ell)$ and the chosen peeling algorithm $\mathcal{A}$.
On top of the perimeter process we can also keep track of the \emph{nesting process} $(N_i)_i$ such that $N_0=0$ and $N_i$ increases, $N_{i+1}=N_i+1$, only in the event $\mathrm{L}_{*,k_1, \ldots, k_m}$. Then $N_n$ is precisely the \emph{nesting statistic} of $(\mathfrak{m}_\bullet, \boldsymbol\ell)$, i.e. the number of loops in $\boldsymbol\ell$ that surround the marked vertex seen from the root face.

Note that both targeted and untargeted peeling processes introduced above apply as well to undecorated maps $\mathfrak{m}$ by identifying $\mathfrak{m}$ with $(\mathfrak{m}, \boldsymbol\ell= \emptyset)$.

\section{Peeling process of Boltzmann loop-decorated maps}
\label{sec:Boltzmann}
In this section we study the law of the peeling exploration applied to a Boltzmann loop-decorated map. Both the untargeted exploration and the targeted exploration of a pointed Boltzmann loop-decorated map are discussed. Before proceeding to this application, we present the gasket decomposition result from \cite{BBG11}, which delivers one direction of the equivalence of admissibility criteria of Theorem \ref{thm:admissibility}. The proof of this result was summarized in \cite{TB18} for the rigid $O(n)$ model and can be easily extended to our setting.

Recall from Section \ref{subsubsec:intro_loop_model} that the gasket $\mathfrak{g}(\mathfrak{m}, \boldsymbol\ell)$ of a loop-decorated map $(\mathfrak{m}, \boldsymbol\ell) \in \mathcal{LM}^{(p)}$ is the undecorated map obtained from $\mathfrak{m}$ by removing all edges intersected by a loop in $\boldsymbol\ell$ and retaining only the connected component containing the root. 
\begin{lemma}
\label{lemma:gasket}
    The gasket of $(\mathbf{q}, g, n)$-Boltzmann loop-decorated map is itself a $\hat{\mathbf{q}}$-Boltzmann map, where the \emph{effective weight sequence} $\hat{\mathbf{q}}$ is defined via the fixed-point equation (\ref{fpe}),
    \begin{align*}
        \hat{q}_k = q_k + n \sum_{l \geq 0} \binom{k + l - 1}{l} g^{k + l} F^{(l)}(\mathbf{q}, g, n).
    \end{align*} 
    In particular, if $(\mathbf{q}, g, n)$ is admissible, so is $\hat{\mathbf{q}}$.
\end{lemma}

\subsection{Untargeted exploration}
Let $(\mathbf{q}, g, n)$ be an admissible triplet and let $(\mathfrak{m}, \boldsymbol\ell)$ be a $(\mathbf{q}, g, n)$-Boltzmann loop-decorated map of perimeter $p > 0$. It satisfies the following Markov property, which is an analogue of \cite[Lemma 1]{TB18}. 
\begin{lemma}
\label{lemma:Markov_untarg}
    Let $(\mathfrak{e}, \boldsymbol{\ell}')$ be a fixed loop-decorated map of perimeter $p$ with holes $h_1, \ldots, h_k$. If $(\mathfrak{e}, \boldsymbol{\ell}') \subset (\mathfrak{m}, \boldsymbol\ell)$ with positive probability, then conditionally on this event the loop-decorated maps $\mathfrak{u}_1,\ldots,\mathfrak{u}_k$ filling in the holes $h_1,\ldots,h_k$ are distributed as independent $(\mathbf{q},g,n)$-Boltzmann loop-decorated maps of perimeters $\mathrm{deg}(h_1),\ldots,\mathrm{deg}(h_k)$.
\end{lemma}
We omit the proof here since it follows closely the lines of the proof in \cite{TB18}. We as well refer to \cite{thesis} for the details.

We assume that a peeling algorithm $\mathcal{A}$ is fixed, deterministic or probabilistic. In the latter case we require independence of the choice of a peeled edge $\mathcal{A} (\mathfrak{e}, \boldsymbol{\ell}') \in \text{Active}(\mathfrak{e})$ of $(\mathfrak{m}, \boldsymbol\ell)$ conditioned on $(\mathfrak{e}, \boldsymbol{\ell}') \subset (\mathfrak{m}, \boldsymbol{\ell})$. Let $(\mathfrak{e}_0, \boldsymbol\ell_0) \subset (\mathfrak{e}_1, \boldsymbol\ell_1) \subset \ldots \subset (\mathfrak{e}_n, \boldsymbol\ell_n) = (\mathfrak{m},\boldsymbol\ell)$ be the associated peeling exploration. By the Markov property and the convention $F^{(0)}(\mathbf{q}, g, n)=1$ if the degree of the hole incident to the peeled edge $e$ is $l$ we get the transition probabilities: 
\begin{align}
\label{trans_prob}
    \mathbb{P}[\mathrm{L}_{k_1, \ldots, k_m} |\; \mathfrak{e}_i,\boldsymbol\ell_i,e] &= 
    \frac{ng^{m} F^{(u)}F^{(l + m - u - 2)}}{F^{(l)}} \; && \text{for} \; u = |\{k_i = \texttt{in}\}|, \; m \geq 1, \; k_1 = \texttt{out}, \nonumber \\
    \mathbb{P}[\mathrm{G}_{k_1, k_2} |\; \mathfrak{e}_i, \boldsymbol\ell_i, e] &= 
    \frac{F^{(k_1)} F^{(k_2)}}{F^{(l)}} \quad && \text{for} \; k_1, k_2 \geq 0, \; k_1 + k_2 + 2 = l,\\
    \mathbb{P}[\mathrm{C}_{k} |\; \mathfrak{e}_i, \boldsymbol\ell_i, e] &= 
    \frac{q_k F^{(l + k - 2)}}{F^{(l)}} \quad && \text{for} \; k \geq 1, \nonumber
\end{align}
where for simplicity $F^{(i)} = F^{(i)}(\mathbf{q}, g, n)$ for all $i \geq 0$. For later usage we also need 
\begin{equation}
    \mathbb{P}[\mathrm{L}^u_{v} |\; \mathfrak{e}_i,\boldsymbol\ell_i,e] = \binom{u + v - 1}{u} \frac{ng^{u + v} F^{(u)}F^{(l + v - 2)}}{F^{(l)}} \quad \text{for} \; u \geq 0, \; v \geq 1. \label{L-prob}
\end{equation}

\subsection{Targeted exploration of pointed Boltzmann loop-decorated maps}
\label{subsec:target_exploration}
In this section a pointed $(\mathbf{q}, g, n)$-Boltzmann loop-decorated map $(\mathfrak{m}_\bullet, \boldsymbol\ell)$ of perimeter $p$ is studied. Its probability is proportional to $\mathbf{w}_{\mathbf{q}, g, n}(\mathfrak{m}_\bullet, \boldsymbol\ell)$. For this to be well-defined, we need the \emph{pointed partition function}
\begin{align}
    F_\bullet^{(p)}(\mathbf{q}, g, n) \coloneqq \sum_{(\mathfrak{m}_\bullet, \boldsymbol\ell) \in \mathcal{LM}_\bullet^{(p)}} \mathbf{w}_{\mathbf{q}, g, n}(\mathfrak{m}_\bullet, \boldsymbol\ell)
\end{align}
to be finite. This will be a consequence of Theorem \ref{thm:admissibility}, but for now we make the following stronger assumption.

Let $(\mathbf{q}, g, n)$ be \emph{strongly admissible}, i.e. $F_\bullet^{(p)}(\mathbf{q}, g, n) < \infty$ for all $p \geq 1$, and $(\mathfrak{m}_\bullet, \boldsymbol\ell)$ be a pointed $(\mathbf{q}, g, n)$-Boltzmann loop-decorated map of perimeter $p$. One can formulate the Markov property result equivalent to Lemma \ref{lemma:Markov_untarg} in the pointed case:
\begin{lemma}
\label{lemma:Markov_targ}
    Let $(\mathfrak{e}, \boldsymbol\ell')$ be a fixed loop-decorated map of perimeter $p$ with a single hole $h$. If with positive probability $(\mathfrak{e}, \boldsymbol\ell') \subset (\mathfrak{m}_\bullet, \boldsymbol\ell)$ and the marked vertex of $\mathfrak{m}_\bullet$ is not an \emph{inner vertex} of $\mathfrak{e}$ (inner vertices of $\mathfrak{e}$ are those not incident to a hole), then the pointed map $\mathfrak{u}_\bullet$ filling in the hole $h$ is distributed as a pointed $(\mathbf{q},g,n)$-Boltzmann loop-decorated map of perimeter $\mathrm{deg}(h)$.
\end{lemma}

Now using Lemma \ref{lemma:Markov_targ} we can easily study the distribution of the targeted peeling exploration (\ref{loop_target_peel}),
\begin{align*}
    (\mathfrak{e}_0, \boldsymbol\ell_0) \subset (\mathfrak{e}_1, \boldsymbol\ell_1) \subset \ldots \subset (\mathfrak{e}_n, \boldsymbol\ell_n) = (\mathfrak{m}_\bullet, \boldsymbol\ell),
\end{align*}
of a pointed $(\mathbf{q},g,n)$-Boltzmann loop-decorated map $(\mathfrak{m}_\bullet, \boldsymbol{\ell})$ of perimeter $p$. Conditionally on $(\mathfrak{e}_i,\boldsymbol\ell_i)$, if the degree of the hole of $\mathfrak{e}_i$ is $P_i = l$, then the exploration events $\mathrm{C_k}, \mathrm{G}_{k, *}, \mathrm{G}_{*, k}, \mathrm{L}_{k_1, \ldots, k_m, *}, \mathrm{L}_{*, k_1, \ldots, k_m}$ with $k_1 = \texttt{out}$ occur with probabilities:
\begin{align}
\label{targ_trans_prob}
    \mathbb{P}[\mathrm{L}_{*, k_1, \ldots, k_m} |\; \mathfrak{e}_i,\boldsymbol\ell_i, e] &= 
    \frac{ng^{m} F_\bullet^{(u)} F^{(l + m - u - 2)}}{F_\bullet^{(l)}} \; && \text{for} \; u = |\{k_i = \texttt{in}\}|, \; m \geq 1, \nonumber \\
    \mathbb{P}[\mathrm{L}_{k_1, \ldots, k_m, *} |\; \mathfrak{e}_i,\boldsymbol\ell_i, e] &= 
    \frac{ng^{m} F^{(u)} F_\bullet^{(l + m - u - 2)}}{F_\bullet^{(l)}} \; && \text{for} \; u = |\{k_i = \texttt{in}\}|, \; m \geq 1, \\
    \mathbb{P}[\mathrm{G}_{k, *} |\; \mathfrak{e}_i, \boldsymbol\ell_i, e] 
    = \mathbb{P}[\mathrm{G}_{*, k} |\; \mathfrak{e}_i, \boldsymbol\ell_i, e] &= 
    \frac{F^{(k)} F_\bullet^{(l - k - 2)}}{F_\bullet^{(l)}} \quad && \text{for} \; 0 \leq k \leq l - 2, \nonumber\\
    \mathbb{P}[\mathrm{C}_{k} |\; \mathfrak{e}_i, \boldsymbol\ell_i, e] &= 
    \frac{q_k F_\bullet^{(l + k - 2)}}{F_\bullet^{(l)}} \quad && \text{for} \; k \geq 1, \nonumber
\end{align}
where $F^{(i)} = F^{(i)}(\mathbf{q}, g, n)$ for all $i \geq 0$, and by convention $F_\bullet^{(0)} = 1$. We also have
\begin{equation}
    \begin{aligned}
        \mathbb{P}[\mathrm{L}^u_{v, *} |\; \mathfrak{e}_i,\boldsymbol\ell_i, e] &= \binom{u + v - 1}{u} \frac{ng^{u + v} F^{(u)} F_\bullet^{(l + v - 2)}}{F_\bullet^{(l)}} \quad &&\text{for} \; u \geq 0, \; v \geq 1,\\
        \mathbb{P}[\mathrm{L}^{*, u}_v |\; \mathfrak{e}_i,\boldsymbol\ell_i, e] &= \binom{u + v - 1}{u} \frac{ng^{u + v} F_\bullet^{(u)} F^{(l + v - 2)}}{F_\bullet^{(l)}} \quad &&\text{for} \; u \geq 0, \; v \geq 1.
    \end{aligned}
\label{L*-prob}
\end{equation}
For simplicity when we do not need to know the precise event $\mathrm{L}_{k_1, \ldots, k_{u + v}}$, we directly work with $\mathrm{L}^u_v$ instead without further specification.
After the peeling step the degree $P_{i + 1}$ of the hole is given by
\begin{equation*}
    P_{i+1} = 
    \begin{cases}
        P_i - k - 2 \; &\text{after } \mathrm{G}_{k, *} \; \text{or} \; \mathrm{G}_{*,k},\\
        P_i + v - 2 \; &\text{after } \mathrm{C}_{v} \; \text{or} \; \mathrm{L}_{k_1, \ldots, k_{u + v}, *} \in \mathrm{L}^u_{v, *},\\
        u \; &\text{after } \mathrm{L}_{*, k_1, \ldots, k_{u + v}} \in \mathrm{L}^{*, u}_v.
    \end{cases}
\end{equation*}
We conclude that the perimeter and nesting process $(P_i, N_i)$ of $(\mathfrak{m}_\bullet, \boldsymbol\ell)$ is a Markov process with transition probabilities given explicitly by
\begin{equation}
\label{trans_tuple}
    \begin{aligned}
        \mathbb{P}[P_{i + 1} = p, \; N_{i + 1} = N_i + 1 |\; P_i = l] &= \frac{F_\bullet^{(p)}}{F_\bullet^{(l)}} \sum_{v \geq 1} \binom{v + p - 1}{p} ng^{v + p} F^{(l + v - 2)}, \\
        \mathbb{P}[P_{i + 1} = p, \; N_{i + 1} = N_i |\; P_i = l] &= \frac{ F_\bullet^{(p)}}{F_\bullet^{(l)}}
        \begin{cases}
            \hat{q}_{p - l + 2} \; &\text{for} \; p \geq l - 1, \\
            2F^{(l - p - 2)} \; &\text{for} \; p < l - 1.
        \end{cases}
    \end{aligned}
\end{equation}
where the fixed point equation (\ref{fpe}) and transition probabilities (\ref{targ_trans_prob}), (\ref{L*-prob}) were used. 

Now we introduce the notion from \cite{TB16}, which we will need later to establish the connection between the exploration process and the so-called accelerated ricocheted random walk:
\begin{align}
\label{nu}
    \nu_{\hat{\mathbf{q}}}(k) \coloneqq \gamma_+^k (\hat{\mathbf{q}})
    \begin{cases}
        \hat{q}_{k + 2} \; &\text{for} \; k \geq -1, \\
        2W^{(-k - 2)}(\hat{\mathbf{q}}) \; &\text{for} \; k < -1,
    \end{cases}
\end{align}
where $\gamma_+(\hat{\mathbf{q}})$ is defined by (\ref{gamma+}). Once $\hat{\mathbf{q}}$ is admissible, which is the case here, one can prove that $\nu_{\hat{\mathbf{q}}}$ is a probability measure on $\mathbb{Z}$ using \cite{TB16} - see Appendix A.2 in \cite{thesis}. With this and the equality $F^{(p)}(\mathbf{q}, g, n) = W^{(p)}(\hat{\mathbf{q}})$, which follows from Lemma \ref{lemma:gasket}, the transition probabilities (\ref{trans_tuple}) can be rewritten as 
\begin{align}
\label{trans_nu}
    \mathbb{P}[P_{i + 1} = p, \; N_{i + 1} = N_i + 1 |\; P_i = l] &= 
    \frac{F_\bullet^{(p)} \gamma_+^{-p}}{F_\bullet^{(l)} \gamma_+^{-l}} \sum_{v \geq 1} \binom{v + p - 1}{p} \frac{n}{2} \left(g\gamma_+\right)^{p + v} \nu_{\hat{\mathbf{q}}}(-l - v), \\
    \mathbb{P}[P_{i + 1} = p, \; N_{i + 1} = N_i |\; P_i = l] &= 
    \frac{F_\bullet^{(p)} \gamma_+^{-p}}{F_\bullet^{(l)} \gamma_+^{-l}} \nu_{\hat{\mathbf{q}}}(p - l), \nonumber
\end{align}
where $l > 0, p \geq 0$ and $\gamma_+ = \gamma_+(\hat{\mathbf{q}})$. In particular, for $l > 0$ the following holds
\begin{align}
\label{rec_eq_F.gamma+}
    F_\bullet^{(l)} \gamma_+^{-l} = \sum_{p \geq 0} F_\bullet^{(p)} \gamma_+^{-p} \left( \nu_{\hat{\mathbf{q}}}(p - l) + \frac{n}{2} \sum_{k \geq 1} \binom{k + p - 1}{p} \left(g\gamma_+\right)^{k + p} \nu_{\hat{\mathbf{q}}}(-k - l) \right).
\end{align}

\subsection{Separating loops on pointed Boltzmann loop-decorated maps}
In this section we present additional definitions and results which will be of importance in the proof of Proposition \ref{prop:h-props} $(\romannum{3})$. The latter, in turn, provides the asymptotic behaviour of the functions $h^r_\mathfrak{p}$ appearing in Theorem \ref{thm:nesting} and is crucial for the proof of Theorem \ref{thm:scallim}.

We say that a loop in $\boldsymbol\ell$ with $(\mathfrak{m}_\bullet, \boldsymbol\ell) \in \mathcal{LM}_\bullet^{(p)}$ is \emph{separating} if, after its removal together with the edges intersected by it, one connected component of the map contains the marked vertex (target) and the other contains the root and root face.

Following \cite{BBD18} we generalize the model by assigning an additional weight $s \in \left(0, \frac{2}{n}\right)$ to the separating loops. Let $(\mathbf{q}, g, n)$ be strongly admissible, $n \in (0, 2)$, and for now additionally assume that the partition function corresponding to this model, denoted as $F^{(p)}_\bullet[s](\mathbf{q}, g, n)$, is finite. We will see in the proof of Proposition \ref{prop:h-props} $(\romannum{3})$ that the additional assumption is redundant as long as we are in the non-generic critical phase. Since the performed change is deterministic the exploration process in this situation fulfils the Markov property analogous to Lemma \ref{lemma:Markov_targ}. The only modification is that $\mathfrak{u}_\bullet$ is distributed as a pointed $(\mathbf{q}, g, n)$-Boltzmann loop-decorated map of perimeter $\mathrm{deg}(h)$ with additional weight $s$ per separating loop. Then the transition probabilities (\ref{targ_trans_prob}) for the peeling events have to be modified as follows
\begin{align}
\label{targ_s_trans_prob}
    \mathbb{P}[\mathrm{L}_{*, k_1, \ldots, k_m} |\; \mathfrak{e}_i,\boldsymbol\ell_i, e] &= 
    \frac{sng^{m} F_\bullet^{(u)}[s] F^{(l + m - u - 2)}}{F_\bullet^{(l)}[s]} \; && \text{for} \; u = |\{k_i = \texttt{in}\}|, \; m \geq 1, \nonumber \\
    \mathbb{P}[\mathrm{L}_{k_1, \ldots, k_m, *} |\; \mathfrak{e}_i,\boldsymbol\ell_i, e] &= 
    \frac{ng^{m} F^{(u)} F_\bullet^{(l + m - u - 2)}[s]}{F_\bullet^{(l)}[s]} \; && \text{for} \; u = |\{k_i = \texttt{in}\}|, \; m \geq 1, \\
    \mathbb{P}[\mathrm{G}_{k, *} |\; \mathfrak{e}_i, \boldsymbol\ell_i, e] 
    = \mathbb{P}[\mathrm{G}_{*, k} |\; \mathfrak{e}_i, \boldsymbol\ell_i, e] &= 
    \frac{F^{(k)} F_\bullet^{(l - k - 2)}[s]}{F_\bullet^{(l)}[s]} \quad && \text{for} \; 0 \leq k \leq l - 2, \nonumber \\
    \mathbb{P}[\mathrm{C}_{k} |\; \mathfrak{e}_i, \boldsymbol\ell_i, e] &= 
    \frac{q_k F_\bullet^{(l + k - 2)}[s]}{F_\bullet^{(l)}[s]} \quad && \text{for} \; k \geq 1. \nonumber
\end{align}

Following the above steps taken for the model with $s = 1$, which is exactly the targeted exploration of pointed Boltzmann loop-decorated maps in Section \ref{subsec:target_exploration}, we get from (\ref{targ_s_trans_prob}) via (\ref{nu}) analogously to (\ref{rec_eq_F.gamma+}),
\begin{align*}
    F_\bullet^{(l)}[s] \gamma_+^{-l} = \sum_{p \geq 0} F_\bullet^{(p)}[s] \gamma_+^{-p} \left( \nu_{\mathbf{\hat{q}}}(p - l) + \frac{sn}{2} \sum_{k \geq 1} \binom{k + p - 1}{p} \left(g\gamma_+\right)^{k + p} \nu_{\mathbf{\hat{q}}}(-k - l) \right).
\end{align*}
In the non-generic critical case, i.e. $g\gamma_+ = 1/2$, the recursive equation reads as
\begin{align}
\label{rec_eq_F_s.gamma+}
    F_\bullet^{(l)}[s] \gamma_+^{-l} = \sum_{p \geq 0} F_\bullet^{(p)}[s] \gamma_+^{-p} \left( \nu_{\mathbf{\hat{q}}}(p - l) + \frac{sn}{2} \sum_{k \geq 1} \binom{k + p - 1}{p} \left(\frac{1}{2}\right)^{k + p} \nu_{\mathbf{\hat{q}}}(-k - l) \right).
\end{align}

\section{Accelerated ricocheted random walks}
\label{sec:ricochetedRW}
The main goal of this chapter is to give an interpretation of the law (\ref{trans_nu}) in the non-generic critical case as a Doob-transform of a random walk with law $\nu$ that is restricted to the non-negative integers by a specific boundary condition. For this reason we extensively study properties of $\nu$ and start working with functions $h^r_{\mathfrak{p}}$ defined in (\ref{h-fct}). This chapter is a non-trivial adaptation of the results of \cite[Chapter 4]{TB18} to our setting.

Suppose $\nu: \mathbb{Z} \rightarrow \mathbb{R}_+$ is non-negative and $\sum_{k \in \mathbb{Z}} \nu(k) = 1$, such that it defines a probability measure on $\mathbb{Z}$. Let $(W_i)_{i \geq 0}$ be a random walk started at $p \in \mathbb{Z}$ under $\mathbb{P}_p$ with independent increments distributed according to $\nu$. The random walk $(W_i)_{i \geq 0}$ is said to \emph{drift to $\pm \infty$} if $\lim_{i \rightarrow \infty} W_i = \pm \infty$ $\mathbb{P}_p$-almost surely . If none of these hold, it is said to \emph{oscillate}.

\subsection{Admissibility criteria}
Let us define $h^r_0: \mathbb{Z} \rightarrow \mathbb{R}$ as in (\ref{h-fct}),
\begin{align*}
    h^r_0(p) &= \left(\frac{-r}{4}\right)^p \binom{2p}{p} \hypgeom\left(\frac{1}{2}, -p; \; \frac{1}{2} - p; -\frac{1}{r}\right) \mathds{1}_{\{p \geq 0\}}.
\end{align*}
In particular, these are coefficients of the generating function $1 / \sqrt{(1 - z)(1 + rz)}$ for $z \in [-1, 1)$.  

We say $h^r_0$ is \emph{$\nu$-harmonic} on $\mathbb{N}$ if for all $l \geq 1$,
\begin{align*}
    \sum_{k \in \mathbb{Z}} h^r_0(l + k) \nu(k) = h^r_0(l).
\end{align*}
In \cite[Proposition 3]{TB16} it was shown that the mapping $\mathbf{\hat{q}} \mapsto \nu_{\mathbf{\hat{q}}}$ given by (\ref{nu}) determines a one-to-one correspondence between admissible weight sequences $\mathbf{\hat{q}}$ and probability measures $\nu$ for which there is $r \in (-1,1]$ such that $h^r_0$ is $\nu$-harmonic on $\mathbb{N}$. Even though there $\mathbf{\hat{q}}$ was called admissible iff $W_\bullet^{(2)}(\mathbf{\hat{q}}) < \infty$, it is compatible with our definition under the additional assumption of $\nu(-2) > 0$. Indeed, since $W^{(l)}(\mathbf{\hat{q}}) \leq W^{(l)}_\bullet(\mathbf{\hat{q}})$, one direction is obvious, the other direction is proven in Appendix \ref{A:subsec:admis_no_loop}. We slightly extend the mentioned result in the proposition below. The same was done in \cite{TB18} for the rigid loop model on bipartite maps and their proof can be easily extended to our setting by substituting $h^\downarrow_0$ in \cite{TB18} with $h_0^r$ defined above. We refer the interested reader to \cite{thesis} for details.
\begin{proposition}
\label{prop:nu_admissible}
    Let $\nu: \mathbb{Z} \rightarrow \mathbb{R}$ be a probability measure on $\mathbb{Z}$ with $\nu(-2) > 0$. Then the following are equivalent:
    \setlist{nolistsep}
    \begin{enumerate}[label = (\roman*), itemsep = 0.5em]
        \item $\nu = \nu_{\hat{\mathbf{q}}}$ for some admissible sequence $\hat{\mathbf{q}}$.
        \item There is $r \in (-1,1]$ such that the function $h^r_0$ is $\nu$-harmonic on $\mathbb{N}$.
        \item There is $r \in (-1,1]$ such that the strict descending ladder process of the random walk $(W_i)$ with law $\nu$ has probability generating function $G^<(z) = 1 - \sqrt{(1 - z)(1 + rz)}$.
    \end{enumerate}
    The $r$ in $(\romannum{2})$ and $(\romannum{3})$ is the same. It is uniquely determined by $\hat{\mathbf{q}}$ from $(\romannum{1})$ and is given by $r = - \frac{\gamma_-(\hat{\mathbf{q}})}{\gamma_+(\hat{\mathbf{q}})}$. Conversely, $r$ uniquely determines $\hat{\mathbf{q}}$ such that \eqref{nu} is satisfied.
\end{proposition}
Recall that the \emph{strict descending ladder epochs} $(T_j^<)$ are defined by $T^<_0 = 0$ and 
\begin{align*} 
    T_{j + 1}^< \coloneqq \inf\left\{i > T_j^<: W_i < W_{T_j^<}\right\} \in \mathbb{Z} \cup \{ \infty\} 
    \quad \text{for} \; j \geq 0.
\end{align*}
The \emph{strict descending ladder process} $(H_j^<)$ is given by $H_j^< \coloneqq -W_{T_j^<}$ provided $T_j^< < \infty$ and otherwise $H_j^< \coloneqq \dagger$. It is a defective walk if $(W_i)$ drifts to $\infty$ and proper otherwise.

We call $\nu: \mathbb{Z} \rightarrow \mathbb{R}$ satisfying any of the equivalent conditions in Proposition \ref{prop:nu_admissible} \emph{admissible}.

In particular, we see that the law of the strict descending ladder process is almost universal for any admissible $\nu$, meaning that it depends on the single parameter $r$. This is useful for the following statement, which is a counterpart of the result \cite[Lemma 2]{TB18}. Let $H_l^r: \mathbb{Z} \rightarrow \mathbb{R}$ for $l \geq 0$ be defined as 
\begin{align}
    \label{H-fct}
    H_l^r(p) = \frac{p}{l + p} \left( h^r_0(p) h^r_0(l) + r h^r_0(p - 1) h^r_0(l - 1)\right)\mathds{1}_{\{p> 0\}} + \mathds{1}_{\{l = p = 0\}}
\end{align}
so that $H_0^r = h^r_0$ as $h^r_0 (-1) = 0$.
\begin{lemma}
\label{lemma:H-fct}
    If $\nu$ is admissible, for $l \geq 0, p \geq 1$ we have 
    \begin{align*}
        \mathbb{P}_p [(W_i) \; \text{hits} \;\; \mathbb{Z}_{\leq 0} \; \text{at} \; -l] = H_l^r(p),
    \end{align*}
    where $r$ is the one given by Proposition \ref{prop:nu_admissible}.
\end{lemma}
\begin{proof}
Analogously to the proof of \cite[Lemma 2]{TB18} using Proposition \ref{prop:nu_admissible} we obtain the following equation for $P_l(p) \coloneqq \mathbb{P}_p [(W_i) \; \text{hits} \;\; \mathbb{Z}_{\leq 0} \; \text{at} \; -l]$ with $P_l(0) = \mathds{1}_{\{l = 0\}}$,
\begin{equation*}
    h^r_0(p + k) = \sum_{l = 0}^k P_l(p) h^r_0(k - l) \quad \text{for all} \; p, k \geq 0.
\end{equation*}
This yields the following equation for the generating function $\mathcal{P}(x, y) \coloneqq \sum_{p \geq 0} \sum_{l \geq 0} P_l(p) x^l y^p$, which is necessarily convergent for $|x|, |y| < 1$,
\begin{align*}
    \frac{\mathcal{P}(x, y)}{\sqrt{(1 + rx)(1 - x)}} 
    = \frac{\frac{x}{\sqrt{(1 + rx)(1 - x)}} - \frac{y}{\sqrt{(1 + ry)(1 -y)}}}{x - y},
\end{align*}
where $\sum_{p \geq 0} h^r_0(p) x^p = 1/\sqrt{(1 - x)(1 + rx)}$ for $|x| < 1$ was used. From this we conclude,
\begin{align}
\label{diffeq}
    (x\partial_x + y\partial_y)\mathcal{P}(x, y) = y\partial_y \left( \frac{1 + rxy}{\sqrt{(1 - x)(1 + rx)(1 - y)(1 + ry)}} \right).
\end{align}
Note that the expression inside the brackets on the right side is a product of $(1 + rxy)$ and two generating functions of $h^r_0$ evaluated at $x, y$. Therefore, comparing the coefficients of the series we get $(p + l)P_l(p) = p\left(h^r_0(p) h^r_0(l) + rh^r_0(p - 1) h^r_0(l - 1) \right)$ for $l, p \geq 0$. 
\end{proof}
In the course of the proof we have established that $H^r_l(p)$ has the generating function 
\begin{align}
    \label{H_gen_fct}
    \sum_{p \geq 0} \sum_{l \geq 0} H^r_l(p) x^l y^p = \frac{x - y\sqrt{\frac{(1 - x)(1 + rx)}{(1 - y)(1 + ry)}}}{x - y} \quad \text{for} \; |x|, |y| \leq 1, y \neq 1.
\end{align}
which is to be understood as $\frac{1}{2} \left( \frac{1}{1 - x} + \frac{1}{1 + rx}\right)$ when $x = y$.

\subsection{Accelerated ricocheted random walk}
\label{subsec:ARRW}
For an arbitrary probability measure $\nu$ on $\mathbb{Z}$ and a constant $\mathfrak{p} \in [0, 1]$ we define the \emph{accelerated $\mathfrak{p}$-ricocheted random walk (ARRW)} to be a Markov process $(W^*_i, N^*_i)_{i \geq 0}$ on $\mathbb{Z}^2$ started at $(k, 0)$ under $\mathbb{P}_k^*$ with the transition probabilities 
\begin{align}
    \label{ricRW_trans}
    \mathbb{P}^*_k [W^*_{i + 1} = l, \; N^*_{i + 1} = N^*_i + 1 |\; W^*_i = p] &= \mathfrak{p} \sum_{v \geq 1} \binom{v + l - 1}{l} \frac{\nu(-p - v)}{2^{l + v}}, \quad p>0, l\geq 0; \nonumber\\
    \mathbb{P}^*_k [W^*_{i + 1} = l, \; N^*_{i + 1} = N^*_i |\; W^*_i = p] &= \left(1 - \mathfrak{p}\mathds{1}_{\{l < 0\}} \right) \nu(l - p), \quad p>0,l \in \mathbb{Z};\\
    \mathbb{P}^*_k [W^*_{i + 1} = p, \; N^*_{i + 1} = N^*_i |\; W^*_i = p] &= 1, \quad p \leq 0. \nonumber
\end{align}

One should think of this process as a random walk with distribution $\nu$ on $\mathbb{N}$ with a specific boundary condition. Whenever it is leaving $\mathbb{N}$ and jumping to $-k < 0$: with independent probability $1 - \mathfrak{p}$ it ``penetrates" the ``wall" $\mathbb{Z}_{\leq 0}$ and gets trapped at $-k$ forever, or with independent probability $\mathfrak{p}$ it ``ricochets off the wall with some random acceleration" and lands at some $l \in \mathbb{Z}_+$, once the walk hits $0$ it is trapped there almost surely. The process $(N^*_i)$ counts the number of ricochets that have occurred up to time $i$. If the process is trapped eventually, its final position is denoted by $(W^*_\infty, N^*_\infty)$. 

Unlike the $\mathfrak{p}$-ricocheted walk introduced in \cite{TB18} this process is not very illustrative since the landing position after ricochet does not determine the ``intentional" (if there were no ``wall") landing position. Rather with a positive probability it can be any negative integer. The reason for changing the process in such a way becomes understandable once we compare it with the peeling process discussed in Section \ref{subsec:target_exploration}. 

To any accelerated ricocheted walk $(W^*_i, N^*_i)_{i \geq 0}$ we can associate the \emph{accelerated ricochet sequence} $(\ell^*_n)_{n \geq 0}$ that tracks the landing point after ricochet. More precisely, $\ell^*_0 = W^*_0$ and for $1 \leq n \leq N^*_\infty$: $\ell^*_n = W^*_i$ if the $n$th accelerated ricochet occurs at the $i$th step, i.e. $N^*_{i - 1} = n - 1$ and $N^*_{i} = n$. We set $\ell^*_n = W^*_\infty$ if $n > N^*_\infty$. Then it is a Markov process with transition probabilities
\begin{equation*}
    \begin{aligned}
        \mathbb{P}^*_k[\ell^*_{n + 1} = l |\; \ell^*_n = p] &= 
        \mathfrak{p} \sum_{n \geq 1} \binom{n + l - 1}{l} \left(\frac{1}{2}\right)^{n + l} \mathbb{P}_p[(W_i) \; \text{hits} \; \mathbb{Z}_{\leq 0} \; \text{at} \; -n], \\
        \mathbb{P}^*_k[\ell^*_{n + 1} = 0 |\; \ell^*_n = p] &= 
        \mathbb{P}_p[(W_i) \; \text{hits} \; \mathbb{Z}_{\leq 0} \; \text{at} \; 0] + 
        \mathfrak{p} \sum_{n \geq 1} \left(\frac{1}{2}\right)^{n} \mathbb{P}_p[(W_i) \; \text{hits} \; \mathbb{Z}_{\leq 0} \; \text{at} \; -n], \\
        \mathbb{P}^*_k[\ell^*_{n + 1} = k |\; \ell^*_n = p] &= 
        (1 - \mathfrak{p}) \mathbb{P}_p[(W_i) \; \text{hits} \; \mathbb{Z}_{\leq 0} \; \text{at} \; k], 
    \end{aligned}
\end{equation*}
where $p > 0, l > 0, k < 0$ and $(W_i)$ under $\mathbb{P}_p$ is a random walk started at $p$ with law $\nu$. For $p \leq 0, \; \mathbb{P}^*_k[\ell^*_{n + 1} = p |\; \ell^*_n = p] = 1$.
For an admissible measure $\nu$ we can rewrite the above probabilities via Lemma \ref{lemma:H-fct} for $p > 0, l \in \mathbb{Z}$ as follows
\begin{equation}
    \label{ric_seq}
    \begin{split}
    \phantom{\mathbb{P}^*_k[\ell^*_{n + 1} = l |\; \ell^*_n = p]}
    &\begin{aligned}
        \mathllap{\mathbb{P}^*_k[\ell^*_{n + 1} = l |\; \ell^*_n = p]} =
        (1 &- \mathfrak{p}) H^r_{|l|}(p) \mathds{1}_{\{ l < 0 \}} + H^r_0(p) \mathds{1}_{\{ l = 0 \}} \\
        &+ \mathfrak{p} \sum_{n \geq 1} \binom{n + l - 1}{l} \left( \frac{1}{2} \right)^{n + l} H_n^r(p) \mathds{1}_{\{ l \geq 0 \}},
    \end{aligned} \\
    &\mathllap{\mathbb{P}^*_k[\ell^*_{n + 1} = p |\; \ell^*_n = p]} = 1 \quad \text{for} \; p \leq 0. 
    \end{split} 
\end{equation}
\begin{lemma}
\label{lemma:trap}
    If $\nu$ is admissible and $\mathfrak{p} \in [0, 1)$, then the accelerated $\mathfrak{p}$-ricocheted random walk almost surely gets trapped in finite time.
\end{lemma}
Consistent with the previous statements Theorem \ref{thm:nesting} and Theorem \ref{thm:scallim} we exclude the case $n = 2$, which corresponds to $\mathfrak{p} = 1$. 
\begin{proof}
    Let $(W_i)$ under $\mathbb{P}_p$ be a random walk started at $p$ with law $\nu$. When $\mathfrak{p} < 1$, as $(W_i)_i$ does not drift to $\infty$ (see \cite[Proposition 4]{TB16} and Proposition \ref{prop:nu_admissible}) and every time it is about to jump to $\mathbb{Z}_{\leq 0}$ it is trapped with probability at least $1 - \mathfrak{p}$. So the probability of $N$ ricochets happening in a row is $\leq \mathfrak{p}^N$. This yields directly that for $\mathfrak{p} < 1$, the ARRW gets trapped almost surely in finite time. 
\end{proof}

We now define the family of functions which we need for Theorem \ref{thm:nesting} and have already mentioned in the introduction. 
\begin{definition}
\label{def:h^r_fct}
    For $\mathfrak{p} \in [0, 1]$ and admissible $\nu$ we define $h^r_\mathfrak{p}: \mathbb{Z}_{\geq 0} \rightarrow \mathbb{R}$ by setting
    \begin{align*}
        h^r_\mathfrak{p}(p) &\coloneqq \mathbb{P}^*_p \left[ W^*_\infty = 0\right]  \quad \text{for} \; \mathfrak{p} < 1, \\
        h^r_1(p) &\coloneqq 
        \begin{cases}
            \mathbb{P}^*_p \left[ W^*_\infty = 0 |\; T < \infty\right]  &\text{if} \; \mathbb{P}^*_p \left[T < \infty\right] > 0, \\
            1 & \text{otherwise},
        \end{cases} 
    \end{align*} 
    where $(W^*_i)_i$ is the $\mathfrak{p}$-ARRW induced by $\nu$, $T$ is the first hitting time of $\mathbb{Z}_{\leq 0}$ by the $1$-ARRW and $r$ as in Proposition \ref{prop:nu_admissible}.
\end{definition} \vspace{-0.4cm}
We would like to mention that $h^r_\mathfrak{p}(\cdot)$ is well-defined since we will show that it indeed depends only on $r$ and $\mathfrak{p}$. More precisely, this will follow from \eqref{h_expans} in the proof of the next proposition and the explicit expression \eqref{H-fct} for $H^r_\cdot$, which depends only on $r$.
Moreover, this definition is consistent with Lemma \ref{lemma:H-fct} for $\mathfrak{p} = 0$ since in this case $W^*_i = W_i$ with jump law $\nu$ and thus $\mathbb{P}^*_p \left[W^*_\infty = 0\right] = \mathbb{P}_p \left[(W_i) \; \text{hits} \; \mathbb{Z}_{\leq 0} \; \text{at} \; 0\right]$. It is also consistent with (\ref{h-fct}) for $\mathfrak{p} = 1$ since if $\mathbb{P}^*_p \left[T < \infty\right] > 0$, $\mathbb{P}^*_p \left[ W^*_\infty = 0, T < \infty\right] = \mathbb{P}^*_p \left[ T < \infty\right]$ as in this situation $(W^*_i)_i$ can only get trapped at $0$; and thus, $h^r_1(p) = \mathbb{P}^*_p \left[ W^*_\infty = 0 |\; T < \infty\right] = 1$. We further explore the properties of $h^r_\mathfrak{p}$.
\begin{proposition}
    \label{prop:h-props}
    If $\nu$ is admissible such that $\nu(k) > 0$ for an odd $k \geq -1$\footnote{This assumption corresponds to the assumption of $\mathbf{\hat{q}}$ being non-bipartite. Thus, to $r \in (-1,1)$.} and $\mathfrak{p} \in [0, 1]$, then the function $h^r_\mathfrak{p}$ satisfies the following properties:
    \setlist{nolistsep}
    \begin{enumerate}[label = (\roman*), itemsep = 0.5em]
        \item $\mathfrak{p} \mapsto h^r_\mathfrak{p}(p)$ is analytic on $(0, 1)$ and continuous on $[0, 1)$ for all $p \geq 0$;
        \item $h^r_\mathfrak{p}(p) \in (0, 1]$ for all $\mathfrak{p} \in [0, 1]$ and $p \geq 0$;
        \item if $\nu \equiv \nu_{\mathbf{\hat{q}}}$ as in Proposition \ref{prop:nu_admissible} with $\mathbf{\hat{q}}$ corresponding to a non-generic critical triplet $(\mathbf{q}, g, n)$ via \eqref{fpe} and {\bfseries Assumption} holds, then for $\mathfrak{p} \in [0, 1)$ the following asymptotics hold uniformly in $\mathfrak{p}$ in any compact subinterval
        \begin{align*}
            h^r_\mathfrak{p}(p) \sim C_\mathfrak{p} p^{-b} \quad \text{with} \; C_\mathfrak{p} > 0 \; \text{independent of} \; p,
        \end{align*}
        where $b = \arccos(\mathfrak{p}) / \pi, \; C_\mathfrak{p} \coloneqq \frac{8 \sqrt{2}}{\pi} \frac{u \sqrt{1 - \mathfrak{p}} \Gamma(b) \sin\left( b \arccos(\sqrt{u}) \right)}{\sqrt{1 - u} (8u(1 + r))^b}$ and $u = u(\mathfrak{p}) \in (0, 1)$ for all $\mathfrak{p} \in [0, 1)$ is a continuous function such that $u(0) = \frac{17 + \sqrt{33}}{128}$.
    \end{enumerate}
\end{proposition}
\begin{proof}
The fact that $h^r_\mathfrak{p}(p) \in (0, 1]$ for $\mathfrak{p} < 1$ follows easily from the definition and explicit law of ARRW (\ref{ricRW_trans}) together with Lemma \ref{lemma:trap}. 

If $\mathfrak{p} = 0$, then by Lemma \ref{lemma:H-fct} we have $\mathbb{P}^*_{p} \left[ W^*_\infty = 0\right] = H^r_0(p)$, which agrees with $h^r_0(p)$ given by (\ref{h-fct}). Let us now consider $\mathfrak{p} \in (0, 1)$ in order to show that the function $h^r_\mathfrak{p}$ is analytic in $\mathfrak{p}$ on this interval. \\
We introduce the following notation for the transition probability (\ref{ric_seq}), for $p, l > 0$,
\begin{align*}
    \mathfrak{p} \: p^*_{p, l} &\coloneqq \mathbb{P}^*_k[\ell^*_{n + 1} = l |\; \ell^*_n = p] = \mathfrak{p} \sum_{n \geq 1} \binom{n + l - 1}{l} \left(\frac{1}{2}\right)^{n + l} H^r_n(p); \\
    p^*_{p} &\coloneqq \sum_{n \geq 1} \left(\frac{1}{2}\right)^{n} H^r_n(p) = p^*_{p, 0}.
\end{align*}
Exploring the number of ricochets until the ARRW gets trapped via (\ref{ric_seq}) and our new notation we can express the probability $h^r_\mathfrak{p}(p)$ as
\begin{align}
\label{h_expans}
    h^r_\mathfrak{p} (p) 
    &= H^r_0(p) + 
    \sum_{N \geq 1} \mathfrak{p}^N \sum_{(p_1, \ldots, p_{N - 1}) \in \mathbb{Z}_{>0}^{N - 1}} p^*_{p, p_1} p^*_{p_1, p_2} \ldots p^*_{p_{N - 2}, p_{N - 1}} \sum_{p_N \geq 0} p^*_{p_{N - 1}, p_N} H^r_0(p_N). 
\end{align}
Moreover, for any $k > 0$ we have
\begin{align*}
    \sum_{l \geq 0} p^*_{k, l} \overset{\text{Fubini}}{=} \sum_{n \geq 1} H^r_n(k) \underbrace{\sum_{l \geq 0} \binom{n + l - 1}{l} \left(\frac{1}{2}\right)^{n + l}}_{= 1} 
    \overset{\text{Lemma \ref{lemma:H-fct}}}{\leq} \mathbb{P}_k [(W_i) \; \text{hits} \;\; \mathbb{Z}_{\leq 0}] \leq 1.
\end{align*}
Therefore, since $H^r_0 \leq 1$, we inductively conclude that 
\begin{align*}
    \sum_{(p_1, \ldots, p_{N - 1}) \in \mathbb{Z}_{>0}^{N - 1}} &p^*_{p, p_1} p^*_{p_1, p_2} \ldots p^*_{p_{N - 2}, p_{N - 1}} \sum_{p_N \geq 0} p^*_{p_{N - 1}, p_N} H^r_0(p_N) \leq 1.
\end{align*}
This means that the coefficients in the expansion (\ref{h_expans}) of $h^r_\mathfrak{p}$ in front of $\mathfrak{p}^N$ are uniformly bounded by $1$ and independent of $\mathfrak{p}$, which implies that $h^r_\mathfrak{p}$ is analytic in $\mathfrak{p} \in (0, 1)$ and by monotonicity we see that pointwise $h^r_\mathfrak{p} \rightarrow h^r_0$ as $\mathfrak{p} \rightarrow 0$. The proof of $(\romannum{1})$ is complete. 

For $\mathfrak{p} \in (0, 1)$ by the first step analysis, i.e. conditioning the probability $\mathbb{P}^*_p [W_\infty^* = 0]$ on the first step of the ARRW $W^*_1$, we obtain the following recursive equation satisfied by $h^r_\mathfrak{p}$ for $p > 0$,
\begin{align}
    \label{recur_hp}
    h^r_\mathfrak{p}(p) = \sum_{l \geq 0} h^r_\mathfrak{p}(l) \left(\nu(l - p) + \mathfrak{p}\sum_{k \geq 1} \binom{k + l - 1}{l} \left(\frac{1}{2}\right)^{k + l} \nu(-k - p) \right).
\end{align}
We recall (\ref{rec_eq_F_s.gamma+}) and note that by setting $s \coloneqq 2\mathfrak{p}/n \in \left(0, \frac{2}{n} \right)$, if $(\mathbf{q}, g, n)$ is non-generic critical, we get that $h^r_\mathfrak{p}(p)$ and $F_\bullet^{(p)}[s](\mathbf{q}, g, n) \gamma_+^{-p}(\mathbf{\hat{q}})$ satisfy the same recursive equation. Moreover, from these equations we see that both functions $(p, m) \mapsto h^r_\mathfrak{p}(p) \mathds{1}_{\{p \geq 0\}}$ and $(p, m) \mapsto F_\bullet^{(p)}[s] \gamma_+^{-p} \mathds{1}_{\{p \geq 0\}}$ are harmonic w.r.t. the law of the $\mathfrak{p}$-ARRW. In particular, by the uniqueness property of harmonic functions we deduce that $F_\bullet^{(p)}[s] \gamma_+^{-p} = h^r_\mathfrak{p}(p)$ for any $p \geq 0$ and any non-generic critical triplet $(\mathbf{q}, g, n)$. Note that we have used that both $F_\bullet^{(p)}[s] \gamma_+^{-p} \mathds{1}_{\{p \geq 0\}}$ and $h^r_\mathfrak{p}(p) \mathds{1}_{\{p \geq 0\}}$ are equal to $1$ at $p=0$. Indeed, $h^r_\mathfrak{p}(0) = 1$ since by definition the ARRW gets trapped in $\mathbb{Z}_{\leq 0}$ almost surely; $F_\bullet^{(0)}[s] = F^{(0)}$ since no loops are allowed to traverse the root face and $F^{(0)}=1$ by convention. Furthermore, since $\gamma_+$ and $h^r_\mathfrak{p}(p)$ are well-defined and finite for any admissible $\nu$ (or equivalently by Proposition \ref{prop:nu_admissible}, any admissible $\mathbf{\hat{q}}$), we conclude that for any non-generic critical $(\mathbf{q}, g, n)$ related to $\mathbf{\hat{q}}$ by the fixed-point equation (\ref{fpe}), $F_\bullet^{(p)}[s](\mathbf{q}, g, n) = \gamma_+^{p}(\mathbf{\hat{q}}) h^r_\mathfrak{p}(p) < \infty$. \\ 
Hence now we can work with $F_\bullet^{(p)}[s] \gamma_+^{-p}$ instead of $h^r_\mathfrak{p}$ to determine the asymptotic behaviour of the latter function as $p \rightarrow \infty$ in the regime of $(\romannum{3})$. More precisely, we use the result \cite[Theorem 6.7]{BBD18}. We choose there $\xi \coloneqq 8u(\mathfrak{p}) (1 + r)\gamma_+ = 4u(\mathfrak{p}) (1 + r)/g$ for some continuous function $u \coloneqq u(\mathfrak{p}) \in (0, 1)$ and note that since in our situation the bending energy $\alpha = 1$ the shift in (5.26) is automatically zero. Then the statement of \cite[Theorem 6.7]{BBD18} is that the resolvent function $\mathcal{F}_\bullet[s] (x) \coloneqq \sum_{p \geq 0} F_\bullet^{(p)}[s] / x^{p + 1}$ defined analogously to (\ref{resolvent}) has the following form when expanded around $\gamma_+$:
\begin{align}
\label{singul_expans}
    \mathcal{F}_\bullet[s] (x) = \frac{\left(\frac{x - \gamma_+}{\xi}\right)^{b(s) - 1}}{1 - \left(\frac{x - \gamma_+}{\xi}\right)^{2 b(s)}} \Phi_{b(s)} (\xi) + \mathcal{O}\left( (x - \gamma_+)^{b(s)}\right) \quad \text{as} \; x \searrow \gamma_+.
\end{align}
The error is uniform for $u$ in any compact subinterval of $(0, 1)$, which is by definition of $u(\mathfrak{p})$ is equivalent to being uniform for $\mathfrak{p}$ in any compact subinterval of $(0, 1)$.
Here $\xi \coloneqq 8u(1 + r)\gamma_+$, $b(s) \coloneqq \arccos\left(\frac{ns}{2}\right) / \pi = \arccos(\mathfrak{p}) / \pi = b$ as in $(\romannum{3})$, ($h$ in \cite{BBD18} corresponds to our $g$, $\rho = 1 + r$ with $r = -\gamma_-/\gamma_+$ as defined in Section \ref{subsubsec:intro_loop_model}), and 
\begin{align}
\label{phi_coef}
    \Phi_b (\xi) = \frac{2 \sqrt{2} g}{(1 + r) \sqrt{1 + \mathfrak{p}}} \frac{\sin\left( b \arccos(\sqrt{u}) \right)}{\sin\left( \arccos(\sqrt{u}) \right)} = 
    \frac{\sqrt{2} \gamma_+^{-1}}{(1 + r) \sqrt{1 + \mathfrak{p}}} \frac{\sin\left( b \arccos(\sqrt{u}) \right)}{\sqrt{1 - u}}.
\end{align}
In the last equality we have used that $\gamma_+ g = 1/2$. \\
Moreover, since $1/(1 - \varepsilon) \approx 1 + \varepsilon$ as $\varepsilon \rightarrow 0$, we further can simplify (\ref{singul_expans}):
\begin{align}
\label{singul_expans_simpl}
    \mathcal{F}_\bullet[s] (x) = \frac{\Phi_{b}(\xi)}{\xi^{b - 1}} (x - \gamma_+)^{b - 1} + \mathcal{O}\left( (x - \gamma_+)^{b \wedge (3b - 1)}\right) \quad \text{as} \; x \searrow \gamma_+.
\end{align}
By transfer theorems (see \cite[Theorems \Romannum{6}.1 and \Romannum{6}.3]{FS_combi}) it amounts to the asymptotic behaviour,
\begin{align}
\label{coef_asymp}
    F^{(p)}_\bullet[s] \overset{p \rightarrow \infty}{\sim} \frac{\gamma_+^{b - 1}\Phi_{b}(\xi)}{\xi^{b - 1} \Gamma(1 - b)} \frac{\gamma_+^{p + 1}}{(p + 1)^{b}} \sim 
    \underbrace{\frac{8 \sqrt{2}}{\pi} \frac{u \sqrt{1 - \mathfrak{p}} \Gamma(b) \sin\left( b \arccos(\sqrt{u}) \right)}{\sqrt{1 - u} (8u(1 + r))^b}}_{= C_\mathfrak{p}} \frac{\gamma_+^p}{p^{b}},
\end{align}
where the reflection formula $\Gamma(1 - x) \Gamma(x) = \pi/ \sin(\pi x)$, (\ref{phi_coef}) and definitions of $\xi, b$ were used. In the case $\mathfrak{p} = 0$ we can determine $u$ explicitly (see below). \\
We can finally conclude that $h^r_\mathfrak{p} = F^{(p)}_\bullet[s] \gamma_+^{-p}$ has asymptotics 
\begin{equation}
\label{h_asymptotic}
    h^r_\mathfrak{p}(p) \overset{p \rightarrow \infty}{\sim} C_\mathfrak{p} p^{-b}
\end{equation}
uniformly in $\mathfrak{p}$ in any compact subinterval of $(0, 1)$. This completes the proof of $(\romannum{3})$ for $\mathfrak{p} \in (0, 1)$. For $\mathfrak{p} = 0$, (\ref{h_asymptotic}) was obtained in \cite[(23)]{TB16} explicitly using the exact formulas for $h^r_0$ and its resolvent function. More precisely, $C_0 = \Gamma(1/2) / (\pi \sqrt{1 + r})$. This results from trigonometric identities in the following equation for $u = u(0)$,
\begin{align*}
    \sqrt{\frac{8 u (1 - \sqrt{u})}{1 - u}} = 1 \overset{u \in (0, 1)}{\Longleftrightarrow} \frac{8 u}{1 + \sqrt{u}} = 1.
\end{align*}
Hence, $u(0) = \frac{17 + \sqrt{33}}{128}$. By continuity of $\mathfrak{p} \mapsto u(\mathfrak{p})$ and since $u(0) \in (0, 1)$ we conclude that the asymptotics (\ref{h_asymptotic}) holds uniformly in $\mathfrak{p}$ in any compact subinterval of $[0, 1)$. 
\end{proof}

\subsection{Conditioning to be trapped at zero}
The conditioning of the probability $\mathbb{P}^*_p [W_\infty^* = 0]$ on the first step of the accelerated ricocheted random walk leads us to the following equation satisfied by $h^r_\mathfrak{p}$ for $p > 0$,
\begin{align}
    \label{recur_eq_hp}
    h^r_\mathfrak{p}(p) = \sum_{l \geq 0} h^r_\mathfrak{p}(l) \left(\nu(l - p) + \mathfrak{p}\sum_{k \geq 1} \binom{k + l - 1}{l} \left(\frac{1}{2}\right)^{k + l} \nu(-k - p) \right),
\end{align}
(see also the proof of Proposition \ref{prop:h-props} $(\romannum{3})$).
Recall (\ref{rec_eq_F.gamma+}) and note that once $\mathfrak{p} = n/2$ and $g\gamma_+ = 1/2$, i.e. $(\mathbf{q}, g, n)$ is non-generic critical, $h^r_\mathfrak{p}(p)$ and $F_\bullet^{(p)} \gamma_+^{-p}(\mathbf{\hat{q}})$ satisfy the same recursive equation. Then these equations show that both functions $(p, m) \mapsto h^r_\mathfrak{p}(p) \mathds{1}_{\{p \geq 0\}}$ and $(p, m) \mapsto F_\bullet^{(p)} \gamma_+^{-p}(\mathbf{\hat{q}}) \mathds{1}_{\{p \geq 0\}}$ are harmonic w.r.t. the law of the accelerated $\frac{n}{2}$-ricocheted random walk. In particular, by the uniqueness property of harmonic functions (since both functions are equal to $1$ at $p=0$) we can conclude that $F_\bullet^{(p)} \gamma_+^{-p}(\mathbf{\hat{q}}) = h^r_\frac{n}{2}(p)$ for any non-generic critical triplet $(\mathbf{q}, g, n)$. The same result can be deduced from Theorem \ref{thm:admissibility}, see Section \ref{sec:admissibilityPf}.

We now introduce the process, which is the Doob-transform of the ARRW and will reconstruct the law of the peeling exploration in the non-generic setting. Assume that the law of the random walk $(W_i)_i$ on $\mathbb{Z}$, $\nu$, is admissible and that there exists an odd $k \geq -1$ such that $\nu(k) > 0$. Let the \emph{accelerated $\mathfrak{p}$-ricocheted random walk conditioned to be trapped at $0$} be defined as the process $(W^r_i, N^r_i)_{i \geq 0}$ under $\mathbb{P}^r_p$ that has the law of the accelerated $\mathfrak{p}$-ricocheted random walk $(W^*_i, N^*_i)_{i \geq 0}$ conditioned on $\left\{W^*_\infty = 0\right\}$ under $\mathbb{P}^*_p$. So $(W^r_i, N^r_i)$ is the Markov process obtained from $(W^*_i, N^*_i)$ by the $h$-transform w.r.t. the harmonic function $(p, m) \mapsto h^r_\mathfrak{p}(p) \mathds{1}_{\{p \geq 0\}}$. Note that the above conditions on $\nu$ and Proposition \ref{prop:h-props} ensure that $h^r_\mathfrak{p} > 0$. This means that the transition probabilities are given by
\begin{align}
    \label{ricRW_trap_trans}
    \mathbb{P}^r_k [W^r_{i + 1} = l, \; N^r_{i + 1} = N^r_i + 1 |\; W^r_i = p] &= \frac{h^r_\mathfrak{p}(l)}{h^r_\mathfrak{p}(p)} \mathfrak{p} \sum_{v \geq 1} \binom{v + l - 1}{l} \left(\frac{1}{2}\right)^{l + v} \nu(-p - v), \nonumber \\
    \mathbb{P}^r_k [W^r_{i + 1} = l, \; N^r_{i + 1} = N^r_i |\; W^r_i = p] &= \frac{h^r_\mathfrak{p}(l)}{h^r_\mathfrak{p}(p)} \nu(l - p), \\
    \mathbb{P}^r_k [W^r_{i + 1} = 0, \; N^r_{i + 1} = N^r_i |\; W^r_i = 0] &= 1, \nonumber
\end{align}
where $p > 0, l \geq 0$.

The discussion at the beginning of this section yields the desired main result of this chapter. Alternatively, it can be deduced by comparing (\ref{trans_nu}) with (\ref{ricRW_trap_trans}) after application of (\ref{expect_num_vert}), which says that $F_\bullet^{(p)} \gamma_+^{-p} = h^r_{\frac{n}{2}}(p)$. The latter strategy does not require the {\bfseries Assumption} to hold.
\begin{proposition}
    \label{prop:h-trafo}
    If $(\mathbf{q}, g, n)$ is a non-generic critical triplet such that $\mathbf{q}$ is non-bipartite, then the perimeter and nesting process $(P_i, N_i)_{i \geq 0}$ of the pointed $(\mathbf{q}, g, n)$-Boltzmann map of perimeter $p$ with the triangular loop configuration is equal in distribution to the accelerated $\frac{n}{2}$-ricocheted random walk $(W^r_i, N^r_i)_{i \geq 0}$ with law $\nu_{\mathbf{\hat{q}}}$ conditioned to be trapped at $0$ under $\mathbb{P}^r_p$.
\end{proposition}

\subsection{Nesting statistic}
\label{subsec:nest_st}
The goal of this section is to prove our main statement about the nesting statistic, Theorem \ref{thm:nesting}. For this reason, we verify the following result about the total number of ricochets. 
\begin{proposition}[Distribution of number of ricochets]
    \label{prop:ricochet_props} 
    Let $\nu$ be admissible such that there exists an odd $k \geq -1$ with $\nu(k) > 0$, $\mathfrak{p} \in (0, 1)$, and $(W^r_i, N^r_i)_{i \geq 0}$ be the accelerated $\mathfrak{p}$-ricocheted walk of law $\nu$ conditioned to be trapped at $0$. Then the number $N^r_\infty$ of ricochets has probability generating function 
    \begin{align}
    \label{num_ric_distr}
        \mathbb{E}_k^{r} \left[ x^{N^r_\infty} \right] = \frac{h^r_{x\mathfrak{p}}(p)}{h^r_{\mathfrak{p}}(p)}
        \quad \text{for} \;\; x \in \left[ 0, \frac{1}{\mathfrak{p}} \right).
    \end{align}
    Moreover, $\frac{N^r_\infty}{\log k} \xrightarrow{\mathbb{P}_k^r} \frac{\mathfrak{p}}{\pi\sqrt{1 - \mathfrak{p}^2}} \quad \text{as} \;\; k \rightarrow\infty$,
    and the large deviation property holds
    \begin{equation*}
        \begin{aligned}
            \frac{\log \mathbb{P}_k^{r}\left[N^r_\infty < \lambda \log k\right]}{\log k} &\xrightarrow{k \rightarrow \infty} -\frac{1}{\pi} J_{\mathfrak{p}}(\pi\lambda) \quad &&\text{for} \;\; 0 < \lambda < \frac{\mathfrak{p}}{\pi\sqrt{1 - \mathfrak{p}^2}},\\
            \frac{\log \mathbb{P}_k^{r}\left[N^r_\infty > \lambda \log k\right]}{\log k} &\xrightarrow{k \rightarrow \infty} -\frac{1}{\pi} J_{\mathfrak{p}}(\pi\lambda) \quad &&\text{for} \;\; \lambda > \frac{\mathfrak{p}}{\pi\sqrt{1 - \mathfrak{p}^2}},
        \end{aligned}
    \end{equation*}
    where
    \begin{align*}
        J_\mathfrak{p}(x) = x\log\left(\frac{x}{\mathfrak{p}\sqrt{1 + x^2}}\right) + \arccot x - \arccos\mathfrak{p}.
    \end{align*}
\end{proposition}
The proof is omitted since it is analogous to the proof of \cite[Proposition 6]{TB18} once Proposition \ref{prop:h-props} is established. We refer also to \cite{thesis} for details.

Now the key result about the nesting statistic, Theorem \ref{thm:nesting}, follows by applying Proposition \ref{prop:h-trafo} to the outcomes of Proposition \ref{prop:ricochet_props}.

\section{Equivalence of admissibility of weight sequences}
\label{sec:admissibilityPf}
In this chapter we prove the remaining statement of Theorem \ref{thm:admissibility}. Namely we assume that $\hat{\mathbf{q}}$ is admissible and $g > 0, n \in (0, 2]$ are such that for all $k \geq 1$
\begin{align}
\label{q_weight}
    q_k = \hat{q}_k - n\sum_{l \geq 0} \binom{k + l - 1}{l} g^{k + l} W^{(l)}(\hat{\mathbf{q}}) \geq 0.
\end{align}
We exclude cases $n=0$ or $g=0$ since otherwise the statement of the theorem is trivial. We then show that $(\mathbf{q}, g, n)$ is admissible. Note that by excluding both cases $n=0$ or $g=0$ we as well exclude the situation when $\mathbf{\hat{q}}$ is bipartite. Indeed, suppose that $\mathbf{\hat{q}}$ is bipartite, then since $W^{(0)} \equiv 1$ and since $q_k \geq 0$ for all $k\geq 1$, we conclude that either $n$ or $g$ has to be zero. Equivalently, via Proposition \ref{prop:nu_admissible} and (\ref{nu}), we are given $\hat{q}_k = \left(\frac{\nu(-2)}{2} \right)^{\frac{k - 2}{2}} \nu (k - 2)$ for an admissible probability measure $\nu$ that satisfies $\nu(k)>0$ for an odd $k\geq -1$ and
\begin{align}
\label{assmpt_nu}
    \nu (k - 2) \geq \frac{n}{2} \sum_{l \geq 0} \binom{k + l - 1}{l} \left( \frac{2g^2}{\nu(-2)} \right)^{\frac{k + l}{2}} \nu(-l - 2) \quad \text{for} \; k \geq 1.
\end{align}
This implies that $\sqrt{\frac{2g^2}{\nu(-2)}} \leq \frac{1}{2}$, see \ref{A:subsec:ineq_implication}.

The strategy of the proof is chosen to be be exactly as in \cite{TB18}. Namely, we explicitly construct a random loop-decorated map and show that it is distributed as a $(\mathbf{q}, g, n)$-Boltzmann loop-decorated map, where $\mathbf{q}$ is defined by (\ref{q_weight}). Since most of the proof steps follow \cite[Section 5]{TB18}, we omit many details and only highlight the necessary changes. All of the details skipped here can be found in \cite{thesis}. \\
As we have mentioned before in Section \ref{sec:peeling} a loop-decorated map $(\mathfrak{m}, \boldsymbol\ell)$ of perimeter $p$ can be specified by fixing a peeling algorithm $\mathcal{A}$ and providing a valid sequence of peeling events among $\mathrm{G}_{k_1, k_2}, \mathrm{C}_k, \mathrm{L}_{k_1, \ldots, k_m}$.

Let $(\mathfrak{e}_0, \boldsymbol\ell_0)$ be the hollow map of perimeter $p$. We iteratively construct the random sequence
\begin{align}
\label{peel_seq}
    (\mathfrak{e}_0, \boldsymbol\ell_0) \subset (\mathfrak{e}_1, \boldsymbol\ell_1) \subset (\mathfrak{e}_2, \boldsymbol\ell_2) \subset \ldots 
\end{align}
by setting $(\mathfrak{e}_{i+1}, \boldsymbol\ell_{i+1}) = (\mathfrak{e}_{i}, \boldsymbol\ell_{i})$ if $\mathfrak{e}_i$ has no holes and otherwise $(\mathfrak{e}_{i+1}, \boldsymbol\ell_{i+1})$ is obtained from $(\mathfrak{e}_{i}, \boldsymbol\ell_{i})$ by one of the peeling events. If the peeled edge $\mathcal{A}(\mathfrak{e}_i)$ is incident to a hole of degree $l$, the event $\mathrm{E} \in \{ \mathrm{G}_{k_1, k_2}, \mathrm{C}_k, \mathrm{L}_{k_1, \ldots, k_m}\}$ is chosen independently of everything else with probability $P_l(\mathrm{E})$, which we choose, motivated by (\ref{trans_prob}), to be 
\begin{align}
\label{reverse_trans_prob}
    P_l[\mathrm{L}_{\texttt{out}, k_2, \ldots, k_m}] &= \frac{n}{2} \left( \frac{2g^2}{\nu(-2)}\right)^{\frac{m}{2}} \frac{\nu(-l - v) \nu(-m + v - 2)}{\nu(-l - 2)} && \text{for} \; v = |\{k_i = \texttt{out}\}|, \; m \geq 1,  \nonumber\\
    P_l[\mathrm{C}_{k}] &= q_k \left( \frac{2}{\nu(-2)} \right)^{\frac{k - 2}{2}} \frac{\nu(-l - k)}{\nu(-l - 2)} && \text{for} \; k \geq 1, \\
    P_l[\mathrm{G}_{k_1, k_2}] &= \frac{\nu(-k_1 - 2) \nu(-k_2 - 2)}{2\nu(-l - 2)} && \text{for} \; k_1, k_2 \geq 0, \; k_1 + k_2 + 2 = l. \nonumber
\end{align}
\begin{flalign*}
    &\text{Here,} \quad q_k \left(\frac{2}{\nu(-2)}\right)^{\frac{k - 2}{2}} =  \nu(k - 2) - \frac{n}{2} \sum_{l \geq 0} \binom{k + l - 1}{l} \left(\frac{2g^2}{\nu(-2)}\right)^\frac{k + l}{2} \nu(-l - 2).&
\end{flalign*}
These expressions are non-negative by (\ref{q_weight}) and sum up to one. Indeed, 
\begin{align*}
    P_l[\mathrm{C}_{k}] &+ \sum_{u \geq 0} P_l[\mathrm{L}_k^u]  = \frac{\nu(-l - k) \nu(k - 2)}{\nu(-l - 2)} \quad \text{is independent of} \; n
\end{align*}
and for $n = 0$ the loop equation \cite[(9)]{TB16} delivers the result. The events with non-vanishing probability are always valid, thus (\ref{peel_seq}) is a well-defined random infinite sequence of loop-decorated maps with holes.
\begin{proposition}
    \label{prop:stabil_seq}
    The sequence (\ref{peel_seq}) stabilizes almost surely and gives rise to a well-defined random loop-decorated map $(\mathfrak{m}, \boldsymbol\ell)$ of perimeter $p$.
\end{proposition}
The proof of this result follows from the next lemma as in \cite{TB18} with the only difference that the number of inner vertices increases not only in the events $\mathrm{G}_{0, l - 2}, \mathrm{G}_{l - 2, 0}$ when $l>1$, but also $\mathrm{L}_{\texttt{out}, \ldots, \texttt{out}}$ if $l > 1$ and $\mathrm{L}_{\texttt{out}, \texttt{in}, \ldots, \texttt{in}}, \mathrm{C}_1$ if $l = 1$. This, however, does not destroy the argument used to conclude the proof since $P_1[C_1] > 0$ and $P_l[\mathrm{G}_{0, l - 2}]$ are uniformly bounded away from $0$ also in our setting. The proof of the latter can be found in \cite[A.7]{thesis}.

For a loop-decorated planar map $(\mathfrak{e}, \boldsymbol\ell)$ with holes we set
\begin{align}
    \mathrm{V}(\mathfrak{e}, \boldsymbol\ell) = |\mathfrak{e}| + \sum_h f^r (\mathrm{deg}(h)) \quad \text{for} \;\; f^r(p) \coloneqq \frac{\nu(-2) h^r_\mathfrak{p}(p)}{\nu(-p - 2)},
    \label{vertice_fct}
\end{align}
where $|\mathfrak{e}|$ is the number of inner (not incident to a hole) vertices of $\mathfrak{e}$, $\mathfrak{p} = \frac{n}{2}$, $r$ as in Proposition \ref{prop:nu_admissible} and the sum is over the holes $h$ of $\mathfrak{e}$.
\begin{lemma}
\label{lemma:vert_expect}
    \begin{align*}
        \mathbb{E}\left[ \mathrm{V}(\mathfrak{e}_{i + 1}, \boldsymbol\ell_{i + 1}) |\; (\mathfrak{e}_i, \boldsymbol\ell_i) \right] \leq \mathrm{V}(\mathfrak{e}_i, \boldsymbol\ell_i)
    \end{align*}
    with equality if $g = \frac{1}{2}\sqrt{\frac{\nu(-2)}{2}}$. In particular, $\mathrm{V}$ is superharmonic for (\ref{peel_seq}).
\end{lemma}
The only changes that have to be done to the proof in \cite{TB18} are the following. We need to use $f^r$ defined by (\ref{vertice_fct}) instead of $f^\downarrow$, then in the event of peeling a face intersected by a loop $\mathrm{L}_{k_1, \ldots, k_{l + l'}} \in \mathrm{L}^{l'}_l$, $\mathrm{V}(\mathfrak{e}_{i + 1}, \boldsymbol\ell_{i + 1}) - \mathrm{V}(\mathfrak{e}_i, \boldsymbol\ell_i) = -f^r(p) + f^r(p + l - 2) + f^r(l')$. The difference of \cite[(38)]{TB18} to the explicit probabilities (\ref{reverse_trans_prob}) is compensated by the corresponding modified (\ref{recur_eq_hp}) and the fact that $\sqrt{\frac{2g^2}{\nu(-2)}} \leq \frac{1}{2}$ holds.

The following lemma provides the last component to complete the proof of Theorem \ref{thm:admissibility}.
\begin{lemma}
\label{lemma:Boltzm_distr}
    The probability of getting any particular loop-decorated map $(\mathfrak{m}, \boldsymbol\ell)$ of perimeter $p$ in the described way is $\mathbf{w}_{\mathbf{q}, g, n}(\mathfrak{m}, \boldsymbol\ell) / W^{(p)}(\hat{\mathbf{q}})$.
\end{lemma}
The proof of this result is a straightforward modification of the corresponding proof in \cite{TB18} using (\ref{reverse_trans_prob}) and the aforementioned fact that the number of inner vertices increases not only in the events $\mathrm{G}_{0, l - 2}, \mathrm{G}_{l - 2, 0}$ when $l>1$, but also $\mathrm{L}_{\texttt{out}, \ldots, \texttt{out}}$ if $l > 1$ and $\mathrm{L}_{\texttt{out}, \texttt{in}, \ldots, \texttt{in}}, \mathrm{C}_1$ if $l = 1$.

This gives us all the ingredients to prove Theorem \ref{thm:admissibility}: Proposition \ref{prop:stabil_seq} and Lemma \ref{lemma:Boltzm_distr} together imply that 
\begin{align*}
    F^{(p)}(\mathbf{q}, g, n) \coloneqq \sum_{(\mathfrak{m}, \boldsymbol\ell) \in \mathcal{LM}^{(p)}} \mathbf{w}_{\mathbf{q}, g, n}(\mathfrak{m}, \boldsymbol\ell) = W^{(p)}(\hat{\mathbf{q}}).
\end{align*}
This shows that $F^{(p)}(\mathbf{q}, g, n) < \infty$, and, thus, $(\mathbf{q}, g, n)$ is admissible and the random loop-decorated map $(\mathfrak{m}, \boldsymbol\ell)$ is distributed as a $(\mathbf{q}, g, n)$-Boltzmann loop-decorated map of perimeter $p$. Furthermore, from Lemma \ref{lemma:vert_expect} we deduce that the expected number of vertices of $\mathfrak{m}$ is at most $f^r(p)$. This completes the proof of Theorem \ref{thm:admissibility}. From $F_\bullet^{(p)}(\mathbf{q}, g, n) = F^{(p)}(\mathbf{q}, g, n) \mathbb{E}^{(p)}\left[ |\mathfrak{m}| \right]$ we deduce the following statement.
\begin{corollary}
    If $(\mathbf{q}, g, n)$ is admissible and $n \in [0, 2]$, then $(\mathbf{q}, g, n)$ is strongly admissible.
\end{corollary}

Moreover, if $(\mathbf{q}, g, n)$ is non-generic critical, i.e $g\gamma_+(\hat{\mathbf{q}}) = g\sqrt{2/\nu(-2)} = 1/2$, then it follows from Lemma \ref{lemma:vert_expect} that $f^r(p)$ is exactly the expected number of vertices of the $(\mathbf{q}, g, n)$-Boltzmann loop-decorated map. In this case via (\ref{nu}) we get
\begin{align}
\label{expect_num_vert}
    F_\bullet^{(p)}(\mathbf{q}, g, n) = f^r(p) F^{(p)}(\mathbf{q}, g, n) = \gamma_+^p(\hat{\mathbf{q}}) h^r_\frac{n}{2}(p).
\end{align}
And hence by (\ref{asymptotic}) and Proposition \ref{prop:h-props} $(\romannum{3})$ we obtain the following result.
\begin{proposition}
    If {\bfseries Assumption} holds and $(\mathbf{q}, g, n) \in \mathcal{D}$ is non-generic critical (dilute or dense) with exponent $a$, then the number $|\mathfrak{m}|$ of vertices of a $(\mathbf{q}, g, n)$-Boltzmann loop-decorated map $(\mathfrak{m}, \boldsymbol\ell)$ of perimeter $p$ has expectation value 
    \begin{align}
        \mathbb{E}^{(p)}\left[ |\mathfrak{m}| \right] = \frac{\gamma_+^p(\hat{\mathbf{q}}) h^r_\frac{n}{2}(p)}{F^{(p)}(\mathbf{q}, g, n)} \overset{p \rightarrow \infty}{\sim} Cp^{\max(2, 2a - 2)} \quad \text{for some} \;\: C>0.
    \end{align}
\end{proposition}
We recall that when considering triples in $\mathcal{D}$ we exclude $g=0$ and $n=0$, and, thus, as we have seen at the beginning of this section the case of bipartite weight sequences. Therefore, Proposition \ref{prop:h-props} is indeed applicable.

\section{Ricocheted stable processes}
\label{sec:scalLim}
The goal of this section is to investigate the convergence of the perimeter process $(P_i)_i$ after appropriate rescaling and to prove the main result, Theorem \ref{thm:scallim}. Before we start with the proof, we need to understand the limiting process better. As mentioned above, we will show that this limit coincides with the one introduced in \cite{TB18}, so in the following section we recall most of the results and definitions from \cite[Section 6.1]{TB18}.

\subsection{Definition and some properties of ricocheted stable processes}
\label{subsec:RSP}
Consider $\;\mathbb{S} = \left\{ (\theta, \rho): \theta \in (0, 1), \rho \in (0, 1) \right\} \cup \left\{ (\theta, \rho): \theta \in (1, 2), \rho \in (1 - 1/\theta, 1/\theta) \right\}$.
For $(\theta, \rho) \in \mathbb{S} $ let $(S_t)_{t \geq 0}$ be the (strictly) $\theta$-stable L\'{e}vy process started at $0$ with positivity parameter $\rho = \mathbb{P} [S_1 \geq 0]$.  
More precisely, $(S_t)_{t \geq 0}$ is characterized by $\mathbb{E} \left[ \exp{(\iu \lambda S_t)}\right] = \exp{(-t \psi(\lambda))}$ with characteristic exponent in the L\'evy-Khintchine form 
\begin{align*}
    \psi(\lambda) = \iu a_0 \lambda + \int_{\mathbb{R}} \left( 1 - \expo^{\iu \lambda y} + \iu \lambda y \mathds{1}_{\{|y| < 1\}}\right) \pi(y) \mathrm{d}y 
\end{align*}
with L\'evy measure 
$\; \pi(y) = c_+ y^{-\theta - 1} \mathds{1}_{\{y > 0\}} + c_- |y|^{-\theta - 1} \mathds{1}_{\{y < 0\}}$,
where 
\begin{align}
\label{Levy_coef}
    a_0 = \frac{c_+ - c_-}{\theta - 1}, \quad c_+ = \Gamma(1 + \theta) \frac{\sin(\pi \theta \rho)}{\pi}, \quad 
    c_- = \Gamma(1 + \theta) \frac{\sin(\pi \theta (1 - \rho))}{\pi}.
\end{align}
The infinitesimal generator $\mathcal{A}$ corresponding to $(S_t)$ acting on twice differentiable measurable functions $f: \mathbb{R} \rightarrow \mathbb{R}$ defined as $\mathcal{A}f(x) \coloneqq \lim_{t \rightarrow 0} \frac{1}{t} \mathbb{E}[f(x + S_t) - f(x)]$ is given by (see e.g. \cite[Chapter 6]{Schil_book})
\begin{align}
\label{A-gener}
    \mathcal{A}f(x) = -a_0 f'(x) + \int_{\mathbb{R}} \left( f(x + y) - f(x) - y f'(x) \mathds{1}_{\{|y| < 1\}}\right) \pi(y) \mathrm{d}y.
\end{align}
\par Following \cite[Section 6.1]{TB18}, for $\mathfrak{p} \in [0, 1]$ we define the $\mathfrak{p}$\emph{-ricocheted stable process} $(X^*_t)_{t \geq 0}$ to be the c\`adl\`ag Markov process with the infinitesimal generator $\mathcal{A}^*$ acting on twice-differentiable measurable functions $f: (0, \infty) \rightarrow \mathbb{R}$ given in terms of (\ref{A-gener}) for $y > 0, x \in \mathbb{R}$ by 
\begin{equation}
\label{*-gener}
    \mathcal{A}^*f(y) = \mathcal{A}f^*(y), \quad \text{with} \quad f^*(x) \coloneqq f(x)\mathds{1}_{\{x \geq 0\}} + \mathfrak{p} f(-x)\mathds{1}_{\{x < 0\}}.
\end{equation}
Its explicit expression was derived in \cite[Lemma 6]{TB18}. We cite their result here. 
\begin{lemma}
\label{lemma:*gener}
    The infinitesimal generator $\mathcal{A}^*$ defined by (\ref{*-gener}) can be rewritten as 
    \begin{align}
    \begin{split}
    \label{expl*-gener}
        \mathcal{A}^*f(x) = &\; x^{-\theta} \int_{\mathbb{R}_+} \left( f(ux) - f(x) - xf'(x) (u - 1)\mathds{1}_{\{|u - 1| < 1\}}\right) \mu(u) \mathrm{d}u \\
        &+ (\mathfrak{p} - 1)c_- \theta^{-1} x^{-\theta} f(x) + (\mathfrak{p}c_- \kappa - a_0) x^{1 - \theta} f'(x),
    \end{split}
    \end{align}
    where $\;\mu(u) \coloneqq \pi(u - 1) + \mathfrak{p} \pi(-u - 1), \quad \kappa \coloneqq \int_0^2 (u - 1)(u + 1)^{-\theta - 1} \mathrm{d}u,$.
\end{lemma}

$(X^*_t)_{t \geq 0}$ is an example of a \emph{positive self-similar Markov process} (pssMp) with index $\theta$. 
The result by Lamperti \cite{Lmp72} (reformulated in \cite{CC06} with notation closer to ours) shows that any such pssMp $(X_t)_{t \geq 0}$ may be written as the exponential of a time-change of a (killed) L\'evy process $(\xi_s)_{s \geq 0}$ started at $\xi_0 = 0$. More precisely, let $X_0 = x$ for some $x > 0$ and let $T_0 \coloneqq \inf\{t > 0: X_t = 0\}$ (possibly infinite), then
\begin{align*}
    X_t = x \exp\left({\xi_{s(tx^{-\theta})}}\right) \quad \text{with} \quad s(t) \coloneqq \inf\left\{ s \geq 0: \int_0^s \expo^{\theta \xi_u}\mathrm{d}u \geq t \right\}, \; 0 \leq t \leq T_0.
\end{align*}
Such a L\'evy process $(\xi_s)_{s \geq 0}$ is called the \emph{Lamperti representation} of the pssMp $(X_t)_{t \geq 0}$.
The following result is cited from \cite[Proposition 11]{TB18}.
\begin{proposition}
\label{prop:Lamperti_repr}
    The Lamperti representation $(\xi^*_s)_{s \geq 0}$ of $(X^*_t)_{t \geq 0}$ is given by the sum of 
    \begin{enumerate}[label = (\roman*), itemsep = 0.5em]
        \item the Lamperti representation $(\hat{\xi}_s)_{s \geq 0}$ of the $\theta$-stable process killed upon hitting $(-\infty, 0)$ with Laplace exponent
        \begin{equation}
        \label{killed_lamp}
            \hat{\Psi}(z) \coloneqq \log \mathbb{E}\left[ \expo^{z\hat{\xi}_1} \mathds{1}_{\{1 < \zeta(\hat{\xi})\}}\right] = \frac{1}{\pi} \Gamma(\theta - z) \Gamma(1 + z) \sin(\pi z - \pi\theta (1 - \rho));
        \end{equation}
        \item a compound Poisson process $(\xi^\diamond_s)_{s \geq 0}$ with rate $\mathfrak{p} \Gamma(\theta) \sin(\pi\theta (1 - \rho)) / \pi$ and Laplace exponent 
        \begin{equation}
        \label{cpp_lamp}
            \Psi^\diamond(z) \coloneqq \log \mathbb{E}\left[ \expo^{z\xi^\diamond_1} \mathds{1}_{\{1 < \zeta(\xi^\diamond)\}}\right] = \frac{\mathfrak{p}}{\pi} \Gamma(\theta - z) \Gamma(1 + z) \sin(\pi\theta (1 - \rho)).
        \end{equation}
    \end{enumerate}
    In both these cases $\zeta(\cdot)$ denotes the lifetime of the process (when the process is sent to the cemetery state, which we set to be $+\infty$). With the notation $\sigma \coloneqq 1/2 - \theta(1 - \rho) \in (-1/2, 1/2) \; \text{and} \; b \coloneqq \frac{1}{\pi} \arccos(\mathfrak{p}\cos(\pi\sigma)) \in [|\sigma|, 1/2]$, the Laplace exponent $\Psi^*(z) = \hat{\Psi}(z) + \Psi^\diamond(z)$ of $(\xi^*_s)_{s \geq 0}$ can be expressed in factorized form as 
    \begin{equation}
    \label{psi*}
        \Psi^*(z) = 2^\theta \frac{\Gamma\left(\frac{1 + z}{2}\right) \Gamma\left(\frac{2 + z}{2}\right) \Gamma\left(\frac{\theta - z}{2}\right) \Gamma\left(\frac{1 + \theta - z}{2}\right)}{\Gamma\left(\frac{\sigma + b + z}{2}\right) \Gamma\left(\frac{\sigma - b + z}{2}\right) \Gamma\left(\frac{2 - \sigma + b - z}{2}\right) \Gamma\left(\frac{2 - \sigma - b - z}{2}\right)}.
    \end{equation}
\end{proposition}

We note that for all allowed parameters $(\theta, \rho) \in \mathbb{S}$ and $\mathfrak{p} \in [0, 1]$, the Laplace exponent $\Psi^*$ is analytic in $(-1, \theta)$ and has precisely two zeroes $z = \pm b - \sigma$ on this interval, except when $\mathfrak{p} = 1$ and $\rho = 1 - 1/(2\theta)$ when both zeroes coincide. Thus, let us exclude this case and for now assume that $\mathfrak{p} < 1$ whenever $\rho = 1 - 1/(2\theta)$. Then the functions $\mathfrak{h}^\downarrow, \mathfrak{h}^\uparrow: (0, \infty) \rightarrow \mathbb{R}$ defined as
\begin{equation}
    \label{h_arrow}
    \mathfrak{h}^\downarrow (x) \coloneqq x^{-b - \sigma}, \quad \mathfrak{h}^\uparrow (x) \coloneqq x^{b - \sigma}
\end{equation}
are $\mathcal{A}^*$-harmonic on $(0, \infty)$. In the following, we restrict ourselves to $\mathfrak{h}^\downarrow$, but the same can be done with $\mathfrak{h}^\uparrow$.
Now we can introduce the corresponding $h$-transform $(X^\downarrow_t)_{t \geq 0}$ of the process $(X^*_t)_{t \geq 0}$ under $\mathfrak{h}^\downarrow$ with generator
\begin{align}
\label{arrow_gener}
    \mathcal{A}^\downarrow f(x) = \frac{1}{\mathfrak{h}^\downarrow (x)} \left( \mathcal{A}^* \mathfrak{h}^\downarrow f\right) (x).
\end{align}

This process is again a pssMp with index $\theta$. We denote its Lamperti representation by $(\xi^\downarrow_t)_{t \geq 0}$. Its Laplace exponent is then obtained by a shift of $\Psi^*(z)$, i.e.
\begin{align*}
    \Psi^\downarrow(z) &= \Psi^*(z - b - \sigma) = \frac{1}{\pi} \Gamma(\theta + b + \sigma - z) \Gamma(1 - b - \sigma + z) \left(\cos(\pi b) - \cos(\pi(z - b))\right). 
\end{align*}
The corresponding L\'evy measure is $\Pi^\downarrow(\mathrm{d}y) = \expo^{-(b + \sigma)y} \Pi^*(\mathrm{d}y)$. Since $\Psi^\downarrow(0) = 0$, the corresponding L\'evy process has infinite lifetime. Furthermore, if $\varphi^\downarrow$ is the corresponding Laplace transformation, then $\mathbb{E}\left[ \xi^\downarrow_1\right] = (\varphi^\downarrow)'(0) < 0$. 
This means that $(\xi^\downarrow_s)_{s \geq 0}$ drifts almost surely to $-\infty$, which implies that the pssMp $(X^\downarrow_t)_{t \geq 0}$ continuously hits zero. We set $X^\downarrow_t = 0$ after the process hits the origin.

\subsection{Scaling limit}
\label{subsec:scalLim}
The goal of this Section is to prove Theorem \ref{thm:scallim}. For this reason we restrict ourselves to the according setting, less general than in the discussion of Section \ref{subsec:RSP}. Let $\nu$ be a \emph{non-generic critical} probability measure \emph{of index $\theta \in \left( \frac{1}{2}, \frac{3}{2}\right)$}, meaning that $\nu$ is admissible, the random walk with jump probabilities $\nu$ oscillates, and the tails of $\nu$ as $k \rightarrow \infty$ are $\nu(-k) \sim c k^{-\theta - 1}$ and $\nu(k) \sim -c \cos{(\pi \theta)} k^{-\theta - 1}$ for some $c > 0$. A random walk with such a law $\nu$ is in the domain of attraction of $\theta$-stable process with $c_+ / c_- = -\cos{(\pi \theta)}$. This follows, e.g., from \cite[Theorem 1a Chapter \Romannum{9}.8]{Feller_book}. 
We also note that $c_+ / c_- = -\cos{(\pi \theta)}$ is equivalent to a positivity parameter $\rho = 1 - 1 / (2\theta)$ by (\ref{Levy_coef}). We, therefore, restrict ourselves in the remainder of this chapter to such processes. More precisely, we assume
\begin{align*}
    \rho = 1 - \frac{1}{2\theta}, \quad \sigma = 0, \quad c_- = \frac{1}{\pi} \Gamma(\theta + 1), \quad c_+ = -\cos(\pi \theta) c_-, \quad b = \frac{1}{\pi} \arccos(\mathfrak{p}).
\end{align*}

Analogously to Section \ref{subsec:nest_st} we first prove the statement similar to Theorem \ref{thm:scallim} for the accelerated ricocheted random walk and then via Proposition \ref{prop:h-trafo} yield the result for the perimeter process. The similar result for the rigid $O(n)$ loop model was proven in \cite[Section 6.2]{TB18}. We will refer to their proof highlighting and establishing the changes which have to be made due to the additional acceleration appearing in our case. Most of the technical details will be presented in the appendix.
\begin{proposition}
\label{prop:ric_converg}
    Suppose $\nu$ is non-generic critical of index $\theta \in (1/2, 3/2)$ and $\mathfrak{p} \in (0, 1)$. Let $(W_i^r)_{i \geq 0}$ under $\mathbb{P}^r_p$ be the associated accelerated $\mathfrak{p}$-ricocheted random walk of law $\nu$ started at $p$ and conditioned to be trapped at zero. Then the following convergence holds in distribution in the $J_1$-Skorokhod topology,
    \begin{align}
    \label{ricoch_conv}
        \left( \frac{W^r_{\lfloor Cp^\theta t \rfloor}}{p} \right)_{t \geq 0} \xrightarrow[p \rightarrow \infty]{(\text{d})} (X^\downarrow_t)_{t \geq 0},
    \end{align}
    where $(X^\downarrow_t)_{t \geq 0}$ is the $\mathfrak{p}$-ricocheted stable process defined in Section \ref{subsec:RSP} with parameters $(\theta, \rho = 1 - 1/(2\theta), \mathfrak{p})$ started at $X^\downarrow_0 = 1$, and $C \coloneqq c_+ / \lim_{k \rightarrow \infty} k^{\theta + 1}\nu(k) = \Gamma(1 + \theta) / (\pi\lim_{k \rightarrow \infty} k^{\theta + 1}\nu(-k))$. 
\end{proposition}
The main ingredient of the proof of this proposition is an invariance principle established in \cite{BK16}. To apply this result we need to verify some assumptions. The main step in this direction is the following lemma.
\begin{lemma}
\label{lemma:gen_conv}
    Under the assumptions of Proposition \ref{prop:ric_converg} for any twice-differentiable function $f: (0, \infty) \rightarrow \mathbb{R}$ with compact support it holds that
    \begin{align}
    \label{scal_gener_conv}
        C p^\theta \mathbb{E}^r_p \left[ f\left(\frac{W^r_1}{p}\right) - f(1)\right] \xrightarrow{p \rightarrow \infty} \mathcal{A}^\downarrow f(1),
    \end{align}
    where $\mathcal{A}^\downarrow$ is the infinitesimal generator (\ref{arrow_gener}).
\end{lemma}
We can easily extend the proof of \cite[Lemma 7]{TB18} to conclude the desired statement once we establish that as $p \rightarrow \infty$,
\begin{align}
\label{asympt_scal}
    Cp^\theta \Bigg[-f(1) &+ \sum_{k > -p} \left(1 + \frac{k}{p}\right)^{-b} f\left(1 + \frac{k}{p}\right)
    \left(\nu(k) + \mathfrak{p} \sum_{u \geq 1} \binom{u + p + k - 1}{p + k} \frac{\nu(-u - p)}{2^{u + p + k}}  \right) \Bigg] \nonumber \\
    &= C p^\theta \sum_{s \in \mathbb{Z}} \left[f^\downarrow \left(1 + \frac{s}{p}\right)  -f^\downarrow (1) \right] \nu(s) + \mathcal{O}\left( \frac{\log^{3/2}(p)}{\sqrt{p}}\right),
\end{align}
where $f^\downarrow$ is defined as
$f^\downarrow(x) \coloneqq \mathfrak{h}^\downarrow(x) f(x) \mathds{1}_{\{x \geq 0\}} + \mathfrak{p} \mathfrak{h}^\downarrow(-x) f(-x) \mathds{1}_{\{x < 0\}}$. \\
The latter is done in Appendix \ref{A:subsec:f_downarrow}. See also \cite{thesis} for the detailed proof of the lemma.

\begin{proof}[Proof of Proposition \ref{prop:ric_converg}]
The main ingredient to prove the convergence (\ref{ricoch_conv}) to the self-similar Markov process with the Lamperti representation $\xi^\downarrow$, which has Laplace exponent $\Psi^\downarrow(z)$ computed in the previous section, is \cite[Theorem 2]{BK16}. We recall that $\Pi^\downarrow(\mathrm{d}y) = \expo^{-by} \Pi^*(\mathrm{d}y) = \expo^{(1 - b)y} \mu(\expo^y) \mathrm{d}y$ is the L\'evy measure of $\xi^\downarrow$, and we can rewrite the Laplace exponent $\Psi^\downarrow$ in the L\'evy-Khintchine form as
\begin{equation*}
    \Psi^\downarrow(z) = b^\downarrow z + \int_{\mathbb{R}} \left( \expo^{zy} - 1 - zy \mathds{1}_{\{|y| \leq 1\}}\right) \Pi^\downarrow(\mathrm{d}y),
\end{equation*}
where $b^\downarrow$ is some constant which can be computed explicitly from the expression for $\Pi^*$ obtained in the proof of Proposition \ref{prop:Lamperti_repr}. To use \cite[Theorem 2]{BK16} we need to verify three assumptions, (A1) - (A3). \par
Lemmas \ref{lemma:*gener} and \ref{lemma:gen_conv} imply that for any twice-differentiable function $f: (0, \infty) \rightarrow \mathbb{R}$ with compact support in $(0, 1) \cup (1, \infty)$,
\begin{align}
    C p^\theta \mathbb{E}^r_p \left[ f\left(\frac{W^r_1}{p}\right) - f(1)\right] \xrightarrow{p \rightarrow \infty} \mathcal{A}^\downarrow f(1) = \int_{\mathbb{R}_+} \mathfrak{h}^\downarrow(u) f(u) \mu(u) \mathrm{d}u = \int_\mathbb{R} f(\expo^y) \Pi^\downarrow(\mathrm{d}y).
    \label{assumption1}
\end{align}
This convergence is the assumption (A1) in \cite{BK16}, except that it has to hold true for continuous instead of twice-differentiable functions. This can be done by approximating these continuous with twice-differentiable functions and comparing the large $p$-limit of (\ref{asympt_scal}) with the right-hand side of (\ref{assumption1}) (see \cite[A.9]{thesis} for the details). Assumption (A2) is equivalent to Lemma \ref{lemma:gen_conv} with $f$ chosen first such that $f(z) = \log(z)$ and then such that $f(z) = \log^2(z)$, resp., for $z$ in some small neighbourhood of $1$. Assumption (A3) follows from the fact that for any $\beta \in (0, b)$, as $p \rightarrow \infty$,
\begin{equation}
\label{upper_bound}
    \begin{aligned}
        p^\theta \mathbb{E}^r_p &\left[ \left(\frac{W^r_1}{p}\right)^\beta \mathds{1}_{\{\log(W^r_1/p) > 1\}}\right] \\
        &\lesssim p^\theta \sum_{k \geq 2p} \left( \frac{k}{p}\right)^{\beta - b} \left[ (k - p)^{-\theta - 1} + \sum_{l \geq 1} \binom{l + k - 1}{k} \frac{(p + l)^{-\theta - 1}}{2^{k + l}}\right] \lesssim 1, 
    \end{aligned}
\end{equation}
where $a_p \lesssim b_p$ means that as $p \rightarrow \infty$ there exists a uniform constant $C > 0$ such that $a_p \leq C b_p$. The first estimate follows from the tail-asymptotics of $\nu$ and the asymptotics $(\romannum{3})$ in Proposition \ref{prop:h-props}. The second inequality we prove in \ref{A:subsec:upper_bound}. \\
Since $\xi^\downarrow$ drifts to $-\infty$ and (A1) - (A3) are fulfilled, we can apply \cite[Theorem 2]{BK16} to yield the desired convergence (\ref{ricoch_conv}) of $(W^r_i)_{i \geq 0}$.
\end{proof}

\begin{proof}[Proof of Theorem \ref{thm:scallim}]
Suppose $(\mathbf{q}, g, n) \in \mathcal{D}$ is admissible and non-generic (dilute, dense) critical with exponent $a$. By Proposition \ref{prop:h-trafo} the perimeter process $(P_i)_{i \geq 0}$ of the $(\mathbf{q}, g, n)$-Boltzmann triangular loop-decorated map with perimeter $p$ has the same law as the accelerated $\frac{n}{2}$-ricocheted random walk $(W_i^r)_{i \geq 0}$ of law $\nu_{\mathbf{\hat{q}}}$ under $\mathbb{P}^r_p$. \\
We recall that in the non-generic critical phase $g \gamma_+ = 1/2$ and $\mathbf{\hat{q}}$ is critical.
This implies by \cite[Proposition 4]{TB16} that the random walk with distribution $\nu_{\mathbf{\hat{q}}}$ oscillates. Moreover, from (\ref{asymptotic}) we obtain 
\begin{align*}
    \nu_{\hat{q}}(-k) = 2 \gamma_+^{-k} W^{(k - 2)}(\mathbf{\hat{q}}) \sim 2 \gamma_+^{-2} C_{(\ref{asymptotic})} k^{-a} \quad \text{as} \; k \rightarrow \infty,
\end{align*}
and since $\mathbf{q}$ has finite support, i.e. $q_k = 0$ for all large $k$, via (\ref{fpe}) we get
\begin{align*}
    \nu_{\hat{q}}(k) = \hat{q}_{k + 2} \gamma_+^{k} \sim 2 \gamma_+^{-2} C_{(\ref{asymptotic})} \cos(\pi a) k^{-a} \quad \text{as} \; k \rightarrow \infty
\end{align*} 
(see \ref{A:subsec:nu_+tail} for the proof of the latter asymptotics). Hence, we are in the setting of Proposition \ref{prop:ric_converg} with $\theta = a - 1, \mathfrak{p} = n/2$, and $b = |a - 2|$. \\
The above two arguments imply the statement of the theorem. 
\end{proof}

\appendix
\section{Appendix}
\label{sec:Appendix}
\subsection{Admissibility criteria for Boltzmann maps without loops}
\label{A:subsec:admis_no_loop}
Let us consider Boltzmann maps without loops. We call the weight sequence $\mathbf{\hat{q}}$ \emph{$\bullet$-admissible} if $W_\bullet^{(2)}(\mathbf{\hat{q}}) < \infty$, or equivalently $W_\bullet^{(p)}(\mathbf{\hat{q}}) < \infty$ for all $p \geq 1$. The equivalence follows from \cite[Proposition 2 \& (17)]{TB16}.

The main goal of this section is to prove that the admissibility of the weight sequence implies the  $\bullet$-admissibility. An analogous result for the bipartite maps was shown in \cite[Corollary 3.15]{Curien_notes}.
\begin{lemma}
\label{lemma:admis_bullet}
    Assume that $\mathbf{\hat{q}}$ is an admissible weight sequence with $\hat{q}_k > 0$ for some odd $k \geq 3$, then $\mathbf{\hat{q}}$ is $\bullet$-admissible. 
\end{lemma}

Before proceeding to the proof we need to introduce several notations and make some observations. We note that there is a natural bijection between pointed planar maps with root face of degree $2$ (with more than one face) and pointed rooted planar maps with arbitrary root face degree, corresponding to the gluing of the edges of the root face. We denote the latter set as $\mathcal{M}_\bullet$ and to each $\mathfrak{m}_\bullet \in \mathcal{M}_\bullet$ we associate the weight $\mathbf{w}_{\mathbf{\hat{q}}}(\mathfrak{m}_\bullet) \coloneqq \prod_{f} \hat{q}_{\mathrm{deg}(f)}$, where the product runs over all faces of $\mathfrak{m}_\bullet$ (including the root face). We set $Z_\bullet(\mathbf{\hat{q}}) \coloneqq \mathbf{w}_{\mathbf{\hat{q}}}(\mathcal{M}_\bullet)$. Hence $W_\bullet^{(2)}(\mathbf{\hat{q}}) - 1 = Z_\bullet(\mathbf{\hat{q}})$, which coincides with the definition of $Z_\bullet$ in the proposition. By comparing the distances between the distinguished vertex and the source (target) vertex of the root, we can split $\mathcal{M}_\bullet$ in three disjoint sets $\mathcal{M}_\bullet^+, \mathcal{M}_\bullet^-$ and $\mathcal{M}_\bullet^0$. Namely, $\mathcal{M}_\bullet^+, \mathcal{M}_\bullet^-$ and $\mathcal{M}_\bullet^0$ correspond to the cases when the distance between the marked vertex and the target vertex of the root is greater (smaller and equal, resp.) than the distance between the marked vertex and the source vertex of the root. On top of that we assume that the map consisting only of one vertex belongs to $\mathcal{M}^+_\bullet$. Then $Z_\bullet^i (\mathbf{\hat{q}}) \coloneqq \mathbf{w}_{\mathbf{\hat{q}}}(\mathcal{M}^i_\bullet)$ with $i \in \{+, -, 0\}$. 

The starting point for the proof is the following result \cite[Proposition 1]{GM15} by Miermont.
\begin{proposition}
\label{prop:miermont}
    Define 
    \begin{align*}
        f_\mathbf{\hat{q}}^\bullet (x, y) &\coloneqq \sum_{k \geq 0} \sum_{k' \geq 0} x^k y^{k'} \binom{2k + k' + 1}{k + 1} \binom{k + k'}{k} \hat{q}_{2+2k+k'} \\
        f_\mathbf{\hat{q}}^\diamond (x, y) &\coloneqq \sum_{k \geq 0} \sum_{k' \geq 0} x^k y^{k'} \binom{2k + k'}{k} \binom{k + k'}{k} \hat{q}_{1+2k+k'}
    \end{align*}
    The non-bipartite sequence $\mathbf{\hat{q}}$ is $\bullet$-admissible iff there exist $z^+, z^\diamond > 0$ such that 
    \begin{align*}
        f_\mathbf{\hat{q}}^\bullet (z^+, z^\diamond) = 1 - \frac{1}{z^+}, \quad f_\mathbf{\hat{q}}^\diamond (z^+, z^\diamond) = z^\diamond
    \end{align*}
    and the matrix
    \begin{align*}
        \mathfrak{M}_{\mathbf{\hat{q}}} (z^+, z^\diamond) \coloneqq 
        \begin{pmatrix}
            0 & 0 & z^+ - 1\\
            \frac{z^+}{z^\diamond} \partial_x f_{\mathbf{\hat{q}}}^\diamond (z^+, z^\diamond) & \partial_y f_{\mathbf{\hat{q}}}^\diamond (z^+, z^\diamond) & 0 \\
            \frac{(z^+)^2}{z^+ - 1} \partial_x f_{\mathbf{\hat{q}}}^\bullet (z^+, z^\diamond) & \frac{z^+ z^\diamond}{z^+ - 1} \partial_y f_{\mathbf{\hat{q}}}^\bullet (z^+, z^\diamond) & 0
        \end{pmatrix}
    \end{align*}
    has spectral radius $\rho_{\mathbf{\hat{q}}} \leq 1$. Moreover, if $\mathbf{\hat{q}}$ is $\bullet$-admissible the solution $z^+, z^\diamond$ is unique and $Z_\bullet^+ (\mathbf{\hat{q}}) = z^+, Z_\bullet^0 (\mathbf{\hat{q}}) = (z^\diamond)^2$, so that the partition function $Z_\bullet (\mathbf{\hat{q}}) \coloneqq W_\bullet^{(2)}(\mathbf{\hat{q}}) - 1 = 2z^+ + (z^\diamond)^2 - 1$.
\end{proposition}
\begin{proof}[Proof of Lemma \ref{lemma:admis_bullet}]
Let $u \in (1/2, 1)$, define $(\hat{q}_u)_k = q_k u^{k - 2}$. If $\mathfrak{m} \in \mathcal{M}^{(p)}$, using Euler's formula and $\sum_{f} \mathrm{deg}(f) = 2|\text{Edges}(\mathfrak{m})| - p$, where $f$ varies over inner faces of $\mathfrak{m}$, we obtain
\begin{align*}
    \mathbf{w}_{\mathbf{\hat{q}}_u}(\mathfrak{m}) 
    = u^{-2 - p} u^{2|\mathfrak{m}|} \mathbf{w}_{\mathbf{\hat{q}}}(\mathfrak{m})
\end{align*}
where  
$|\mathfrak{m}|$ is the number of vertices of $\mathfrak{m}$. Since $k \cdot u^{2k}$ is uniformly bounded from above by some constant depending on $u$, we deduce that
\begin{align*}
    W_\bullet^{(p)}(\mathbf{\hat{q}}_u) = u^{-2 - p} \sum_{\mathfrak{m} \in \mathcal{M}^{(p)}} |\mathfrak{m}| u^{2 |\mathfrak{m}|} \mathbf{w}_{\mathbf{\hat{q}}}(\mathfrak{m}) \leq C_u W^{(p)}(\mathbf{\hat{q}}) < \infty.
\end{align*}
Hence $\mathbf{\hat{q}}_u$ is $\bullet$-admissible for any $u \in (1/2, 1)$. Proposition \ref{prop:miermont} implies that there is a positive solution to $f_{\mathbf{\hat{q}}_u}^\bullet (z^+, z^\diamond) = f_{\mathbf{\hat{q}}}^\bullet (u^2 z^+, u z^\diamond) = 1 - \frac{1}{z^+}, \; u \cdot f_{\mathbf{\hat{q}}_u}^\diamond (z^+, z^\diamond) = f_\mathbf{\hat{q}}^\diamond (u^2 z^+, u z^\diamond) = u z^\diamond$ such that the spectral radius $\rho_{\mathbf{\hat{q}}_u}$ of $\mathfrak{M}_{\mathbf{\hat{q}}_u} (z^+, z^\diamond)$ is $\leq 1$ for any $u \in (1/2, 1)$. With the change of variables $(x, y) \coloneqq (u^2 z^+, uz^\diamond)$ and notation $f^\bullet = f_{\mathbf{\hat{q}}}^\bullet, f^\diamond = f_{\mathbf{\hat{q}}}^\diamond$ we can restate the system of equations as 
\begin{align}
\label{eq_syst}
    f^\bullet (x, y) = 1 - \frac{u^2}{x}, \quad f^\diamond (x, y) = y,
\end{align}
and rewrite $\mathfrak{M}_{\mathbf{\hat{q}}_u} (z^+, z^\diamond)$ as
\begin{align*}
    \mathfrak{M}_{\mathbf{\hat{q}}_u} \left(\frac{x}{u^2}, \frac{y}{u}\right) &=
        \begin{pmatrix}
            0 & 0 & \frac{x - u^2}{u^2}\\
            \frac{x}{y} \partial_1 f^\diamond & \partial_2 f^\diamond & 0 \\
            \frac{x^2}{x - u^2} \partial_1 f^\bullet & \frac{yx}{x - u^2} \partial_2 f^\bullet & 0
        \end{pmatrix}
        \overset{(\ref{eq_syst})}{=} \begin{pmatrix}
            0 & 0 & \frac{x f^\bullet}{u^2}\\
            \frac{x}{f^\diamond} \partial_1 f^\diamond & \partial_2 f^\diamond  & 0 \\
            \frac{x}{f^\bullet} \partial_1 f^\bullet & \frac{f^\diamond}{f^\bullet} \partial_2 f^\bullet & 0
        \end{pmatrix} \\
        &= \begin{pmatrix}
            0 & 0 & \frac{x f^\bullet}{u^2}\\
            x \partial_1 \log f^\diamond & \partial_2 f^\diamond  & 0 \\
            x \partial_1 \log f^\bullet & f^\diamond \partial_2 \log f^\bullet & 0
        \end{pmatrix} 
        = \mathfrak{M}_{\mathbf{\hat{q}}} \left(x, y\right) +
        \begin{pmatrix}
            0 & 0 & x f^\bullet \left(\frac{1}{u^2} - 1\right)\\
            0 & 0 & 0 \\
            0 & 0 & 0
        \end{pmatrix}
\end{align*}
In the entries of matrices $f^\bullet, f^\diamond$ and its derivatives are evaluated at $(x, y)$; and the last equality holds by repeating the same computations for the original system given in the Propostion \ref{prop:miermont}.
Via the performed reformulation and Proposition \ref{prop:miermont} we know that for any $u \in (1/2, 1)$, the system of equations (\ref{eq_syst}) has the unique positive solution, denoted by $(x_u, y_u) \in \mathbb{Z}^2_{> 0}$ and given by $(x_u, y_u) = (u^2 Z_\bullet^+ (\mathbf{\hat{q}}_u), u \sqrt{Z_\bullet^0 (\mathbf{\hat{q}}_u)})$, with $\mathfrak{M}_{\mathbf{\hat{q}}_u} \left(\frac{x_u}{u^2}, \frac{y_u}{u}\right)$ of spectral radius $\leq 1$. Moreover, we get $x_u > u^2$, and since the total weights $Z_\bullet^+ (\mathbf{\hat{q}}_u), Z_\bullet^0 (\mathbf{\hat{q}}_u)$ monotonically increase in $u$, the sequence of solutions $((x_u, y_u))_u$ is non-decreasing as $u \nearrow 1$ as well. From the equations (\ref{eq_syst}) and assumption that $\hat{q}_k > 0$ for some odd $k \geq 3$ we deduce that $(x_u, y_u)_u$ is uniformly bounded from above (in each component). Indeed, since the left-hand side of the first equation in (\ref{eq_syst}) lies in $[0, 1]$ for any value of $x$, and since $\hat{q}_k > 0$ for some odd $k \geq 3$, we deduce that $y_u$ is uniformly bounded from above; now using the second equation in (\ref{eq_syst}) we conclude that the same has to hold for $x_u$. Therefore, we can define 
\begin{equation*}
    (x_\bullet, y_\bullet) \coloneqq \lim_{u \nearrow 1} \: (x_u, y_u) < \infty,
\end{equation*}
and $y_\bullet > 0, x_\bullet \geq 1$ (as $x_u, y_u$ are increasing and each $x_u > u^2, y_u > 0$). Now by taking limit $u \nearrow 1$ in (\ref{eq_syst}), we deduce that $(x_\bullet, y_\bullet) \in \mathbb{Z}^2_{\geq 0}$ solves (\ref{eq_syst}) for $u = 1$. Indeed,
\begin{align*}
    f^\bullet(x_\bullet, y_\bullet) &\overset{\text{Fatou}}{\leq} \liminf_{u \nearrow 1} f^\bullet(x_u, y_u) 
    \overset{(\ref{eq_syst})}{=} 1 - \limsup_{u \nearrow 1} \frac{u^2}{x_u} = 1 - \frac{1}{x_\bullet} \\
    &\; = 1 - \liminf_{u \nearrow 1} \frac{u^2}{x_u} \overset{(\ref{eq_syst})}{=} \limsup_{u \nearrow 1} f^\bullet(x_u, y_u) \leq f^\bullet(x_\bullet, y_\bullet)
\end{align*}
where the last inequality follows from the fact that $x_u \leq x_\bullet, y_u \leq y_\bullet$. Since the right-hand side and the left-hand side coincide, we obtain $f^\bullet(x_\bullet, y_\bullet) = 1 - \frac{1}{x_\bullet}$. Analogously we deduce that $f^\diamond(x_\bullet, y_\bullet) = y_\bullet$.

If we set $A_u \coloneqq \mathfrak{M}_{\mathbf{\hat{q}}_u} \left(x_u / u^2, y_u / u\right), B_u \coloneqq (b_{ij})$ with $b_{13} = (1 - u^2)(x_u / u^2 - 1), b_{kl} = 0$ otherwise ($1 \leq i, j \leq 3, 1 \leq l < 3, 2 \leq k \leq 3$), we get from the above computation and (\ref{eq_syst}) that $\mathfrak{M}_{\mathbf{\hat{q}}} \left(x_u, y_u\right) = A_u - B_u$. Since $B_u^2 = 0$, it holds $(A_u - B_u)^n = A_u^n - n A_u^{n - 1} B_u$ for any $n$. Furthermore, w.r.t. the matrix norms $\norm{\cdot}_1$ (or $\norm{\cdot}_\infty$) corresponding to the maximum absolute column (row) sum of the matrix, the following is true:
\begin{align}
\label{matrix_estim}
    \norm{A_u^n - n A_u^{n - 1} B_u}^{\frac{1}{n}} &\leq \left(\norm{A_u^n} + n |b_{13}| \norm{A_u^{n - 1}}\right)^{\frac{1}{n}} \\
    &\leq \norm{A_u^n}^{\frac{1}{n}} + n^{\frac{1}{n}} |b_{13}|^{\frac{1}{n}} \norm{A_u^{n - 1}}^{\frac{1}{n}} 
    \xrightarrow{n \rightarrow \infty} \rho\left(A_u \right) + 1 \cdot 1 \cdot \rho\left(A_u \right), \nonumber
\end{align}
where the convergence on the right-hand side follows by Gelfand's theorem (note that for each fixed $u \in (1/2, 1)$, $|b_{13}|$ is finite and independent of $n$). Applying Gelfand's theorem, i.e. $\rho(A) = \lim_{n \rightarrow \infty} \norm{A^n}^{1/n}$ for any matrix norm $\norm{\cdot}$, to the left-hand side we obtain that the spectral radius $\rho\left( \mathfrak{M}_{\mathbf{\hat{q}}} \left(x_u, y_u\right) \right) \leq 2\rho\left(A_u \right) \leq 2$ for all $u \in (1/2, 1)$. On the other hand, for any $\varepsilon > 0$ we can choose $n$ sufficiently large such that $\norm{A_u^n}^{\frac{1}{n}} < 1 + \varepsilon$ (since $\mathbf{\hat{q}}_u$ is admissible for all $u < 1$), now if we first take the limit $u \nearrow 1$ in (\ref{matrix_estim}) we get 
\begin{align*}
    \norm{\mathfrak{M}^n_{\mathbf{\hat{q}}} \left(x_\bullet, y_\bullet\right)}^{\frac{1}{n}} &= \lim_{u \nearrow 1} \norm{\mathfrak{M}^n_{\mathbf{\hat{q}}} \left(x_u, y_u\right)}^{\frac{1}{n}} \\
    &\leq 1 + \varepsilon + n^{\frac{1}{n}} (1 + \varepsilon) \lim_{u \nearrow 1} |b_{13}|^{\frac{1}{n}}  = 1 + \varepsilon.
\end{align*}
In the last line we have used that $x_\bullet < \infty$. Now taking the limit $n \rightarrow \infty$ we obtain by Gelfand's theorem that $\rho\left( \mathfrak{M}_{\mathbf{\hat{q}}} \left(x_\bullet, y_\bullet\right) \right) \leq 1 + \varepsilon$. Sending $\varepsilon$ to zero, yields that the spectral radius of $\mathfrak{M}_{\mathbf{\hat{q}}} \left(x_\bullet, y_\bullet\right)$ is $\leq 1$. Hence Proposition \ref{prop:miermont} implies that $\mathbf{\hat{q}}$ is $\bullet$-admissible. 
\end{proof}

\subsection{(\ref{assmpt_nu}) implies \texorpdfstring{$\sqrt{\frac{2g^2}{\nu(-2)}} \leq \frac{1}{2}$}{\ref{assmpt_nu}) implies 2gg / v(-2) <= 1/4}}
\label{A:subsec:ineq_implication}
Recall that we want to show that if for all $k \geq 1$
\begin{align*}
    \nu (k - 2) \geq \frac{n}{2} \sum_{l \geq 0} \binom{k + l - 1}{l} \left( \frac{2g^2}{\nu(-2)} \right)^{\frac{k + l}{2}} \nu(-l - 2).
\end{align*}
for an admissible probability measure $\nu$, then $x \coloneqq \sqrt{\frac{2g^2}{\nu(-2)}} \leq \frac{1}{2}$. Since the left hand side is summable in $k$, so is the right one. Therefore, $x \leq 1$ and summation over $k\geq 1$ yields
\begin{align*}
    \mathcal{S}_x \coloneqq \frac{n}{2} \sum_{l \geq 0} \left( \frac{x}{1-x} \right)^{l+1} \nu(-l - 2) \leq 1.
\end{align*}
Now, it suffices to prove that $\lim_{l\rightarrow\infty} \frac{\nu(-l-1)}{\nu(-l)} = 1$ as it would imply that the radius of convergence of the above power series $\mathcal{S}_x$ is $1$. To this end, let $\mathbf{\hat{q}}$ be an admissible non-bipartite weight sequence and recall from \ref{A:subsec:admis_no_loop} the notion $(\hat{q}_u)_k = \hat{q}_k u^{k - 2}$ for $u \in (0, 1)$ and that then
\begin{align*}
    \mathbf{w}_{\mathbf{\hat{q}}_u}(\mathfrak{m}) =  u^{-2 - p} u^{2|\mathfrak{m}|} \mathbf{w}_{\mathbf{\hat{q}}}(\mathfrak{m}),
\end{align*}
where $|\mathfrak{m}|$ is the number of vertices of $\mathfrak{m}$. Moreover,
\begin{align}
\label{eq:to_be_integrated}
    W_\bullet^{(p)}(\mathbf{\hat{q}}_u) = u^{-p} \sum_{\mathfrak{m} \in \mathcal{M}^{(p)}} |\mathfrak{m}| (u^{2})^{|\mathfrak{m}| - 1} \mathbf{w}_{\mathbf{\hat{q}}}(\mathfrak{m}) < \infty.
\end{align}
For each fixed $u \in (0,1)$, let $(z^+_u, z^\diamond_u)$ be the solution to the system of equations corresponding to $\mathbf{\hat{q}_u}$ as in Proposition \ref{prop:miermont}. We have seen in the proof of Lemma \ref{lemma:admis_bullet} that $u \mapsto (z^+_u, z^\diamond_u)$ is non-decreasing in both components and that $(z^+, z^\diamond) = \lim_{u\uparrow 1} (z^+_u, z^\diamond_u)$ exists and solves the system of equations corresponding to $\mathbf{\hat{q}}$ ($u=1$). Furthermore, $\gamma_+(\mathbf{\hat{q}}_u) = z^\diamond_u + 2\sqrt{z^+_u}$ (see \cite[end of A.1]{TB16}), and, thus, also non-decreasing as a function of $u$.\\
We recall as well that $W^{(l)}_\bullet(\mathbf{\hat{q}}) \gamma_+^{-l}(\mathbf{\hat{q}}) = h^r_0(l)$ for any fixed $l \geq 0$ and any admissible $\mathbf{\hat{q}}$. Set $v = u^2$ and integrate \eqref{eq:to_be_integrated} over $v$ from $0$ to $1$, then
\begin{align*}
    h^r_0(p) \int_0^1 \mathrm{d}v \left(\sqrt{v} \gamma_+(\mathbf{\hat{q}_{\sqrt{v}}})\right)^{p} = W^{(p)}(\mathbf{\hat{q}_u}).
\end{align*}
By the above considerations, we conclude that $\sqrt{v} \gamma_+(\mathbf{\hat{q}_{\sqrt{v}}}) \leq \gamma_+(\mathbf{\hat{q}})$ for $v \in (0,1)$ and $\sqrt{v} \gamma_+(\mathbf{\hat{q}_{\sqrt{v}}}) \geq \sqrt{1-\varepsilon} \gamma_+(\mathbf{\hat{q}_{\sqrt{1-\varepsilon}}})$ for $v \in (1-\varepsilon, 1)$. Therefore,
\begin{align*}
    \frac{W^{(p+1)}(\mathbf{\hat{q}})}{W^{(p)}(\mathbf{\hat{q}})} = \frac{h^r_0(p+1) \int_0^1 \mathrm{d}v \left(\sqrt{v} \gamma_+(\mathbf{\hat{q}_{\sqrt{v}}})\right)^{p+1}}{h^r_0(p) \int_0^1 \mathrm{d}v \left(\sqrt{v} \gamma_+(\mathbf{\hat{q}_{\sqrt{v}}})\right)^{p}} \leq \gamma_+(\mathbf{\hat{q}}) \frac{h^r_0(p+1)}{h^r_0(p)} \xrightarrow{p \rightarrow \infty} \gamma_+(\mathbf{\hat{q}})
\end{align*}
since $h^r_0(p) \sim \frac{p^{-1/2}}{\Gamma(1/2) \sqrt{1+r}}$ as $p\rightarrow\infty$ by \cite[(23)]{TB16}. On the other hand, using Laplace's method,
\begin{align*}
    \liminf_{p \rightarrow \infty} \frac{W^{(p+1)}(\mathbf{\hat{q}})}{W^{(p)}(\mathbf{\hat{q}})} \geq  \liminf_{p \rightarrow \infty} \frac{\int_{1-\varepsilon}^1 \mathrm{d}v \left(\sqrt{v} \gamma_+(\mathbf{\hat{q}_{\sqrt{v}}})\right)^{p+1}}{\int_{1-\varepsilon}^1 \mathrm{d}v \left(\sqrt{v} \gamma_+(\mathbf{\hat{q}_{\sqrt{v}}})\right)^{p}}\geq \sqrt{1-\varepsilon} \gamma_+(\mathbf{\hat{q}_{\sqrt{1-\varepsilon}}})
\end{align*}
for any $\varepsilon \in (0,1)$. By sending $\varepsilon$ to zero, we get that $\lim_{p \rightarrow \infty} \frac{W^{(p+1)}(\mathbf{\hat{q}})}{W^{(p)}(\mathbf{\hat{q}})} = \gamma_+(\mathbf{\hat{q}})$. We conclude the statement for $\nu$ -  $\lim_{l\rightarrow\infty} \frac{\nu(-l-1)}{\nu(-l)} = 1$ - by Proposition \ref{prop:nu_admissible} and \eqref{nu}.

\subsection{Proof of equality in (\ref{asympt_scal})}
\label{A:subsec:f_downarrow}
In this section we want to complete the proof of Lemma \ref{lemma:gen_conv} by showing that as $p \rightarrow \infty$,
\begin{align}
\label{A:f_down}
    Cp^\theta \Bigg[-f(1) &+ \sum_{k > -p} \left(1 + \frac{k}{p}\right)^{-b} f\left(1 + \frac{k}{p}\right)
    \left(\nu(k) + \mathfrak{p} \sum_{u \geq 1} \binom{u + p + k - 1}{p + k} \frac{\nu(-u - p)}{2^{u + p + k}}  \right) \Bigg] \nonumber \\
    &= C p^\theta \sum_{s \in \mathbb{Z}} \left[f^\downarrow \left(1 + \frac{s}{p}\right)  -f^\downarrow (1) \right] \nu(s) + \mathcal{O}\left( \frac{\log^{3/2}(p)}{\sqrt{p}}\right),
\end{align}
where $f:(0, \infty) \rightarrow \mathbb{R}$ is a twice-differentiable function with compact support, $\mathfrak{h}^\downarrow(x) = x^{-b}$ and $f^\downarrow(x) \coloneqq \mathfrak{h}^\downarrow(x) f(x) \mathds{1}_{\{x \geq 0\}} + \mathfrak{p} \mathfrak{h}^\downarrow(-x) f(-x) \mathds{1}_{\{x < 0\}}$. \\
Let $\delta > 0$ such that $\mathrm{supp}f \subset [\delta, \infty)$, we will work with the sum in the first line of (\ref{A:f_down}).
\begin{align*}
    \mathcal{S}_p &\coloneqq \sum_{k > -p} \mathfrak{h}^\downarrow \left(1 + \frac{k}{p}\right) 
    f\left(1 + \frac{k}{p}\right) \left(\nu(k) + \mathfrak{p} \sum_{u \geq 1} \binom{u + p + k - 1}{p + k} \frac{\nu(-u - p)}{2^{u + p + k}}  \right) \\
    &= \sum_{s \in \mathbb{Z}} \nu(s) \Bigg[ \mathds{1}_{\{s \geq -p\}} \mathfrak{h}^\downarrow \left(1 + \frac{s}{p}\right) f\left(1 + \frac{s}{p}\right) + \\ 
    &\hspace{2cm} + \mathds{1}_{\{s < -p\}} \mathfrak{p} \sum_{k \geq -p} \mathfrak{h}^\downarrow \left(1 + \frac{k}{p}\right) f\left(1 + \frac{k}{p}\right) \binom{-s + k - 1}{p + k} \left(\frac{1}{2}\right)^{k - s} \Bigg].
\end{align*}
We define 
\begin{align*}
    f^\downarrow_p(x) \coloneqq \mathds{1}_{\{x \geq 0\}} \mathfrak{h}^\downarrow (x) f(x) + 
    \mathds{1}_{\{x < 0\}} \mathfrak{p} \sum_{l \geq 0} \mathfrak{h}^\downarrow \left(\frac{l}{p}\right) f\left(\frac{l}{p}\right) \binom{-px + l - 1}{l} \left(\frac{1}{2}\right)^{l - px}.
\end{align*}
Then $\mathcal{S}_p = \sum_{s \in \mathbb{Z}} \nu(s) f^\downarrow_p \left( 1 + \frac{s}{p} \right), \; f^\downarrow_p(1) = f(1) = f^\downarrow(1)$, and therefore, the right hand side of \ref{A:f_down} is equal to 
\begin{align*}
    Cp^\theta \sum_{s \in \mathbb{Z}} \nu(s) \left[f^\downarrow_p \left(1 + \frac{s}{p}\right) - f^\downarrow_p(1) \right] 
    = Cp^\theta \sum_{s \in \mathbb{Z}} \nu(s) \left[f^\downarrow \left(1 + \frac{s}{p}\right) - f^\downarrow(1) \right] + \mathcal{O}\left( \frac{\log^{3/2}(p)}{\sqrt{p}}\right) 
\end{align*}
provided $f^\downarrow_p(x) = f^\downarrow(x) + \mathcal{O}\left( \frac{\log^{3/2}(p)}{\sqrt{p}}\right)$ for $x < 0$ as $p \rightarrow \infty$ with the uniform error in $-x \in \mathrm{supp}(f)$, and using $\nu(s) \sim c |s|^{-\theta - 1}$ as $s \rightarrow \infty$. This is the desired result, so we only need to show that for $x < 0$,
\begin{align*}
    f^\downarrow_p(x) = f^\downarrow(x) + \mathcal{O}\left( \frac{\log^{3/2}(p)}{\sqrt{p}}\right)
\end{align*}
with the uniform error for $-x \in \mathrm{supp}(f)$.
Note also that $f^\downarrow_p(x) = f^\downarrow(x)$ for any $x \geq 0$. We fix an $x \leq -\delta$, then as $\mathrm{supp}f \subset [\delta, \infty)$,
\begin{align*}
    \sum_{l \geq 0} \mathfrak{h}^\downarrow &\left(\frac{l}{p}\right) f\left(\frac{l}{p}\right) \binom{-px + l - 1}{l} \left(\frac{1}{2}\right)^{l - px} \\
    &\quad = \sum_{u \in \frac{1}{p}\mathbb{Z}: \: u \geq \delta} \mathfrak{h}^\downarrow (u) f(u) \underbrace{\binom{p(u + |x|) - 1}{pu}}_{= 1 /(pu \mathrm{B}(pu, p|x|))} \left(\frac{1}{2}\right)^{p(u + |x|)} \\
    &\overset{\text{Stirling}}{=} \sqrt{\frac{p}{2\pi}} \sum_{u \in \frac{1}{p}\mathbb{Z}: \: u \geq \delta} \underbrace{\frac{\mathfrak{h}^\downarrow(u) f(u)}{u} \sqrt{\frac{u|x|}{u + |x|}}}_{\eqqcolon g_x(u)} \left( \frac{\left(\frac{u + |x|}{2}\right)^{u + |x|}}{u^u |x|^{|x|}}\right)^{p} \left(1 + \mathcal{O}(1/p)\right) \underbrace{\frac{1}{p}}_{\Delta u} \\
    &\overset{\text{Riemann}}{=} \frac{1}{\sqrt{2\pi}} \int_{u \geq \delta} g_x(u) \left(p^{1/(2p)} \frac{\left(\frac{u + |x|}{2}\right)^{u + |x|}}{u^u |x|^{|x|}}\right)^{p} \mathrm{d}u \: \left(1 + \mathcal{O}(1/p)\right) \eqqcolon \mathcal{I}_x,
\end{align*}
where the error $\mathcal{O}(1/p)$ is uniform in $u, -x \geq \delta$.

We notice that $g_x(u)$ is twice-differentiable and with compact support because of $f$, uniformly bounded and independent of $p$. We also easily see that the function $t: u \in \mathbb{R}_+ \mapsto \frac{\left((u + |x|)/2\right)^{u + |x|}}{u^u |x|^{|x|}}$ is strictly increasing on $(0, |x|)$, attains its maximum $1$ at $u = |x|$ and is strictly decreasing on $(|x|, \infty)$. Therefore, once we find an $\varepsilon > 0$ dependent on $p, |x|$ such that $\frac{\left((u + |x|)/2\right)^{u + |x|}}{u^u |x|^{|x|}} p^{1/(2p)} < 1$ for $u = |x| \pm \varepsilon$, the part $\int_{[\delta, \infty) \setminus (|x| - \varepsilon, |x] + \varepsilon)}$ converges to zero exponentially fast by dominated convergence theorem. To determine $\varepsilon$ we expand $t(u)$ in the neighbourhood of $u = |x|$ and $p^{-1/(2p)}$ at $p \rightarrow \infty$ and compare the coefficients, this results in $\varepsilon = \sqrt{6|x| \log(p) / p}$ and holds for all $p$ sufficiently large such that $\log(p)/p < 8|x|/27$. In particular, by choosing $\log(p)/p < 8\delta/27$, we get the uniform errors for $-x \geq \delta$. The estimation of the error in the Riemann approximation can be, thus, (compact support of $g_x$) reduced  to the interval $I_x = [|x| - \varepsilon, |x| + \varepsilon]$, and from the classical analysis result the error is bounded from above by $\frac{(2\varepsilon)^2}{4p \varepsilon} \sup_{I_x} |(\sqrt{p} t^p(u) g_x(u))'| = \frac{\sqrt{6|x| \log(p)}}{p} \sup_{I_x} t^p(u) \left|g_x'(u) + p g_x(u) \log\left(\frac{u + |x|}{2u}\right)\right|$. By evaluating the derivative at $|x|$, we get that $|x|$ is not an extremum, and thus, since $t^p(u)$ for $u \neq |x|$ converges exponentially fast to zero, so does the error. This establishes the last equality in the above computation. Moreover, it allows to simplify $\mathcal{I}_x$, namely
\begin{align*}
    \mathcal{I}_x 
    &= \sqrt{\frac{p}{2\pi}} \int_{-\varepsilon}^{\varepsilon} g_x(y + |x|) \underbrace{\left( 1 - \frac{y^2}{4|x|} + \mathcal{O}\left(\left(\frac{\log(p)}{p}\right)^{3/2}\right) \right)^p}_{\overset{p \rightarrow \infty}{\approx} \exp\left(-p y^2/(4|x|) + \mathcal{O}\left(\log^{3/2}(p) / \sqrt{p}\right) \right)} \mathrm{d}y \: \left(1 + \mathcal{O}(1/p)\right) \\
    &= \sqrt{\frac{p}{2\pi}} \int_{-\varepsilon}^{\varepsilon} g_x(y + |x|) \expo^{\frac{-p y^2}{4|x|}} \left( 1 + \mathcal{O}\left(\frac{\log^{3/2}(p)}{\sqrt{p}}\right) \right) \mathrm{d}y.
\end{align*}
The error $\mathcal{O}\left(\frac{\log^{3/2}(p)}{\sqrt{p}}\right)$ in the second line is uniform for any $-x \geq \delta$, as so is $\varepsilon$.

We approximate $g_x(y)$ in the neighbourhood of $|x|$ with $g_x(|x|)$, with the error uniformly bounded by $|y|\sup_{u \in I_x} |g'_x|(u) \leq C \varepsilon = \mathcal{O}\left( \sqrt{\log(p) / p}\right)$, with $C < \infty$ as $g_x$ has compact support. Note that $C$ can be chosen uniformly for all $-x \in \mathrm{supp}(f)$, which follows from the explicit expression and continuity in $x$ of $g'_x$ on $\mathrm{supp}(f) \subset [\delta, \infty)$.  Using the change of variables $z \coloneqq y\sqrt{\frac{p}{2|x|}}$ we obtain
\begin{align*}
    \mathcal{I}_x 
    &= f(|x|) \mathfrak{h}^\downarrow(|x|) \:\mathrm{erf}\left( \sqrt{\frac{3\log(p)}{2}} \right) \left( 1 + \mathcal{O}\left(\frac{\log^{3/2}(p)}{\sqrt{p}}\right) \right) \\
    &= f(|x|) \mathfrak{h}^\downarrow(|x|) + \mathcal{O}\left(\frac{\log^{3/2}(p)}{\sqrt{p}}\right),
\end{align*}
where we have used that $\mathrm{erf}(x) \approx 1 + \expo^{-x^2} (-1/(\sqrt{\pi} x) + \mathcal{O}(1/x^2))$ as $x \rightarrow \infty$. The proof is thus completed. 

We also mention that a similar approximation holds for continuous compactly supported instead of twice-differentiable functions. Indeed, the Riemann approximation still works, except that one needs to multiply the integral with $(1 + o(1))$, where $o(1) \rightarrow 0$ as $p \rightarrow \infty$ uniformly for $-x \in \mathrm{supp}(f)$. The approximation of $g_x(y)$ with $g_x(|x|)$ does not cause any problems since we have a uniformly continuous function.

\subsection{Proof of the estimate (\ref{upper_bound})}
\label{A:subsec:upper_bound}
We recall that we want to prove that for any $\beta \in (0, b)$, as $p \rightarrow \infty$,
\begin{align*}
    p^\theta \sum_{k \geq 2p} \left( \frac{k}{p}\right)^{\beta - b} \left[ (k - p)^{-\theta - 1} + \sum_{l \geq 1} \binom{l + k - 1}{k} \frac{(p + l)^{-\theta - 1}}{2^{k + l}}\right] \lesssim 1.
\end{align*}
We first estimate the easier part
\begin{align*}
    p^\theta \sum_{k \geq 2p} \underbrace{\left( \frac{k}{p}\right)^{\beta - b}}_{\leq 1} (k - p)^{-\theta - 1} 
    \lesssim p^\theta \int_p^\infty x^{-\theta - 1} \mathrm{d}x \lesssim 1.
\end{align*}
For the remaining term we proceed as follows
\begin{align*}
    p^\theta &\sum_{k \geq 2p} \left( \frac{k}{p}\right)^{\beta - b} \sum_{l \geq 1} \binom{l + k - 1}{k} \frac{(p + l)^{-\theta - 1}}{2^{k + l}} \\
    &\lesssim \int_1^\infty \left( \frac{1}{u + 1}\right)^{1 + b - \beta} \int_{1 + 1/p}^\infty \frac{1}{\mathrm{B}(p(u + 1), p(v - 1))} \frac{v^{-\theta - 1}}{2^{p(u + v)}} \mathrm{d}v \mathrm{d}u \eqqcolon \mathcal{J}_p, 
\end{align*}
where the upper estimate up to a constant follows via Riemann approximation and setting $u \coloneqq (k-p)/p, \: v \coloneqq (p+l)/p$. 
Once we manage to show that for some $C > 0$ independent of $u, p$, 
\begin{align}
\label{A:up_estim}
    \mathcal{R}_p(u) \coloneqq \int_{1/p}^\infty \frac{1}{\mathrm{B}(pu, pv)} \frac{(v +1)^{-\theta - 1}}{2^{p(u + v)}} \mathrm{d}v \leq C,
\end{align}
it follows that for any $\beta \in (0, b)$, $\mathcal{J}_p \leq C \int_{2}^\infty u^{-(1 + b - \beta)} \mathrm{d}u < \infty$ independently of $p$. 

Let us now prove (\ref{A:up_estim}): when $x, y \rightarrow \infty, \; \mathrm{B}(x, y) \sim \sqrt{2\pi} \left( x^{x - 1/2} y^{y - 1/2}\right) / (x + y)^{x + y - 1/2}$ by Stirling formula. Hence if $x = o(1)$ such that $x \rightarrow 0$, but $p x \rightarrow \infty$ as $p \rightarrow \infty$, then
\begin{align*}
    \mathcal{R}_p(u) \lesssim \int_{1/p}^{x} \frac{1}{\mathrm{B}(pu, pv)} \frac{(v + 1)^{-\theta - 1}}{2^{p(u + v)}} \mathrm{d}v 
    + \int_{x}^\infty \underbrace{\frac{(v + 1)^{-\theta - 1}}{\sqrt{2\pi}} \sqrt{\frac{u v}{u + v}}}_{\leq (v + 1)^{-\theta - 1/2}/ \sqrt{2\pi}} \left(p^{1/(2p)} \frac{\left(\frac{u + v}{2}\right)^{u + v}}{u^{u} v^{v}}\right)^{p}  \mathrm{d}v.
\end{align*}
We see that the second integral is very similar to $\mathcal{I}_x$ defined in Appendix \ref{A:subsec:f_downarrow}. Combining the result from there with the fact that $v \mapsto (v + 1)^{-\theta - 1/2}$ is integrable over $\mathbb{R}_+$, we get that
\begin{align*}
\label{A:std_estimate}
    \mathcal{R}_p(u) \lesssim \int_{1/p}^{x} \frac{(v + 1)^{-\theta - 1}}{2^{p(u + v)} \mathrm{B}(pu, pv)} \mathrm{d}v 
    + u (u + 1)^{-\theta - 1} + 1 \leq \int_{1/p}^{x} \frac{(v + 1)^{-\theta - 1}}{2^{p(u + v)} \mathrm{B}(pu, pv)} \mathrm{d}v 
    + 2.
\end{align*}
It is left to estimate the first integral. Noting that $m (m + p)^{-\theta - 1} \leq (1 + p)^{-\theta}$ for $m \geq 1$,  
\begin{align*}
    \int_{1/p}^{x} \frac{(v + 1)^{-\theta - 1}}{2^{p(u + v)} \mathrm{B}(pu, pv)} \mathrm{d}v &\approx
    \sum_{m = 1}^{px} p^{\theta} \frac{m(m + p)^{-\theta - 1}}{2^{m + n}} \binom{n + m - 1}{m} 
    \leq \frac{p^{\theta}}{(1 + p)^{\theta}} \leq 1.
\end{align*}

\subsection{Positive tail of \texorpdfstring{$\nu_{\mathbf{\hat{q}}}$}{vq}}
\label{A:subsec:nu_+tail}
The goal of this section is to show that in the non-generic critical setup $(\mathbf{q}, g, n) \in \mathcal{D}$, so that $g \gamma_+ = 1/2$ (we supress the dependence on the triplet in what follows),
\begin{align*}
    \nu_{\mathbf{\hat{q}}}(k) \sim 2 \gamma_+^{-2} C_{(\ref{asymptotic})} \cos(\pi a) k^{-a} \quad 
    \text{as} \; k \rightarrow \infty.
\end{align*} 
We recall that $\nu_{\mathbf{\hat{q}}}(k) = \hat{q}_{k + 2} \gamma_+^{k}$ via (\ref{nu}), and by the fixed-point equation (\ref{fpe}) we have the recursive relation $\hat{q}_k = q_k + n\sum_{k' \geq 0} \binom{k + k' - 1}{k'} g^{k + k'} F^{(k')}(\mathbf{q}, g, n)$. Since $\mathbf{q}$ has finite support, i.e. $q_k = 0$ for all large $k$, we only need to analyse the asymptotic behaviour of the sum 
\begin{align*}
    \mathcal{S}_k \coloneqq n\sum_{k' \geq 0} \binom{k + k' - 1}{k'} g^{k + k'} F^{(k')}(\mathbf{q}, g, n).
\end{align*}
By definition of $b$, we can rewrite $n$ as $n = 2 \cos(\pi b) = 2 \cos(\pi a)$. We thus need to show that

\emph{Claim:}  $k^a \gamma_+^k \mathcal{S}_{k + 2} / n \rightarrow \gamma_+^{-2} C_{(\ref{asymptotic})}$ as $k \rightarrow \infty$. 

We proceed analogously to Appendix \ref{A:subsec:f_downarrow} and \ref{A:subsec:upper_bound}, namely we split the sum in three terms $\sum_{k' = 1}^N$, $\sum_{k' > N}$ and $k' = 0$, where $N = o(k)$ converging to $\infty$ as $k \rightarrow \infty$. We introduce the notation 
\begin{align*}
    \mathcal{T}_k \coloneqq \gamma_+^{k + 2} \mathcal{S}_{k + 2} / n = \left(\frac{1}{2}\right)^{k + 2} \sum_{k' \geq 0} \binom{k + k' + 1}{k'} g^{k'} F^{(k')}(\mathbf{q}, g, n)
\end{align*}

$(\romannum{1}): \; k' = 0$: \; $k^a \left( \frac{1}{2} \right)^{k + 2} = o(1) \rightarrow 0$ as $k \rightarrow \infty$, so this term is irrelevant for determination of the constant. 

$(\romannum{2}): \; 1 \leq k' \leq N$: It follows from (\ref{asymptotic}), $F^{(k)} \sim C \gamma_+^k k^{-a}$ as $k \rightarrow \infty$, and since $F^{(k)} < \infty$ for all $k$ that $F^{(k')} \leq C_{\mathbf{\hat{q}}} \left( 1 + \gamma_+^{k'} (k')^{-a} \right)$ for any $k' \geq 0$. Therefore, $\mathcal{L}_k$ defined as
\begin{align*}
    k^a \left(\frac{1}{2}\right)^{k + 2} \sum_{k' = 1}^N \binom{k + k' + 1}{k'} g^{k'} F^{(k')} \leq C_{\mathbf{\hat{q}}} k^a \left(\frac{1}{2}\right)^{k + 2} \sum_{k' = 1}^N \binom{k + k' + 1}{k'} \left(g^{k'} + \frac{(k')^{-a}}{2^{k'}} \right).
\end{align*}
By Stirling formula since $N = o(k)$, as $k \rightarrow \infty$,
\begin{align*}
    \mathcal{L}_k &\leq \hat{C}_{\mathbf{\hat{q}}} k^a \left(\frac{1}{2}\right)^{k + 2} \sum_{k' = 1}^N \frac{1}{\sqrt{2\pi k'}} \underbrace{\left( \frac{\expo (k + k' + 1)}{k'}\right)^{k'}}_{\text{mon. incr. in } k'} \underbrace{\exp\left( \frac{-(k')^2}{2(k + k' + 1)}\right)}_{\text{mon. decr. in } k', \: \leq 1} \left(g^{k'} + \frac{(k')^{-a}}{2^{k'}} \right)\\
    &\leq \Tilde{C}_{\mathbf{\hat{q}}} \underbrace{k^a \left(\frac{1}{2}\right)^{k + 2} \left( \frac{\expo (k + N + 1)}{N}\right)^{N}}_{\leq k^a \left( \frac{1}{2} \left( \frac{2\expo k}{N} \right)^{N/k} \right)^k = k^a ((1 + o(1))/2)^k} \underbrace{\sum_{k' = 1}^N \sqrt{\frac{1}{2\pi k'}} \left(g^{k'} + \frac{(k')^{-a}}{2^{k'}} \right)}_{\leq \sum \left(\frac{1}{2}\right)^{k'} + \sum \left(\frac{g^{k'}}{\sqrt{k'}}\right)} \\
    &\leq \Bar{C}_{\mathbf{\hat{q}}} k^{a + 1/2} \left( \frac{1}{2} (1 + o(1)) (1 \vee g)^{N/k}\right)^k \xrightarrow{k \rightarrow \infty} 0 \quad \text{as} \;\: N = o(k).
\end{align*}
This shows that this term is also irrelevant for determination of the constant in the limit. 

$(\romannum{3}): \; k' >  N:$ On the contrary to the previous two, this case is relevant for determination of the scaling constant. We recall that $N = o(k)$ tends to infinity if $k$ does, $g\gamma_+ = 1/2$. For this part of the sum we can use the asymptotics $F^{(k')} \sim C \gamma_+^{k'} (k')^{-a}$ and obtain as in \ref{A:subsec:f_downarrow}
\begin{align*}
    k^a C_{(\ref{asymptotic})} &\left(\frac{1}{2}\right)^{k + 2} \sum_{k' > N} \binom{k + k' + 1}{k'} \left(\frac{1}{2}\right)^{k'} (k')^{-a} \\
    &\underset{\text{Stirling}}{\overset{\text{Riemann}}{\approx}} \frac{C_{(\ref{asymptotic})}}{2 \sqrt{2\pi}} \int_{\frac{N}{k}}^\infty u^{-a} \sqrt{\frac{1 + u + \frac{1}{k}}{u}} \left( k^{1/(2k)} \frac{\left(\frac{u + 1 + 1/k}{2}\right)^{1 + u + 1/k}}{u^u \left( 1 + \frac{1}{k}\right)^{1 + 1/k}}\right)^k \mathrm{d}u.
\end{align*}
We observe that this integral is similar to $\mathcal{I}_x$ defined in Appendix \ref{A:subsec:f_downarrow} and we treat it similarly. We take $\varepsilon$ as there adjusting it to this case, i.e. $\varepsilon = \sqrt{6 \left(1 + \frac{1}{k}\right) \log(k) / k}$, and split the above integral in three terms $\int_{\frac{N}{k}}^{1 + 1/k - \varepsilon} + \int_{1 + 1/k - \varepsilon}^{1 + 1/k + \varepsilon} + \int_{1 + 1/k + \varepsilon}^{\infty}$. We recall that for $u$ outside of $\left(1 + \frac{1}{k} - \varepsilon, 1 + \frac{1}{k} + \varepsilon \right)$, the function $k^{1/(2k)} \frac{\left(\frac{u + 1 + 1/k}{2}\right)^{1 + u + 1/k}}{u^u \left( 1 + 1/k\right)^{1 + 1/k}} < 1$. We estimate the mentioned three integrals separately. 

$(a): \:$ We start with the third term. Since $u^{-a} \sqrt{\frac{1 + u + 1/k}{u}} \leq 2 u^{-a}$ and $a \in \left( \frac{3}{2}, \frac{5}{2}\right)$ this function is integrable, and thus, by dominated convergence theorem we conclude that this term converges to $0$. It is, thus, irrelevant for the determination of the limiting constant. 

$(b): \:$ The middle term 
\begin{align*}
    I \coloneqq \frac{C_{(\ref{asymptotic})}}{2 \sqrt{2\pi}} \int_{1 + \frac{1}{k} - \varepsilon}^{1 + \frac{1}{k} + \varepsilon} u^{-a} \sqrt{\frac{1 + u + \frac{1}{k}}{u}} \left( k^{1/(2k)} \frac{\left(\frac{u + 1 + 1/k}{2}\right)^{1 + u + 1/k}}{u^u \left( 1 + \frac{1}{k}\right)^{1 + 1/k}}\right)^k \mathrm{d}u
\end{align*}
is fully analogous to the integrals in \ref{A:subsec:f_downarrow} and \ref{A:subsec:upper_bound}. Using the computation from these sections, 
we get, as $k \rightarrow \infty$,
\begin{align*}
    I \approx \frac{C_{(\ref{asymptotic})}}{2 \sqrt{2\pi}} \sqrt{2} g_x(1) \int_{\mathbb{R}} \expo^{-z^2 /2} \mathrm{d}z \approx C_{(\ref{asymptotic})}.
\end{align*}
The \emph{Claim} follows provided that the contribution of the remained integral is $o(1)$ as $k \rightarrow \infty$. 

$(c): \:$ The goal of this part of the proof is to show that as $k \rightarrow \infty$
\begin{align*}
    J \coloneqq \frac{C_{(\ref{asymptotic})}}{2 \sqrt{2\pi}} \int_{\frac{N}{k}}^{1 + \frac{1}{k} - \varepsilon} u^{-a} \sqrt{\frac{1 + u + \frac{1}{k}}{u}} \left( k^{1/(2k)} \frac{\left(\frac{u + 1 + 1/k}{2}\right)^{1 + u + 1/k}}{u^u \left( 1 + \frac{1}{k}\right)^{1 + 1/k}}\right)^k \mathrm{d}u = o(1).
\end{align*}
We use that $u \mapsto \frac{\left(\frac{u + 1 + 1/k}{2}\right)^{1 + u + 1/k}}{u^u \left( 1 + \frac{1}{k}\right)^{1 + 1/k}}$ is strictly increasing on $\left( \frac{N}{k}, 1 + \frac{1}{k} - \varepsilon \right]$ and conclude
\begin{align*}
    J &\leq \frac{3 C_{(\ref{asymptotic})}}{2 \sqrt{2\pi}}  \Bigg( \underbrace{\frac{k^{1/(2k)} \left(1 + 1/k - \frac{\varepsilon}{2}\right)^{2 + 2/k - \varepsilon}}{\left(1 + 1/k - \frac{\varepsilon}{2}\right)^{1 + 1/k - \frac{\varepsilon}{2}} \left( 1 + \frac{1}{k}\right)^{1 + 1/k}}}_{\approx \left(1 - \frac{3}{2} \frac{\log(k)}{k} + \mathcal{O}\left( \left( \frac{\log(k)}{k}\right)^{3/2} \right)\right) \left(1 + \frac{\log(k)}{2k} + \mathcal{O}\left( \frac{1}{k^2} \right)\right)} \Bigg)^k \int_{\frac{N}{k}}^{1 + \frac{1}{k} - \varepsilon} u^{-a - 1/2} \mathrm{d}u \\
    &\leq \Tilde{C} \underbrace{\left(1 - \frac{\log(k)}{k} + \mathcal{O}\left( \left( \frac{\log(k)}{k}\right)^{3/2} \right)\right)^k}_{\leq 2 \expo^{-\log(k)}} \left( 1 + \left(\frac{N}{k}\right)^{-a + 1/2} \right) 
    \leq \Bar{C} \frac{1}{k} \left(\frac{N}{k}\right)^{-a + 1/2} \xrightarrow{k \rightarrow \infty} 0
\end{align*}
for appropriate $N = o(k)$, e.g. $N = k^{1 - \delta}$ for any $\delta \in (0, 1)$ satisfying $\delta \left(a - \frac{1}{2}\right) < 1$.

\bibliographystyle{alpha}
\bibliography{references}

\end{document}